\documentclass[10pt]{amsart}
\usepackage{amsthm,amssymb,amscd,amsmath,amsfonts,amssymb,amscd,stmaryrd,hyperref,mathrsfs} 
\usepackage{fancyhdr}
\usepackage[shortlabels]{enumitem}
\usepackage[all]{xy}
\usepackage{texmac} 					
\usepackage{hyperref} 					
\usepackage{color} 	

\newtheorem{MainThm}{Theorem}

\theoremstyle{definition}
\newtheorem{defn}{Definition}[subsection]
\newtheorem{thm}[defn]{Theorem}
\newtheorem{cor}[defn]{Corollary}
\newtheorem{prop}[defn]{Proposition}
\newtheorem{lem}[defn]{Lemma}
\newtheorem{ex}[defn]{Example}

\newtheorem*{quest*}{Question}
\newtheorem{rmk}[defn]{Remark}

\newtheorem{notn}[defn]{Notation}

\calsymbols{c}{A,B,C,D,E,F,G,H,I,J,K,L,M,N,O,P,Q,R,S,T,U,V,W,X,Y,Z}
\bbsymbols{b}{A,B,C,D,E,F,G,H,I,J,K,L,M,N,O,P,Q,R,S,T,U,V,W,X,Y,Z}
\scrsymbols{s}{A,B,C,I,L,Q,R,M,P,S,T,U,}
\fraksymbols{f}{g,m,h,t,n,b,}
\bsymbols{b}{g,h,x,y,A,B,C,D,E,F,G,H,I,J,K,L,M,N,O,P,Q,R,S,T,U,V,W,X,Y,Z}
\bbsymbols{}{B,G,R,C,Z,N,P,Q,F,G,H,U,W,X,Y,T}

\operators{Der,gr,Spec,Proj,MaxSpec,ad,Loc,Stab,Aut,Hom,End,im,coker,Lie,Spf,rig,Sp,Ad,aug,ac,tors,Sym,mod,Homeo,op,qc,Res,alg,rk,Sch,id,GL,Ab,ind,Supp,Kdim,diag,Coh,Mod,cts,an,SL,Gal,Rep,sm,Ann}

\DeclareMathOperator{\Char}{char}

\DeclareMathOperator{\Dp}{dp}
\DeclareMathAlphabet{\mathpzc}{OT1}{pzc}{m}{it}
\DeclareSymbolFont{largesymbols}{OMX}{yhex}{m}{n}
\DeclareMathAccent{\wideparen}{\mathord}{largesymbols}{"F3}

\newcommand{\h}[1]{\widehat{#1}}

\newcommand{\hK}[1]{\h{#1_K}}

\newcommand{\w}[1]{\wideparen{#1}}

\newcommand{\fr}[1]{\mathfrak{{#1}}}

\newcommand{\ts}[1]{\texorpdfstring{$#1$}{}}

\newcommand{\be}{\begin{enumerate}[{(}a{)}]}
\newcommand{\ee}{\end{enumerate}}
\newcommand{\qmb}[1]{\quad\mbox{#1}\quad}

\newcommand{\hsp}{\hspace{0.1cm}}

\newcommand{\congs}{\stackrel{\cong}{\longrightarrow}}

\newcommand{\tocong}{\stackrel{\cong}{\longrightarrow}}

\newcommand{\WO}[1]{\underset{#1}{\w\otimes}{}}
\newcommand{\utimes}[1]{\underset{#1}{\otimes}{}}
\newcommand{\wtimes}[1]{\underset{#1}{\w\otimes}{}}
\newcommand{\urtimes}[1]{\underset{#1}{\rtimes}{}}
  \let\leq=\leqslant
\let\ge=\geqslant  \let\geq=\geqslant

\begin{document}

\title[Induction equivalence for equivariant $\mathcal{D}$-modules]{Induction equivalence for equivariant $\mathcal{D}$-modules on rigid analytic spaces}
\author{Konstantin Ardakov}
\begin{abstract} We prove an Induction Equivalence and a Kashiwara Equivalence for coadmissible equivariant $\cD$-modules on rigid analytic spaces. This allows us to completely classify such objects with support in a single orbit of a classical point with co-compact stabiliser. As an application, we use the locally analytic Beilinson-Bernstein equivalence to construct new examples of large families of topologically irreducible locally analytic representations of certain compact semisimple $p$-adic Lie groups.
\end{abstract}
\maketitle
\tableofcontents

\section{Introduction}

\subsection{Support of equivariant $\cD$-modules on rigid analytic
spaces}\label{SuppInd} Let $G$ be a $p$-adic Lie group and let $K$ be a non-Archimedean field of mixed characteristic $(0,p)$. In our recent work \cite{EqDCap}, we
introduced the category $\cC_{\bX/G}$ of coadmissible $G$-equivariant
$\cD$-modules on a smooth rigid $K$-analytic space $\bX$, which serves as a $p$-adic analogue of the classical
category of coherent $\cD$-modules on a smooth complex algebraic
variety that are strongly equivariant with respect to the action of a
real or complex Lie group. The main motivation for this construction
is the Beilinson-Bernstein-style \emph{Localisation Theorem} ---
\cite[Theorem C]{EqDCap} --- which, combined with the Schneider-Teitelbaum equivalence \cite[Theorem 6.3]{ST}, gives an anti-equivalence of categories
$V \mapsto \Loc(V'_b)$ between the category of admissible locally
analytic $G$-representations $V$ of $G$ with trivial infinitesimal central
character and the category of coadmissible $G$-equivariant
$\cD$-modules on the rigid analytic flag variety associated with the
semisimple $p$-adic Lie group $G$.

In this paper, we start to address the following basic question.

\begin{quest*} Given an admissible locally analytic representation $V$ as above, what can be the \emph{support} of the sheaf $\Loc(V'_b)$?
\end{quest*}

In the classical situation, if $\cM$ is a strongly equivariant $\cD$-module on a complex algebraic variety, then it is in particular a coherent $\cD$-module and therefore its support is closed in the Zariski topology.  When we interpret an object $\cM \in \cC_{\bX/G}$ as a sheaf $\tilde{\cM}$ on the Huber space $\tilde{\bX}$ associated with $\bX$, for formal reasons, the support $\Supp(\tilde{\cM})$ of $\tilde{\cM}$ is necessarily a $G$-stable subset of $\tilde{\bX}$ which detects whether $\cM$ is non-zero in the sense that $\Supp(\tilde{\cM}) \neq \emptyset$ if and only if $\cM \neq 0$. Taking our cue from the classical situation, as well as from an examination of the currently known list of examples, we expect that when $\cM \neq 0$, 
\begin{enumerate}[{(}i{)}]
\item $\Supp(\tilde{\cM})$ is a closed subset of the topological space $\tilde{\bX}$,
\item $\Supp(\tilde{\cM}) \cap \bX$ is always non-empty.
\end{enumerate}
In the first instance, we focus our attention on those objects $\cM \in \cC_{\bX/G}$ whose support is as small as possible, set-theoretically. Guided by the expectations (i) and (ii) above, this means considering supports of the form $G\cdot x$ where $x \in \bX$. Note that the requirement that this is a closed subset of $\tilde{\bX}$ already puts a severe restriction on the possible values of $x$: for example, it can be shown that when $G = \SL_2(L)$ for some finite extension $L$ of $\Qp$ contained in $K$, acting on the rigid analytic $K$-projective line $\bX = \mathbb{P}^{1,\an}_K$  by M\"obius transformations, then the $G$-orbit $G\cdot x$ of a point $x \in \bX = \mathbb{P}^1(\overline{K}) / \Gal(\overline{K}/K)$ is closed in the Huber space $\tilde{\bX}$ if and only if $x \in \mathbb{P}^1(L)$. 

\subsection{The main result} Suppose, then, that $x \in \bX$ is such that $G \cdot x$ is closed in $\tilde{\bX}$. We are able to completely classify the $\cM \in \cC_{\bX/G}$ that are supported on $G \cdot x$.  In order to state our main result precisely, we need the following two definitions.

\begin{defn} Let $S$ be a subset of $\bX$.
\be\item Let $\cM$ be an abelian sheaf on $\bX$. We say that $\cM$ is \emph{supported on $S$} if $\cM_{|\bV} = 0$ for every admissible open subset $\bV$ of $\bX \backslash S$.
\item Let $\cC^S_{\bX/G}$ denotes the full subcategory of $\cC_{\bX/G}$ consisting of those $\cM \in \cC_{\bX/G}$ that are supported on $S$.
\ee
\end{defn}

We show in Lemma \ref{Supports} below that an abelian sheaf $\cM$ on $\bX$ is supported on $S$ in this sense if and only if the support $\Supp(\tilde{\cM})$ of the associated sheaf $\tilde{\cM}$ on $\tilde{\bX}$ is contained in the closure $\overline{S}$ of $S$ in $\tilde{\bX}$. 

\begin{defn} Let $G$ be a group acting on a set $X$ and let $Y$ be a subset of $X$. We say that the $G$-orbit of $Y$ is \emph{regular in $X$} if distinct $G$-translates of $Y$ are disjoint:
$g Y \cap Y \neq \emptyset  \quad \Rightarrow \quad g Y = Y$ for all $g \in G$.
\end{defn}
Evidently this condition is satisfied whenever $Y = \{x\}$ is a singleton. When it fails, the space $GY$ is somehow ``singular'' at points that happen to lie on more than one distinct $G$-translate of $Y$; hence the choice of terminology.  With these definitions in place, our main result reads as follows.

\begin{MainThm}\label{InductionEquivalence}  Let $G$ be a $p$-adic Lie group acting continuously on a rigid analytic space $\bX$, and let $\bY$ be a Zariski closed subset of $\bX$. Suppose that
\be 
\item $\bX$ is smooth and separated,
\item $\bY$ is irreducible and quasi-compact,
\item the $G$-orbit of $\bY$ is regular in $\bX$, and
\item the stabiliser $G_{\bY}$ of $\bY$ is co-compact in $G$. 
\ee Then the functor of sections supported on $\bY$ gives an equivalence of categories
\[ \cH^0_{\bY} : \cC_{\bX/G}^{G \bY} \congs \cC_{\bX / G_{\bY}}^{\bY} . \]
\end{MainThm}
Unfortunately, Theorem \ref{InductionEquivalence} requires a large number of conditions: we proceed to discuss these in turn. Condition (a) on $\bX$ is not restrictive at all, because the category $\cC_{\bX/G}$ has only been defined in the case where $\bX$ is smooth, and many rigid spaces occurring in nature are separated. Condition (b) is slightly more restrictive, but is satisfied in many examples of interest. Both (a) and (b) are needed to ensure that the triple $(\bX, \bY, G)$ satisfies a certain technical condition that we call the \emph{Local Stabiliser Condition} --- see Definition \ref{LSCdefn} --- which is in turn needed to ensure that the local cohomology functor $\cH^0_{\bY}$ preserves coadmissibility. 

The most restrictive condition in Theorem \ref{InductionEquivalence} is condition (c): away from the somewhat trivial case where $\bY$ is a point, unfortunately we had to work quite hard to come up with examples of triples $(\bX, \bY, G)$ where this regularity condition on $\bY$ is satisfied. Indeed, we show in $\S \ref{ExamplesSect}$ below that when $\bX = \mathbf{P}^{n, \an}$ is a projective space, its Zariski closed subvariety $\bY$ has positive dimension and $G$ is a subgroup of $\GL_{n+1}(L)$ which contains a transvection \footnote{an element $h \in GL_{n+1}(L)$ such that $\rk(h - 1) = 1$}, then this condition is \emph{never} satisfied. On a more positive note, in Theorem \ref{MainD} below we give an explicit example of a triple $(\bX, \bY, G)$ where the $G$-orbit of $\bY$ in $\bX$ is regular and $\bY$ is bigger than a point. There, $\bY$ is some particular ruled surface inside the ($3$-dimensional) rigid analytic flag variety $\bX$ of $\GL_4$, and $G$ is the group of units of some $p$-adic division algebra of degree $4$. In forthcoming work with Tobias Schmidt, we will apply Theorem \ref{InductionEquivalence} in other interesting situations where $\bY$ is bigger than a single point.

Let $B$ be a closed subgroup of $G$. In $\S\ref{IndFunctor}$ we construct an \emph{induction functor} $\ind_{B}^G : \cC_{\bX / B} \longrightarrow \cC_{\bX/G}$ which, in the setting of Theorem \ref{InductionEquivalence}, serves as an inverse to the local cohomology functor $\cH^0_{\bY}$ with $B = G_{\bY}$.  In order to be able to show that this induction functor preserves coadmissibility, we have to assume that the homogeneous space $G/B$ is compact: this explains the presence of condition (d) in the statement of Theorem \ref{InductionEquivalence}. Note that in the case where $\bY$ is a point, $\{x\}$ say, the co-compactness of the stabiliser --- or equivalently, the compactness of the $G$-orbit $G \cdot x$ --- is essentially forced upon us by our expectation (see $\S \ref{SuppInd}$ above) that this orbit is closed in $\tilde{\bX}$ whenever $\bX$ happens to be quasi-compact. Because the rigid analytic flag variety of relevance to the locally analytic Beilinson-Bernstein Localisation Theorem, \cite[Theorem C]{EqDCap} \emph{always} happens to be quasi-compact, this final restriction in Theorem \ref{InductionEquivalence} is not particularly onerous.

\subsection{The equivariant Kashiwara equivalence}
In those situations where Theorem \ref{InductionEquivalence} does apply, one is left with the task of analysing the category $\cC_{\bX / G_{\bY}}^{\bY}$: although the group $G$ has been replaced by its smaller subgroup $G_{\bY}$, the objects in this category are still certain $\cD$-modules on $\bX$. Our second main result explains how to deal with this category in the case where $\bY$ happens to be smooth.

\begin{MainThm}\label{MainB} Let $\iota : \bY \hookrightarrow \bX$  be the inclusion of a smooth, Zariski closed subset $\bY$ into the smooth rigid analytic space $\bX$. Let $G$ be a $p$-adic Lie group acting continuously on $\bX$ and stabilising $\bY$.  Then there are natural functors
\[\iota_+ : \cC_{\bY/G} \to \cC_{\bX/G}^{\bY} \quad \qmb{and} \quad \iota^\natural : \cC_{\bX/G}^{\bY} \to \cC_{\bY/G}\]
which are mutually inverse equivalences of abelian categories.\end{MainThm}
Combining Theorem \ref{InductionEquivalence} with Theorem \ref{MainB}, we obtain 

\begin{cor}\label{AandB} Let $G$ be a $p$-adic Lie group acting continuously on a rigid analytic space $\bX$, and let $\bY$ be a Zariski closed subset of $\bX$. Suppose that
\be 
\item $\bX$ is smooth and separated,
\item $\bY$ is smooth, irreducible and quasi-compact,
\item the $G$-orbit of $\bY$ is regular in $\bX$, and
\item the stabiliser $G_{\bY}$ of $\bY$ is co-compact in $G$. 
\ee Then there is an equivalence of categories
\[ \iota^\natural \circ \cH^0_{\bY} : \cC_{\bX/G}^{G \bY} \congs \cC_{\bY / G_{\bY}}. \]
\end{cor}

Let us now specialise to the case where $\bY$ is a classical point $\{x\}$, so that conditions (b) and (c) in Corollary \ref{AandB} are satisfied automatically. 

\begin{cor}\label{AandBpt} Let $G$ be a $p$-adic Lie group acting continuously on a smooth, separated rigid analytic space $\bX$ and let $x \in \bX$. Suppose that the stabiliser $G_x$ of $x$ is co-compact in $G$. Then there is an equivalence of categories
\[ \Gamma \circ \iota^\natural \circ \cH^0_{\{x\}} : \cC_{\bX/G}^{G \cdot x} \congs \cC_{D^\infty(G_x,K)}. \]
\end{cor}
Here $D^\infty(G_x, K)$ is the algebra of smooth distributions on $G_x$ from \cite{ST3}; the equivalence between $\cC_{\{x\} / G_x}$ and the category $\cC_{D^\infty(G_x,K)}$ of coadmissible $D^\infty(G_x,K)$-modules is given by the global sections functor $\Gamma$ --- see \cite[Theorem B(c)]{EqDCap}.

\subsection{Applications to locally analytic representations} Combining Corollary \ref{AandBpt} together with \cite[Theorem C]{EqDCap} and \cite[Theorem 6.3]{ST}, we obtain the following consequences for the locally analytic representation theory of $p$-adic semisimple groups. Following \cite[\S 6]{ST}, we will use $\Rep^{L-\an}_K(G)$ to denote the abelian category of all admissible locally $L$-analytic $G$-representations over $K$ with continuous $K$-linear maps, and we will use $\Rep^{\sm}_K(G)$ to denote its full subcategory consisting of the admissible smooth $G$-representations. 

\begin{MainThm}\label{FullSubcat} Let $L$ be a finite extension of $\Qp$ contained in $K$, let $\G$ be an affine algebraic $L$-group such that $\G_K := \G \otimes_L K$ is connected and split semisimple, and let $G$ be an open subgroup of $\G(L)$, and let  $\bX := (\G_K/\B)^{\rig}$ be the rigid $K$-analytic flag variety associated with $\G_K$.
\be
\item Suppose that the stabiliser $G_x$ of the classical point $x \in \bX$ is co-compact in $G$. Then there is a fully faithful and exact embedding of abelian categories
\[ j_x : \Rep^{\sm}_K(G_x) \hookrightarrow \Rep^{L-\an}_K(G)\]
given by $j_x(V) := \Gamma(\bX, \ind_{G_x}^G \iota_+ \Loc^{D^\infty(G_x,K)}_{\{x\}} V'_b)'$.
\item Let $y \in \bX$ be another classical point with co-compact stabiliser $G_y$ and suppose that $y \notin G\cdot x$. Then the essential images of $j_x$ and $j_y$ contain no non-zero isomorphic objects.
\ee
\end{MainThm}
Of course the essential image of $j_x$ may be characterised as the full subcategory of $G$-representations $V \in \Rep^{L-\an}_K(G)$ such that the support of $\Loc^{D(G,K)}_{\bX}(V'_b)$ is contained in $G \cdot x$. 

In the case where the open subgroup $G$ of $\mathbb{G}(L)$ happens to be \emph{compact}, the condition on the stabiliser $G_x$ imposed in Theorem \ref{FullSubcat} is vacuous. In this case, for every $x \in \bX$ and every irreducible admissible smooth representation $V$ of $G_x$, Theorem \ref{FullSubcat} produces a \emph{topologically irreducible} admissible locally $L$-analytic $G$-representation $j_x(V)$. Furthermore, another such representation $j_y(W)$ can be isomorphic to $j_x(V)$ (even after a base-change to a finite extension of $K$) only if $y$ lies in the same $G$-orbit as $x$. Note that because the field of definition $L$ is far from being algebraically closed, the rigid analytic variety $\bX$ admits a very large collection of ``irrational" $G$-orbits $G\cdot x$: we can choose any increasing sequence $L \subset L_0 \subset L_1 \subset L_2 \subset \cdots$ of finite Galois extensions of $L$ such that the semisimple affine algebraic group $\mathbb{G}_{L_0}$ is split with flag variety $\mathbb{X}$, and pick a point $x_n \in \mathbb{X}(L_n) \backslash \mathbb{X}(L_{n-1})$ for each $n \geq 1$; then $G\cdot x_n$ is necessarily contained in $\mathbb{X}(L_n) \backslash \mathbb{X}(L_{n-1})$ and hence $G\cdot x_n \neq  G \cdot x_m$ whenever $n < m$. The image of each of these orbits in $\bX = \mathbb{X}(\overline{K}) / \Gal(\overline{K}/K)$ then produces at least one irreducible admissible $G$-representation, namely $j_x(K)$ where $K$ is the trivial $G_x$-representation. In this way, we obtain a large family of pairwise non-isomorphic topologically irreducible admissible locally $L$-analytic $G$-representations whenever $G$ is a compact open subgroup of $\mathbb{G}(L)$.  As we explained in $\S \ref{SuppInd}$ above, in the case where $G$ is the group $\mathbb{G}(L)$ itself, the restriction that $G_x$ is co-compact is far more restrictive, because it essentially forces $x$ to be an \emph{$L$-rational} point of $\bX$, and then $j_x(V)$ is a principal series representation of $G$. Nevertheless, in this case Theorem \ref{FullSubcat} may be used to establish the topological irreducibility of these representations.

In \cite{KS06}, Kisin and Strauch constructed examples of topologically irreducible locally $L$-analytic representations $V_{\chi,\mathcal{H}}$ associated with certain open $G$-orbits $\mathcal{H} = G\cdot x$ contained in $\mathbb{P}^1(E)$, where $E$ is some finite extension of $L$ and $G$ is some locally $L$-analytic subgroup of $\GL_2(E)$ such that $E \cdot \Lie(G)$ contains $\mathfrak{sl}_2(E)$. They also proved that their representations are admissible whenever the orbit $\mathcal{H}$ is compact. We expect that it can be shown that Theorem C gives a generalisation of their construction at least when $G$ is open in $\SL_2(L)$ and $\mathbb{G} = \SL_{2,L}$: then the dual $(V_{\chi,\mathcal{H}})'_b$ of $V_{\chi,\mathcal{H}}$ should be isomorphic to $j_x(K)$ for those values of $\chi$ where $V_{\chi,\mathcal{H}}$ has trivial infinitesimal central character. Theorem C will then give a new proof of the topological irreducibility of the Kisin-Strauch representations which does not use Schneider-Teitelbaum's $p$-adic Fourier transform from \cite{ST6}. Describing the localisations of the $D(G,K)$-modules $(V_{\chi,\mathcal{H}})'_b$ for other values of $\chi$ will involve the theory of twisted $G$-equivariant $\mathcal{D}$-modules, as was developed in \cite{Mathers}. 

\subsection{A new class of topologically irreducible representations} The representations constructed above are all associated with $G$-orbits of the form $G \cdot x$ where $x$ is a classical point of the rigid analytic flag variety. Our last main result, Theorem \ref{MainD} below, gives an explicit example of a large collection of triples $(\bX, \bY, G)$ where the $G$-orbit of $\bY$ is regular in $\bX$ with $\dim \bY > 0$.

Let $D$ be a division algebra of degree $4$ and dimension $16$ over $L$ and let $\mathbb{G}$ be the connected $L$-form of $\GL_4$ such that $\mathbb{G}(L) = D^\times$. We say that $x \in \mathbb{G}(K)$ is \emph{generic} if the $L$-algebra homomorphism $x : \cO(\mathbb{G}) \to K$ is injective. We also recall that the \emph{twisted cubic} is the projective curve $C \subseteq \mathbb{P}^3$ defined by the homogeneous equations $x_0x_2 = x_1^2, x_1x_3 = x_2^2, x_0x_3 = x_1x_2$.

\begin{MainThm}\label{MainD} Suppose that $K$ contains an algebraic closure of $L$. Let $D$ be a division algebra of degree $4$ over $L$ and let $\mathbb{G}$ be the connected $L$-form of $\GL_4$ such that $\mathbb{G}(L) = D^\times$. Let $X$ be the flag variety of $\GL_4$ and let $Y$ be the preimage in $X$ of the twisted cubic $C$ in $\P^3$ under either one of the two projection maps $X \twoheadrightarrow \P^3$. Let $\bY \subset \bX$ be the corresponding rigid analytic varieties over $K$. Then the $D^\times$-orbit of $x \bY$ is regular in $\bX$ for any generic $x \in \mathbb{G}(K)$.
\end{MainThm}

The centre of the group $D^\times$ acts trivially on the rigid analytic flag variety. Since $D^\times$ is compact modulo its centre, we can combine Theorem \ref{MainD}, Corollary \ref{AandB} and \cite[Theorem C]{EqDCap}, to obtain new examples of topologically irreducible admissible locally $L$-analytic $D^\times$-representations over $K$ of the form
\[V(x\bY) := \Gamma(\bX, \ind_{D^\times_{x \bY}}^{D^\times}\iota_+ \cO_{x \bY})'\]
for every generic $x \in \mathbb{G}(K)$. Note that $V(x \bY)$ is not isomorphic to $V(x' \bY)$ if $D^\times x \bY \neq D^\times x' \bY$, because 
\[\Supp \Loc(V(x\bY)'_b) = D^\times x \bY.\]
\subsection{Acknowledgements} I would like to thank Tobias Schmidt, Simon Wadsley and Christian Johansson for their interest in this work. I would also like to thank Tobias Schmidt for carefully reading this paper and correcting several inaccuracies. This paper would not have been possible without the many conversations I had with Ian Grojnowski on the nature of the `fuzzy delta function' and his questions about supports: thank you Ian. 

The author was partially supported by EPSRC grant EP/L005190/1.

\section{The induction equivalence \ts{ \cC_{\bX/G}^{G \bY} \cong \cC_{\bX / G_{\bY}}^{\bY} }}

\subsection{Berkovich and Huber spaces}\label{LocalCoho}
Recall from \cite[\S 5]{SchVdPut} that to every rigid analytic space $\bX$ we can associate two topological spaces $\sP(\bX)$ and $\sM(\bX)$, which we call the \emph{Huber space} the \emph{Berkovich space}, respectively. The elements of $\sP(\bX)$ are the prime filters on the admissible open subsets of $\bX$, and the elements of $\sM(\bX)$ are the maximal filters. The sets of the form $\tilde{\bU} := \{ p \in \sP(\bX) : \bU \in p\}$ as $\bU$ ranges over the admissible open subsets of $\bX$ form a basis for the topology on $\sP(\bX)$. There is a natural commutative diagram
\[ \xymatrix{ & \sP(\bX) \ar[dr]^r & \\ \bX \ar[ur] \ar[rr] && \sM(\bX). }\]
In this diagram, the inclusions from $\bX$ into $\sM(\bX)$ and $\sP(\bX)$ send to $x \in \bX$ to $\fr{m}_x := \{$admissible open $\bU \subset \bX:  x \in \bU\}$  (the principal maximal filter on $x$), and the continuous \emph{retraction map} $r$ sends $p \in \sP(\bX)$ to the unique maximal filter $r(p)$ containing $p$. If $\bX$ is quasi-compact and quasi-separated (qcqs), then both $\sP(\bX)$ and $\sM(\bX)$ are quasi-compact, and $\sM(\bX)$ is Hausdorff. In fact, when $\bX$ is qcqs, the retraction map $r$ realises $\sM(\bX)$ as the maximal Hausdorff quotient space of $\sP(\bX)$. We will identify $\bX$ with its image in $\sP(\bX)$.  By \cite[\S 5]{SchVdPut} there is an equivalence of categories between the abelian sheaves on $\bX$ and the abelian sheaves on $\sP(\bX)$. Given an abelian sheaf $\cM$ on $\bX$, we will denote the corresponding sheaf on $\sP(\bX)$ by $\tilde{\cM}$; it follows from the proof of \cite[Theorem 1]{SchVdPut} that $\tilde{\cM}(\tilde{\bW}) = \cM(\bW)$ for any admissible open subset $\bW$ of $\bX$. 

\begin{defn} Let $S$ be a subset of $\bX$ and let $\cM$ be an abelian sheaf on $\bX$. We say that $\cM$ is \emph{supported on $S$} if $\cM_{|\bU} = 0$ for every admissible open subset $\bU$ of $\bX \backslash S$. 
\end{defn}

The intuition behind this definition is explained by the following

\begin{lem}\label{Supports} Let $S$ be a subset of $\bX$. An abelian sheaf  $\cM$ on $\bX$ is supported on $S$ if and only if $\Supp(\tilde{\cM})$ is contained in the closure $\overline{S}$ of $S$ in $\sP(\bX)$.
\end{lem}
\begin{proof} By the \emph{support} of an abelian sheaf $\cF$ on a topological space $X$ we mean the set $\Supp(\cF) := \{x \in X : \cF_x \neq 0 \}$. Suppose that $\cM$ is supported on $S$ and that $\tilde{\cM}_p \neq 0$ for some $p \in \sP(\bX)$. Then for all open sets $U \ni p$ there is an admissible open $\bV \subset \bX$ such that $p \in \tilde{\bV} \subset U$ and $\tilde{\cM}(\bV) \neq 0$. Then $\bV$ cannot be contained $\bX \backslash S$ and thus $\bV \cap S \neq \emptyset$. Hence $U \cap S \neq \emptyset$ and $p \in \overline{S}$. Conversely, if $\Supp(\tilde{\cM}) \subseteq \overline{S}$, let $\bU$ be an admissible open subset of $\bX \backslash S$. Then $\tilde{\bU} \cap S = \tilde{\bU} \cap \bX \cap S = \bU \cap S = \emptyset$, so $\tilde{\bU} \cap \overline{S} = \emptyset$. Hence $\tilde{\cM}_p = 0$ for all $p \in \tilde{\bU}$ whence $\tilde{\cM}_{|\tilde{\bU}} = 0$. Therefore $\cM_{|\bU} = 0$, which means that $\cM$ is supported on $S$.
\end{proof}

We refer the reader to \cite[\S 1]{HartLocCoh} for standard notation from the theory of local cohomology of abelian sheaves on topological spaces.

\begin{defn}\label{LocalCohoDefn} Let $\cM$ be an abelian sheaf on $\bX$ and let $S$ be a subset of $\bX$. 
\be \item We define the subgroup of \emph{sections of $\cM$ supported on $S$} to be
\[ \Gamma_{S}(\cM) := H^0_S(\cM) := H^0_{\overline{S}}(\tilde{\cM}) = \ker \left( \cM(\bX) \longrightarrow \tilde{\cM}(\sP(\bX) \backslash \overline{S})  \right),\]
\item We also have the presheaf $\cH^0_S(\cM)$, given by
\[ \cH^0_S(\cM)(\bU) := H^0_{\tilde{\bU} \cap \overline{S}}(\tilde{\cM}_{|\tilde{\bU}}) = \ker \left( \cM(\bU) \longrightarrow \tilde{\cM}(\tilde{\bU} \backslash \overline{S}) \right)\]
for every admissible open $\bU$ of $\bX$.
\ee\end{defn}

\begin{lem}\label{ClosureOfS} Suppose that $\bX \backslash S$ is an admissible open subset of $\bX$. Then the closure $\overline{S}$ of $S$ in $\sP(\bX)$ equals $\sP(\bX) \backslash \widetilde{\bX \backslash S}$.
\end{lem}
\begin{proof} If $s \in S$ then $s \notin \bX \backslash S$ so $\bX \backslash S \notin \fr{m}_s$ and $\fr{m}_s \notin \widetilde{\bX \backslash S}$. Since $\widetilde{\bX \backslash S}$ is open, this implies that $\overline{S} \subseteq \sP(\bX) \backslash  \widetilde{\bX \backslash S}$. Let $\widetilde{\bU}$ be a basic open neighbourhood of an element $p \in \sP(\bX) \backslash \widetilde{\bX\backslash S}$. Then $\bX \backslash S \notin p$, so $\bU$ cannot be contained in $\bX \backslash S$ because $p$ is a filter and because $\bU \in p$. So $S \cap \bU \neq \emptyset$ and $S \cap \tilde{\bU} \neq \emptyset$. This implies that $p \in \overline{S}$.
\end{proof}

\begin{cor}\label{H0S} Suppose that $\bX \backslash S$ is an admissible open subset of $\bX$ and $\cM$ is an abelian sheaf on $\bX$. Then $H^0_S(\cM) = \ker(\cM(\bX) \to \cM(\bX \backslash S))$ and $\cH^0_S(\cM)(\bU) = \ker(\cM(\bU) \to \cM(\bU \backslash S))$ for every admissible open subset $\bU$ in $\bX$.
\end{cor}
\begin{proof}By Lemma \ref{ClosureOfS},  $\cM(\sP(\bX) \backslash \overline{S}) = \cM(\widetilde{\bX\backslash S}) = \cM(\bX \backslash S)$.
\end{proof}

We return to the setting of equivariant $\cD$-modules. Let $G$ be a $p$-adic Lie group acting continuously on the rigid analytic space $\bX$.

\begin{lem}\label{TopOnLocCoh} Let $\cM$ be a $G$-equivariant $\cD$-module on $\bX$ and let $S$ be a subset of $\bX$. Then
\be \item $\cH^0_S(\cM)$ is a $G_S$-equivariant $\cD$-module on $\bX$, and
\item $\cH^0_S(\cM)$ is locally Fr\'echet whenever $\cM$ is coadmissible.
\ee\end{lem}
\begin{proof} Part (a) is straightforward. For part (b), let $\bU \in \bX_w(\cT)$ and note that $H^0_{\bU \cap S}(\cM)$ is the simultaneous kernel of the restriction maps $\cM(\bU) \to \cM(\bV)$ as $\bV$ ranges over all quasi-compact subsets of $\bU$ such that $\bV \cap S = \emptyset$. These restriction maps are continuous maps between Fr\'echet spaces by \cite[Lemma 6.4.5]{EqDCap}. It follows that $H^0_{\bU \cap S}(\cM)$ is a closed subspace of $\cM(\bU)$, and therefore itself carries a canonical Fr\'echet topoogy. We leave to the reader the task of verifying that the $G_S$-equivariant structure maps $\cH^0_{S}(\cM)(\bU) \to \cH^0_{S}(\cM)(g\bU)$ constructed in part (a) are continuous for every $g \in G_{S}$. Therefore the $G_{S}$-equivariant $\cD$-module $\cH^0_{S}(\cM)$ is locally Fr\'echet,  by \cite[Definition 3.6.1(a)]{EqDCap}.
\end{proof}

\begin{defn} Let $S$ be a subset of $\bX$. Then $\cC^S_{\bX/G}$ denotes the full subcategory of $\cC_{\bX/G}$ consisting of those $\cM \in \cC_{\bX/G}$ such that $\cM_{|\bV} = 0$ for every admissible open subset $\bV$ of $\bX \backslash S$. We call objects $\cM \in \cC^S_{\bX/G}$ \emph{coadmissible $G$-equivariant $\cD$-modules on $\bX$ supported on $S$}.
\end{defn}

\begin{defn}\label{RegOrbit} Let $G$ be a group acting on a set $X$ and let $Y$ be a subset of $X$. We say that the $G$-orbit of $Y$ is \emph{regular in $X$} if distinct $G$-translates of $Y$ are disjoint:
$g Y \cap Y \neq \emptyset  \quad \Rightarrow \quad g Y = Y$ for all $g \in G$.
\end{defn}

\begin{lem}\label{HYreg} Suppose that the $G$-orbit of $Y$ is regular in $X$ and let $H$ be a \emph{normal} subgroup of $G$. Then the $G$-orbit of $HY$ is also regular in $X$, and $G_{HY} = H G_Y$.
\end{lem}
\begin{proof} Let $g  \in G$ be such that $g HY \cap HY \neq \emptyset$. Then we can find $h_i \in H$ and $y_i \in Y$ such that $gh_1y_1 = h_2y_2$, which implies $g h_1Y = h_2Y$. Because $H$ is normal, we obtain $g HY = gHh_1Y = Hgh_1Y =  H h_2Y = HY$.

For the second statement, we may assume that $Y$ is nonempty. Let $g \in G_{HY}$. Then $\emptyset \neq gY \subseteq \cup_{h\in H} hY$ so $g Y \cap hY \neq \emptyset$ for some $h \in H$. But then $h^{-1} g \in G_Y$ as the $G$-orbit of $Y$ is regular and hence $g = h (h^{-1}g) \in H G_Y$. On the other hand, $HG_Y \cdot HY = HG_Y Y = HY$ shows that $HG_Y \leq G_{HY}$.
\end{proof}

To motivate the constructions that will follow, let us recall the following basic Lemma about cohomology with supports.

\begin{lem}\label{LocalCohoSeps} Let $X$ be a topological space, and let $\{C_1,\ldots, C_n\}$ be closed subsets of $X$ which admit pairwise disjoint open neighbourhoods $\{U_1,\ldots, U_n\}$. Then the canonical map $\bigoplus_i \Gamma_{C_i}(X,\cF) \longrightarrow \Gamma_{\cup C_i} (X, \cF)$ is an isomorphism for any abelian sheaf $\cF$ on $X$.
\end{lem}
\begin{proof} Consider the following commutative diagram with exact rows
\[\xymatrix{ 0 \ar[r] & \Gamma_{\cup C_i}(\cup U_i, \cF) \ar[r]\ar@{.>}[d] &  \Gamma(\cup U_i, \cF) \ar[r]\ar[d] & \Gamma( \cup (U_i \backslash C_i), \cF) \ar[d] \\ 0 \ar[r] & \oplus \Gamma_{C_i}(U_i, \cF) \ar[r] & \oplus \Gamma(U_i, \cF) \ar[r] & \oplus \Gamma(U_i \backslash C_i, \cF). }\]
The solid vertical arrows in this diagram are isomorphisms because the $U_i$ are pairwise disjoint and because $\cF$ is a sheaf. Therefore the induced dotted vertical arrow on the left is an isomorphism by the Five Lemma. This map fits into the following commutative square:
\[\xymatrix{ \Gamma_{\cup C_i}(X, \cF) \ar[rr]  & & \Gamma_{\cup C_i}(\cup U_i, \cF) \ar[d]\\  \oplus \Gamma_{C_i}(X,\cF) \ar[rr] \ar[u] &  & \oplus \Gamma_{C_i}(U_i, \cF)}\] 
whose horizontal maps are isomorphisms by the excision formula, \cite[Proposition 1.3]{HartLocCoh}. It follows that the vertical restriction map on the left that we are interested in is also an isomorphism.
\end{proof}

We now return to rigid analytic geometry. Recall from \cite[p.100]{SchVdPut} that associated to every morphism $\beta : \bX \to \bY$ of rigid analytic spaces there is a continuous map $\beta_\ast : \sP(\bX) \to \sP(\bY)$ given by $\beta_\ast(p) = \{$admissible open $\bV \subset \bY : \beta^{-1}(\bV) \in p\}$, which satisfies $\beta_\ast  (\sM(\bX)) \subset \sM(\bY)$. 

\begin{lem}\label{YXtop} Let $\bX$ be a rigid analytic space and let $\beta : \bY \hookrightarrow \bX$ be the inclusion of a Zariski closed subspace. Define
\[\begin{array}{lll}  \underline{\bY} &:=& \{a \in \sM(\bX) : \bX \backslash \bY \notin a \} \subset \sM(\bX) \qmb{and} \\
\tilde{\bY} &:=& \{p \in \sP(\bX) : \bX \backslash \bY \notin p\}\subset \sP(\bX).\end{array}\] 
Then
\be \item $r^{-1}(\underline{\bY}) = \tilde{\bY}$,
\item $\underline{\bY}$ is a closed subset of $\sM(\bX)$,
\item $\tilde{\bY}$ is the closure $\overline{\bY}$ of $\bY$ in $\sP(\bX)$,
\item $\beta_\ast(\sM(\bY)) = \underline{\bY}$,
\item $\underline{\bY \cap \bZ} = \underline{\bY} \cap \underline{\bZ}$ if $\bZ$ is another Zariski closed subspace of $\bX$,
\item $\underline{g \bY} = g \underline{\bY}$ for any $g \in \Aut(\bX)$.
\ee\end{lem}
\begin{proof} (a) It follows from \cite[Lemma 5.3]{SchVdPut} and \cite[Proposition 3.3(iii)]{SchPts93} that the Zariski open subset $\bX \backslash \bY$ of $\bX$ is wide open. Hence $r^{-1}(\underline{\bX \backslash \bY}) = \widetilde{\bX \backslash \bY}$ by \cite[Lemma 5.7]{SchVdPut}. Taking complements in $\sP(\bX)$ gives $r^{-1}(\underline{\bY}) = \tilde{\bY}$.

(b) Because $r$ is a quotient map, it is enough to show that $r^{-1}(\underline{\bY})$ is closed. However, its complement in $\sP(\bX)$ is the open set $\widetilde{\bX \backslash \bY}$, by (a). 

(c) By definition, $\tilde{\bY} = \sP(\bX) \backslash \widetilde{\bX \backslash \bY}$. Now $\bX \backslash \bY$ is an admissible open subset of $\bX$ by \cite[Proposition 3.3(iii)]{SchPts93} so $\sP(\bX) \backslash \widetilde{\bX \backslash \bY} = \overline{\bY}$ by Lemma \ref{ClosureOfS}.

The remaining parts of the Lemma are straightforward.
\end{proof}
We will henceforth identify $\sM(\bY)$ with its image $\underline{\bY}$ in $\sM(\bX)$. Lemma \ref{YXtop} immediately implies the following

\begin{cor}\label{BerkReg} Suppose that a group $G$ acts on a qcqs rigid analytic space $\bX$ by automorphisms, and $\bY$ is a Zariski closed subspace of $\bX$ whose $G$-orbit is regular in $\bX$. Then the $G$-orbit of $\underline{\bY}$ is regular in $\sM(\bX)$.
\end{cor}

\begin{lem}\label{CtsAct} Let $G$ be a topological group acting continuously on a rigid analytic space $\bX$. Then the induced action map $a : G \times \sP(\bX) \to \sP(\bX)$ is  continuous.
\end{lem}
\begin{proof} Let $\bU$ be an affinoid subdomain of $\bX$; then its stabiliser $G_{\bU}$ in $G$ is open by \cite[Definition 3.1.8]{EqDCap}. Now if $T$ is a set of right coset representatives for $G_{\bU}$ in $G$, then $a^{-1}(\tilde{\bU}) = \coprod\limits_{t \in T} G_{\bU} t \times t^{-1} \tilde{\bU}$ is open in $G \times \sP(\bX)$.
\end{proof}

\begin{cor}\label{ClosedGY} Let $G$ be a topological group acting continuously on the rigid analytic space $\bX$, and let $\bY$ be a Zariski closed subset of $\bX$. Suppose that $\sM(\bX)$ is Hausdorff and that $G_{\bY}$ is co-compact. Then $G \underline{\bY}$ is closed in $\sM(\bX)$ and $G \tilde{\bY}$ is closed in $\sP(\bX)$.
\end{cor}
\begin{proof} Because $r$ is surjective, it follows from Lemma \ref{YXtop}(a) that $\underline{\bY} = r(\tilde{\bY})$. Now $G \underline{\bY}$ is the image of the quasi-compact space $G/G_{\bY} \times \sP(\bY)$ under the map $r \circ a$, which is continuous by Lemma \ref{CtsAct}. It is therefore quasi-compact. Since $\sM(\bX)$ is Hausdorff, we see that $G \underline{\bY}$ is closed in $\sM(\bX)$. Hence $G\tilde{\bY} = r^{-1}(G \underline{\bY})$ is closed in $\sP(\bX)$.
\end{proof}

\begin{rmk}\label{MXT2} The hypothesis that $\sM(\bX)$ is Hausdorff is automatic in many cases of interest, e.g. when $\bX$ is quasi-compact, or more generally when $\bX$ satisfies the condition from \cite[Proposition 5.4]{SchVdPut}.
\end{rmk}

\begin{cor}\label{ClosureOfGY} With the notation of Corollary \ref{ClosedGY}, $\overline{G \bY} = G \tilde{\bY}$. \end{cor}
\begin{proof} Note that $\overline{\bY} = \tilde{\bY}$ by Lemma \ref{YXtop}(c).  Since $G \tilde{\bY}$ is closed by Corollary \ref{ClosedGY} and contains $G \bY$, it also contains $\overline{G \bY}$. On the other hand, $\overline{G \bY}$ is a $G$-stable closed subset of $\sP(\bX)$ containing $\bY$, so it contains $\tilde{\bY} = \overline{\bY}$ as well as $G \tilde{\bY}$.
\end{proof}

\begin{thm}\label{Separation} Let $G$ be a topological group acting continuously on the qcqs rigid analytic space $\bX$. Let $\bY_1,\ldots, \bY_m$ be Zariski closed subsets of $\bX$ such that the sets $G\bY_1,\ldots,G\bY_m$ are pairwise disjoint and such that $G_{\bY_i}$ is co-compact for each $i$.  Then there exist pairwise disjoint  open neighbourhoods of their closures $\overline{G \bY_i}$ in $\sP(\bX)$.
\end{thm}
\begin{proof} Each $G\underline{\bY}_i$ is closed in $\sM(\bX)$ by Corollary \ref{ClosedGY} and Remark \ref{MXT2}. If $p \in G \underline{\bY}_i \cap G \underline{\bY}_j$ then there is some $g\in G$ such that $p \in \underline{\bY}_i \cap g \underline{\bY}_j$. But this intersection equals $\underline{ \bY_i \cap g \bY_j }$ by Lemma \ref{YXtop}(d,e) which is empty unless $i = j$ by our assumption. So, the closed subsets $G \underline{\bY}_i$ of $\sM(\bX)$ are pairwise disjoint. 

Because $\sM(\bX)$ is quasi-compact and Hausdorff, it is a normal topological space by \cite[Chapter I, \S 9.3, Proposition 2]{BourGenTop}. Hence we can find pairwise disjoint open neighbourhoods $V_i$ of the $G \underline{\bY}_i$'s. Then $U_i := r^{-1}(V_i)$ is an open neighbourhood of $r^{-1}(G \underline{\bY}_i) = G r^{-1}(\underline{\bY}_i) = G \tilde{\bY}_i = \overline{G \bY_i}$ by Lemma \ref{YXtop} and Corollary \ref{ClosureOfGY}, and the $U_1,\ldots,U_m$ are pairwise disjoint.
\end{proof}

\subsection{Construction of the induction functor}
\label{IndFunctor} Let $G$ be a $p$-adic Lie group acting continuously on a smooth rigid analytic space $\bX$, in the sense of \cite[Definition 3.1.8]{EqDCap}, and let $B$ be a closed subgroup of $G$ such that \emph{$G/B$ is compact}. Our goal here is to construct an \emph{induction functor}
\[ \ind_B^G : \cC_{\bX/B} \longrightarrow \cC_{\bX/G}\]
from $B$-equivariant coadmissible $\cD$-modules on $\bX$ to $G$-equivariant ones. Our first elementary topological Lemma explains how we use the compactness of $G/B$.

\begin{lem}\label{CompactGmodB} Let $H$ be an open subgroup of $G$. Then the set of double cosets $|H \backslash G / B|$ is finite.
\end{lem}
\begin{proof} The quotient map $\pi : G \to G/B$ is open because $\pi^{-1}(\pi(U)) = U \cdot B = \bigcup_{b \in B} U b$ is a union of $B$-translates of $U$ whenever $U$ is an open subset of $G$. Since $H$ is open in $G$, the open covering $G = \bigcup_{g \in G} Hg$ of $G$ maps to an open covering $\bigcup_{g \in G} H \pi(g)$ of $G/B$ which by the compactness of $G/B$ has a finite subcovering. So the $H$-action on $G/B$ has finitely many orbits, and the preimages of these orbits in $G$ are precisely the $H,B$-double cosets in $G$.
\end{proof}

Recall from \cite[Definition 3.4.6(a)]{EqDCap} that $\bX_w(\cT)$ denotes the set of affinoid subdomains $\bU$ of $\bX$ such that $\cT(\bU)$ admits a free $\cA$-Lie lattice for some affine formal model $\cA$ in $\cO(\bU)$. Recall also from \cite[Definition 3.4.4]{EqDCap} that the subgroup $H$ of $G$ is \emph{$\bU$-small} if the pair $(\bU,H)$ is small, which in turns means that $\bU$ is an affinoid subdomain of $\bX$, $H$ is a compact open subgroup of the stabiliser $G_{\bU}$ of $\bU$ in $G$, and $\cT(\bU)$ has an $H$-stable free $\cA$-Lie lattice $\cL$ for some $H$-stable affine formal model $\cA$ in $\cO(\bU)$.  Until the end of $\S \ref{IndFunctor}$ we will fix the following data:

\begin{itemize}
\item $B$ is a closed subgroup of $G$ such that $G/B$ is compact,
\item $\cN \in \cC_{\bX/B}$, 
\item $\bU \in \bX_w(\cT)$,
\item $H$ is a $\bU$-small compact open subgroup of $G$.
\end{itemize}

Whenever $J$ is a subgroup of $G$ and $s \in G$, we will write ${}^s J := s J s^{-1} \qmb{and} J^s := s^{-1} J s$. In \cite[Lemma 3.4.3]{EqDCap} we constructed a continuous isomorphism of $K$-Fr\'echet algebras $\w{s}^{-1}_{U,H} : \w\cD(\bU, H) \longrightarrow \w\cD(s^{-1} \bU, H^s)$ induced by the action of $s^{-1}$ on $\bX$ and the conjugation automorphism $g \mapsto s^{-1}g s$ of $G$. In view of \cite[Theorem 3.3.12]{EqDCap}, this map restricts to an isomorphism $\w\cD(\bU, {}^sB\cap H) \longrightarrow \w\cD(s^{-1} \bU, B \cap H^s)$ which we will simply denote by $\w{s}^{-1}$.  Now $B \cap H^s$ stabilises $s^{-1}\bU$, and in fact it is an $s^{-1} \bU$-small subgroup of $B$, so we see that $\cN(s^{-1} \bU)$ is a module over $\w\cD(s^{-1}\bU, B \cap H^s)$ whenever $\cN \in \cC_{\bX/B}$ by \cite[Theorem 4.4.3]{EqDCap}. 

\begin{defn} Let $\cN \in \cC_{\bX/B}$ and let $H$ be a compact open subgroup of $G$. Suppose that $(\bU,H)$ is small and let $s \in G$. Let $[s]$ be a formal symbol and set
\[ [s] \cN(s^{-1}\bU) := \{ [s] m : m \in \cN(s^{-1}\bU)\}.\]
This becomes a module over $\w\cD(\bU, {}^sB \cap H)$ via the rule
\[ a \cdot [s] m := [s] \w{s}^{-1}(a)m \qmb{for all} a \in \w\cD(\bU, {}^sB \cap H) \qmb{and} m \in \cN(s^{-1}\bU).\] 
We call this $\w\cD(\bU, {}^sB \cap H)$-module the \emph{$s$-twist of $\cN(s^{-1} \bU)$}.
\end{defn}

\begin{lem}\label{TwistCoadm}  $[s] \cN(s^{-1} \bU)$ is a coadmissible $\w\cD(\bU, sBs^{-1} \cap H)$-module.
\end{lem}
\begin{proof} Because $\cN \in \cC_{\bX/B}$, the $\w\cD(s^{-1} \bU, B \cap H^s)$-module $\cN(s^{-1} \bU)$ is in fact coadmissible, by \cite[Theorem 4.4.3]{EqDCap}. The result follows because $\w{s}^{-1}$ is an isomorphism of $\w\cD(\bU, {}^sB \cap H)$ onto $\w\cD(s^{-1} \bU, B \cap H^s)$.
\end{proof}

\begin{prop}\label{sTwist} Every $s \in G$ defines a \emph{twisting functor} $\cC_{\bX/B} \to \cC_{\bX / {}^sB}$ given by $\cN \mapsto [s] s_\ast \cN$. 
\end{prop}
\begin{proof} Fix $\bU \in \bX_w(\cT)$. Then $([s] s_\ast \cN)(\bU) = [s]\cN(s^{-1}\bU)$ is a coadmissible $\w\cD(\bU, sBs^{-1} \cap H)$-module by Lemma \ref{TwistCoadm} for any $\bU$-small open subgroup $H$ of $G$, and therefore in particular is naturally a $\cD(\bU)$-module, as well as a Fr\'echet space. These structures do not depend on the choice of $H$, and $[s] s_\ast \cN$ naturally becomes a sheaf of $\cD$-modules on $\bX_w(\cT)$. If $c \in {}^s B$, then $c^s \in B$, so we have at our disposal the morphism of sheaves $(c^s)^{\cN}:   \cN  \to (c^s)^\ast \cN$ because $\cN$ is a $B$-equivariant sheaf. We define $c^{[s] s_\ast \cN} : [s] s_\ast \cN \to c^\ast ([s] s_\ast \cN)$ by $c^{[s] s_\ast \cN}:= [s] s_\ast (c^s)^{\cN}$. It is straightforward to verify that in this way $[s] s_\ast \cN$ becomes a ${}^sB$-equivariant $\cD$-module on $\bX$. The maps $c^{[s] s_\ast \cN}(\bU) : [s] \cN(s^{-1}\bU) \to [s] \cN( s^{-1} c \bU)$ are continuous because the maps $(c^s)^{\cN}(s^{-1}\bU) : \cN(s^{-1}\bU) \to \cN(s^{-1} c \bU)$ are continuous for each $c \in {}^s B$. Finally, we omit the straightforward verification of the fact that the restriction of $[s] s_\ast \cN$ to $\bU$ is naturally isomorphic to $\Loc^{\w\cD(\bU,C)}_{\bU}([s] \cN(s^{-1} \bU))$ as a $C$-equivariant locally Fr\'echet $\cD$-module on $\bU$. Hence, $[s]s_\ast \cN$ is coadmissible.  The functorial nature of this construction is clear. \end{proof}

\begin{defn}\label{sTwistDefn} We call the object $[s] s_\ast \cN \in \cC_{\bX/{}^sB}$ the \emph{$s$-twist} of $\cN \in \cC_{\bX/B}$.  \end{defn}

\begin{cor}\label{sTwistSupp} Let $s \in G$. The $s$-twist functor sends objects supported on $\bY$ to objects supported on $s\bY$: $[s] s_\ast : \cC_{\bX/B}^{\bY} \longrightarrow \cC_{\bX / {}^sB}^{s \bY}.$
\end{cor}

Suppose $(\bU,H)$ is small. It follows from Lemma \ref{TwistCoadm} that we now have at our disposal the coadmissible $\w\cD(\bU,H)$-module
\[ \boxed{ M(\bU,H,s) := \w\cD(\bU,H) \WO{\w\cD(\bU, {}^sB \hsp \cap \hsp H)} [s] \cN(s^{-1}\bU).}\]
We will show that in a precise sense, $M(\bU,H,s)$ only depends on the double coset $HsB$ containing $s$. Let $Z \in H \backslash G / B$ be an $H,B$-double coset; we regard it as a category as follows. The objects of this category are the elements of $Z$; for two elements $s,t \in Z$, the set of morphisms $Z(s,t)$ in this category is defined by 
\[Z(s,t) := \{ (h,b) \in H \times B : hsb^{-1} = t \}\]
and the composition of arrows maps $Z(t,u) \times Z(s,t) \to Z(s,u)$ are obtained by restricting the group operation in $H \times B$.  Note that each double coset $Z \in H \backslash G / B$ is becomes a groupoid in this way.

\begin{prop}\label{MUHfunctor} Let $Z \in H \backslash G / B$. The rule $s \mapsto M(\bU,H,s)$ defines a functor $M(\bU,H,-):Z \to \cC_{\w\cD(\bU,H)}$.
\end{prop}
\begin{proof} Let $s,t \in Z$ and $(h,b) \in Z(s,t)$, so that $s = h^{-1} t b$. We have to define a $\w\cD(\bU,H)$-linear map $\eta^H_{(h,b)} : M(\bU,H,s) \longrightarrow M(\bU,H,t)$. We define this map 
\begin{equation}\label{Etahb} \eta^H_{(h,b)} : \w\cD(\bU,H) \WO{\w\cD(\bU, {}^sB \hsp \cap \hsp H)} [s] \cN(s^{-1}\bU) \longrightarrow \w\cD(\bU,H) \WO{\w\cD(\bU, {}^tB \hsp \cap \hsp H)} [t] \cN(t^{-1}\bU)\end{equation}
as follows: given $a \in \w\cD(\bU,H)$ and $m \in \cN(s^{-1}\bU)$, we set
\[ \boxed{\eta^H_{(h,b)}( a \hsp \w\otimes \hsp [s] m ) := a \gamma(h)^{-1} \hsp \w\otimes \hsp [t] \hsp b^{\cN}(m)}\]
where $\gamma : H \to \w\cD(\bU,H)^\times$ is the canonical group homomorphism; see \cite[Remark 3.3.2(c)]{EqDCap}. To check that this is well-defined, because
\[a x \hsp \w\otimes \hsp [s] m = a \hsp \w \otimes \hsp x \cdot [s] m = a \hsp \w\otimes \hsp [s] \hsp \w{s^{-1}}(x) \cdot m\]
holds whenever $x \in \w\cD(\bU, {}^s B \cap H)$, we must show that
\begin{equation}\label{EtaWellDef}a x \hsp \gamma(h)^{-1} \hsp \w\otimes \hsp [t] \hsp b^{\cN}(m) = a \gamma(h)^{-1} \hsp \w\otimes \hsp [t] \hsp b^{\cN}(\w{s^{-1}}(x) \cdot m)\end{equation}
holds whenever $x \in \w\cD(\bU, {}^s B \cap H)$. Now $x \gamma(h)^{-1} = \gamma(h)^{-1} \w{h}(x)$ implies that
\[ \begin{array}{lll} a x \gamma(h)^{-1} \hsp\w\otimes \hsp [t] b^\cN(m) &=& a \gamma(h)^{-1} \hsp \w\otimes \hsp \w{h}(x) \cdot [t] b^{\cN}(m) = \\

&=& a \gamma(h)^{-1} \hsp \w\otimes\hsp [t] \hsp \w{t^{-1}}( \w{h}(x)) b^{\cN}(m). \end{array}\]
We know that $t^{-1} h = bs^{-1}$ since $(h,b) \in Z(s,t)$, so
\[ \w{t^{-1}}( \w{h}(x)) b^{\cN}(m) = \w{b}( \w{s^{-1}}(x)) \cdot b^{\cN}(m)  = \gamma(b) \w{s^{-1}}(x) \gamma(b)^{-1} \cdot b^{\cN}(m) = b^{\cN}(\w{s^{-1}}(x)\cdot m)\]
and equation (\ref{EtaWellDef}) follows. Thus $\eta^H_{(h,b)}$ is a well-defined $\w\cD(\bU,H)$-linear map, and the verification that $\eta^H$ respects composition in the category $Z$ is straightforward.
\end{proof}

\begin{defn}\label{DefnOfMUH} \hsp \be \item For each $Z \in H \backslash G / B$, we define $M(\bU,H,Z) := \lim\limits_{s\in Z} M(\bU,H,s).$
\item We define $M(\bU,H) := \bigoplus\limits_{Z \in H \backslash G / B} M(\bU,H,Z).$
\ee
\end{defn} 

\begin{cor}\label{MUHcoadm} Let $Z \in H \backslash G / B$.
\be \item The limit $M(\bU,H,Z)$ exists, and for any $s \in Z$ the canonical map \newline $M(\bU,H,Z) \to M(\bU,H,s)$ is an isomorphism in $\cC_{\w\cD(\bU,H)}$.
\item $M(\bU,H)$ is a coadmissible $\w\cD(\bU,H)$-module.
\ee \end{cor}
\begin{proof} Part (a) follows Proposition \ref{MUHfunctor} and that fact that $Z$ is a groupoid, and part (b) follows part (a) together with Lemma \ref{CompactGmodB}. \end{proof}

Using the above language, we can now establish an analogue of Mackey's Theorem concerning the restrictions to $\w\cD(\bX,H)$ of modules induced up to $\w\cD(\bX,G)$. 

\begin{prop}\label{Mackey} Suppose that $(\bX,G)$ is small  and $\bU = \bX$. Then there is a $\w\cD(\bX,H)$-module isomorphism
\[\bigoplus_{Z \in H \backslash G / B} M(\bX, H, Z) \congs \w\cD(\bX, G) \WO{\w\cD(\bX,B)} \cN(\bX) \]
which is functorial in $\cN \in \cC_{\bX/B}$.
\end{prop}
\begin{proof} We begin by observing that it follows from \cite[Definition 3.4.9(a) and Proposition 3.4.10]{EqDCap} that the canonical map $\w\cD(\bX,H) \WO{K[H]} K[G] \longrightarrow \w\cD(\bX,G)$ is an isomorphism. Consider the $\w\cD(\bX,H) - \w\cD(\bX,B)$-bimodule decomposition
\begin{equation}\label{Mackey1}  \w\cD(\bX,G) = \bigoplus_{Z \in H\backslash G/B} \w\cD(\bX,H) \hsp \gamma(Z) \hsp \w\cD(\bX,B). \end{equation}
Fix $Z \in H \backslash G / B$ and $s \in Z$. If $u \in \w\cD(\bX,B \cap H^s)$ then $\w{s}(u) \in \w\cD(\bX, {}^s B \cap H) \subset \w\cD(\bX,H)$, so $a \hsp \gamma(s) \hsp u \gamma(b) = a \w{s}(u) \hsp \gamma(s) \hsp \gamma(b) \in \w\cD(\bX, H) \gamma(s) \hsp \gamma(b)$ for any $a \in \w\cD(\bX,H)$ and $b \in B$.  Since $G$ is compact and $B$ is closed in $G$, the open subgroup $B \cap H^s$ of $B$ has finite index in $B$. We can now see that if $T$ is a (finite) set of right coset representatives for $B \cap H^s$ in $B$, then there is an isomorphism of left $\w\cD(\bX,H)$-modules
\[\w\cD(\bX,H) \gamma(s) \w\cD(\bX,B) = \bigoplus_{b \in T}  \w\cD(\bX,H) \hsp \gamma(s) \gamma(b).\]
Now let $a \in \w\cD(\bX,H), b \in T$ and $m \in \cN(\bX)$. We omit the verification of the fact that  $a \hsp\w\otimes\hsp [s] \hsp m \mapsto a \hsp \gamma(s) \hsp\w\otimes\hsp m$ defines a $\w\cD(\bX,H)$-linear isomorphism
\begin{equation}\label{Mackey2} \varphi_s : \w\cD(\bX,H) \WO{\w\cD(\bX, {}^sB \cap H)} [s] \hsp \cN(\bX) \congs \w\cD(\bX,H) \gamma(s) \w\cD(\bX,B) \WO{\w\cD(\bX,B)} \cN(\bX) \end{equation}
whose inverse is given by $a \hsp \gamma(s) \hsp \gamma(b) \hsp\w\otimes\hsp m \mapsto a \hsp\w\otimes\hsp [s] \hsp \gamma(b)\cdot m$. Recalling the maps $\eta^H_{(h,b)}$ from (\ref{Etahb}), we have $\varphi_t \circ \eta^H_{(h,b)} = \varphi_s$ whenever $s,t \in Z$ and $(h,b) \in Z(s,t)$. Since $\w\cD(\bX,H) \gamma(s) \w\cD(\bX,B)$ evidently only depends on the double coset $Z$, passing to the limit over all $s \in Z$ we obtain a left $\w\cD(\bX,H)$-module isomorphism
\begin{equation}\label{Mackey3}  \lim\limits_{s\in Z} \w\cD(\bX,H) \WO{\w\cD(\bX, {}^sB \cap H)} [s] \hsp \cN(\bX)\congs \w\cD(\bX,H) \hsp \gamma(Z) \hsp \w\cD(\bX,B) \WO{\w\cD(\bX,B)} \cN(\bX) .\end{equation}
The result follows from $(\ref{Mackey1})$ and $(\ref{Mackey3})$ because $\cN(s^{-1} \bX) = \cN(\bX)$ for all $s \in G$.
\end{proof}

We now use of Proposition \ref{Mackey} to establish the following important fact.

\begin{prop}\label{NUHtoJ} Let $H \leq J$ be $\bU$-small subgroups of $G$. Then there is a canonical $\w\cD(\bU,H)$-linear \emph{isomorphism} $\alpha^J_H : M(\bU, H) \congs M(\bU,J)$.
\end{prop}
\begin{proof}  Fix $Z \in J \backslash G / B$ and $s \in Z$, and write $C := {}^sB \cap J$. Let $\cN_s$ denote the restriction to $\bU$ of the $s$-twist $[s] s_\ast \cN$. Then $\cN_s \in \cC_{\bU/C}$ by Proposition \ref{sTwist}. Applying Proposition \ref{Mackey} to $(\bU,J)$ and $\cN_s \in \cC_{\bU/C}$ gives a $\w\cD(\bU,H)$-linear isomorphism
\[ \alpha_s : \bigoplus_{Y \in H \backslash J / C} \lim\limits_{y \in Y} \w\cD(\bU, H) \WO{\w\cD(\bU, {}^yC \cap H)} [y] \hsp \cN_s( y^{-1} \bU ) \congs  \w\cD(\bU, J) \WO{\w\cD(\bU,C)} \cN_s(\bU). \]
However ${}^yC \cap H = {}^y({}^sB \cap J) \cap H = {}^{ys}B \cap {}^yJ \cap H = {}^{ys}B \cap H$ since $y \in J \supseteq H$, and the $\w\cD(\bU, {}^yC \cap H)$-module $[y]\cN_s(y^{-1}\bU)$ is simply $[ys]\cN(s^{-1}\bU)$, so we may rewrite $\alpha_s$ as follows:
\begin{equation}\label{TechHJ}
\alpha_s : \bigoplus_{Y \in H \backslash J / {}^sB \cap J} \lim\limits_{y \in Y} M(\bU,H, ys) \congs M(\bU, J,s)  .
\end{equation}
Now suppose that $t$ is some other member of the double coset $Z$. Then $t = j s b^{-1}$ for some $j \in J, b \in B$. Because ${}^tB \cap J = {}^{js}B \cap J$, the rule $Y \mapsto Yj^{-1}$ defines a bijection $H \backslash J / {}^sB \cap J \to H \backslash J / {}^t B \cap J$. As $y$ runs over elements of $Y$, $w := yj^{-1}$ runs over elements of $W := Yj^{-1}$; since $w t =  ys b^{-1}$, we have the isomorphism $\eta^H_{(1,b)} : M(\bU, H, ys) \to M(\bU, H, wt)$ from the proof of Proposition \ref{MUHfunctor}. The direct sum of these isomorphisms gives the vertical arrow on the left in the following commutative diagram:
\[ \xymatrix{ \bigoplus\limits_{Y \in H \backslash J / {}^s B \cap J}\hsp   \lim\limits_{y \in Y} \hsp M(\bU,H, ys) \ar[d]  \ar[r]^(0.7){\alpha_s} & M(\bU,J,s) \ar[d]^{\eta^J_{(j,b)}} \\
\bigoplus\limits_{W \in H \backslash J / {}^t B \cap J} \hsp  \lim\limits_{w \in W} M(\bU,H, wt)  \ar[r]_(0.7){\alpha_t} &    M(\bU,J,t)  . }\]
The map $\theta : H \backslash G / B \twoheadrightarrow J \backslash G / B$ given by $W \mapsto J W$ is surjective because $H \leq J$, and the function $H \backslash J / {}^sB \cap J \longrightarrow \theta^{-1}(Z) = H \backslash Z / B$ given by $Y \mapsto YsB$ is a bijection. Because $Ys$ is a subgroupoid of $YsB$, there is a $\w\cD(\bU,H)$-linear isomorphism  $\lim\limits_{y \in Y} \hsp M(\bU,H, ys) \congs M(\bU,H, YsB)$. Passing to the limit over all $s \in Z$ in the above commutative diagram we obtain a $\w\cD(\bU,H)$-linear isomorphism
\[ \alpha_Z : \bigoplus_{V \in \theta^{-1}(Z)} M(\bU,H, V) \congs M(\bU,J,Z).\]
Finally, by definition we have the direct sum decompositions
\[ M(\bU,H) = \bigoplus\limits_{Z \in J \backslash G / B} \hsp \bigoplus\limits_{V \in \theta^{-1}(Z)} M(\bU,H,V) \qmb{and} M(\bU,J) = \bigoplus\limits_{Z \in J \backslash G / B} M(\bU, J, Z).\]
Taking the direct sum of these $\alpha_Z$ over all $Z \in J \backslash G / B$ produces the required $\w\cD(\bU,H)$-linear isomorphism $M(\bU,H) \congs M(\bU,J)$.
\end{proof}
Using the connecting maps constructed in Proposition \ref{NUHtoJ}, we can form the inverse limit $\lim\limits_{\stackrel{\longleftarrow}{H}} M(\bU,H)$.
\begin{defn}\label{DefnOfM} Let $\cN \in \cC_{\bX/B}$ and let $\bU \in \bX_w(\cT)$. We define 
\[\boxed{\ind_B^G(\cN)(\bU) := \lim\limits_{\stackrel{\longleftarrow}{H}} \bigoplus_{Z \in H \backslash G / B} \lim\limits_{s \in Z} \w\cD(\bU,H) \WO{\w\cD(\bU, {}^sB \hsp \cap \hsp H)} [s] \hsp \cN(s^{-1}\bU)}\]
where the first limit runs over all $\bU$-small compact open subgroups $H$ of $G$.
\end{defn}
For brevity, we will write $\cM$ to mean $\ind_B^G(\cN)$ until the end of $\S \ref{PropIndFunctor}$.
\begin{cor}\label{NUHindepofH} The canonical maps $\cM(\bU) \to M(\bU,H)$ are bijections for any $\bU$-small compact open subgroup $H$ of $G$. \end{cor}
\begin{proof} This follows immediately from Proposition \ref{NUHtoJ}. \end{proof}

Our next goal is to define the restriction maps $\cM(\bU) \to \cM(\bV)$.
\begin{lem}\label{TauNUHs} Let $\bV$ be an $H$-stable affinoid subdomain of $\bU$ and let $s \in G$.
\be \item There is a $\w\cD(\bU,H)$-linear map $\tau^{\bU}_{\bV}(H,s) : M(\bU,H,s) \longrightarrow M(\bV,H,s)$. 
\item If $\bW \subset \bV$ is another $H$-stable affinoid subdomain, then \newline $\tau^{\bU}_{\bW}(H,s) = \tau^{\bV}_{\bW}(H,s) \circ \tau^{\bU}_{\bV}(H,s)$.
\item For every $h \in H$ and $b \in B$ there is a commutative diagram
\[ \xymatrix{ M(\bU, H, s) \ar[rr]^{\eta^H_{(h,b)}} \ar[d]_{\tau^{\bU}_{\bV}(H,s)} && M(\bU, H, hsb^{-1}) \ar[d]^{\tau^{\bU}_{\bV}(H,hsb^{-1})} \\ M(\bV, H, s) \ar[rr]_{\eta^H_{(h,b)}}  && M(\bV, H, hsb^{-1}). }\]
\ee\end{lem}
\begin{proof} Let the restriction maps $\w\cD(\bU,H) \to \w\cD(\bV,H)$ and $\cN(s^{-1}\bU) \to \cN(s^{-1}\bV)$ be denoted by $a \mapsto a_{|\bV}$ and $m \mapsto m_{|s^{-1}\bV}$. Then we set 
\[\tau^{\bU}_{\bV}(H,s)(a \hsp \w\otimes \hsp [s] m) := a_{|\bV} \hsp \w\otimes \hsp [s] m_{|s^{-1}\bV}.\]
This is well-defined and satisfies all of the properties stated in the Lemma.
\end{proof}

\begin{cor}\label{TauUVH} With the notation of Lemma \ref{TauNUHs}, there is a $\w\cD(\bU,H)$-linear map $\tau^{\bU}_{\bV}(H) : M(\bU,H) \longrightarrow M(\bV,H)$ such that $\tau^{\bU}_{\bW}(H) = \tau^{\bV}_{\bW}(H) \circ \tau^{\bU}_{\bV}(H)$.
\end{cor}
\begin{proof} Define $\tau^{\bU}_{\bV}(H) := \bigoplus\limits_{Z \in H \backslash G / B} \lim\limits_{s \in Z} \tau^{\bU}_{\bV}(H,s)$ and apply Lemma \ref{TauNUHs}.
\end{proof}

\begin{defn} We define $\tau^{\bU}_{\bV} : \cM(\bU) \to \cM(\bV)$ to be the inverse limit of the $\tau^{\bU}_{\bV}(H)$ as $H$ runs over all $\bU$-small and $\bV$-small compact open subgroups of $G$. 
\end{defn}

From Corollary \ref{TauUVH} we immediately deduce the following statement.
\begin{cor}\label{IndBGDmod} $\cM$ becomes a presheaf of $\cD$-modules on $\bX_w(\cT)$ when equipped with the restriction maps $\tau^{\bU}_{\bV}$.
\end{cor}

Next, we explain how to define the structure of a $G$-equivariant $\cD$-module presheaf on $\cM$. We omit some straightforward proofs.

\begin{lem}\label{Phigs} For every $g, s\in G$ there is a continuous $K$-linear isomorphism 
\[ \begin{array}{cccc} \phi_H(g,s) : & M(\bU,H,s) & \longrightarrow &M(g\bU, {}^gH, gs)\\
& a \hsp \w\otimes\hsp [s]m &\mapsto & \w{g}(a) \hsp \w\otimes\hsp [gs] m \end{array}\]
satisfying the following properties:
\be \item $\phi_H(g,s)(a \cdot b) = \w{g}(a) \cdot \phi_H(b,s)$ for all $a \in \w\cD(\bU,H)$ and $b \in M(\bU,H,s)$,
\item $\phi_H(g,gs) \circ \phi_H(h,s) = \phi_H(gh,s)$ for all $g,h \in G$, and 
\item $\phi_H(g,hsb^{-1}) \circ \eta^H_{(h,b)} = \eta^{{}^gH}_{({}^gh,b)}\circ\phi_H(g,s)$ for all $(h,b) \in H \times B$.
\ee
\end{lem}
Note that there is a bijection $H \backslash G / B \to {}^gH \backslash G / B$ given by $Z \mapsto gZ$. 
\begin{lem}\label{Phig} For every $g \in G$ there is a continuous $K$-linear isomorphism $\phi(g) : M(\bU,H) \longrightarrow M(g \bU, {}^g H)$ given by
\[\phi_H(g) := \bigoplus\limits_{Z \in H \backslash G / B}\lim\limits_{s \in Z} \phi_H(g,s) : M(\bU,H) \to M(g \bU, {}^g H)\] 
satisfying the following properties:
\be \item $\phi_H(g)(a \cdot b) = \w{g}(a) \cdot \phi_H(b)$ for all $a \in \w\cD(\bU,H)$ and $b \in M(\bU,H)$ and
\item $\phi_H(g) \circ \phi_H(h) = \phi_H(gh)$ for all $g,h \in G$.\ee
\end{lem}

\begin{cor} $\cM$ is a $G$-equivariant presheaf of $\cD$-modules on $\bX_w(\cT)$.\end{cor}
\begin{proof} In view of Corollary \ref{IndBGDmod} and \cite[Definition 2.3.4(a)]{EqDCap}, we have to define a $K$-linear morphism $g^{\cM} : \cM \to g^\ast \cM$ of presheaves on $\bX_w(\cT)$ for each $g \in G$, such that $g^{\cM}(a \cdot u) = g^{\cD}(a) \cdot g^{\cM}(u)$ for all $a \in \cD$ and $u \in \cM$ and $g^{\cM} \circ h^{\cM} = (gh)^{\cM}$ for all $g,h \in G$. Let $H \leq J$ be $\bU$-small compact open subgroups of $G$. It can be verified that $\phi_J(g) \circ \alpha^J_H = \alpha^{{}^gJ}_{{}^gH}\circ \phi_H(g)$ where  $\alpha^J_H$ are the isomorphisms from Proposition \ref{NUHtoJ}. Now define $g^{\cM} := \lim\limits_H \phi_H(g)$ and use Lemma \ref{Phig} to conclude.
\end{proof}

\subsection{Properties of the induction functor}\label{PropIndFunctor} We continue with the notation and hypotheses of $\S \ref{IndFunctor}$: thus, $G$ is a $p$-adic Lie group acting continuously on the smooth rigid analytic space $\bX$, $B$ is a co-compact closed subgroup of $G$ and $\cN$ is an object in $\cC_{\bX/B}$. We continue to abbreviate $\cM := \ind_B^G(\cN)$.
\begin{prop}\label{LocDeterm} Suppose that $(\bU,J)$ is small. Then there is an isomorphism  $\varphi : \cP^{\cD(\bU,J)}_{\bU}(\cM(\bU)) \congs \cM_{|\bU_w}$ of $J$-equivariant presheaves of $\cD$-modules on $\bU_w$.
\end{prop}
\begin{proof} The $\w\cD(\bU,J)$-module $\cM(\bU)$ is coadmissible by Corollaries \ref{MUHcoadm} and \ref{NUHindepofH}, so we may form presheaf $\cP := \cP^{\cD(\bU,J)}_{\bU}(\cM(\bU))$ on $\bU_w$ from \cite[Definition 3.5.3]{EqDCap}. Let $\bV \in \bU_w$, fix a $\bV$-small open subgroup $H$ of $J$ and let $s \in G$. Because $\cN \in \cC_{\bX/B}$ and the pair $(s^{-1}\bU, B \cap H^s)$ is small, it follows from \cite[Theorem B]{EqDCap} that there is a $\w\cD(\bV, {}^s B \cap H)$-module isomorphism
\[ \w\cD(\bV, {}^s B \cap H) \WO{\w\cD(\bU, {}^s B \cap H)} [s] \cM(s^{-1}\bU) \congs [s]\cM(s^{-1}\bV).\]
After applying the associativity isomorphism for $\w\otimes$ from \cite[Corollary 7.4]{DCapOne} and chasing the definitions, we deduce from this a $\w\cD(\bV,H)$-linear isomorphism
\[ \w\cD(\bV, H) \WO{\w\cD(\bU,H)} M(\bU,H,s) \congs M(\bV, H, s)\]
which assembles to a $\w\cD(\bV,H)$-linear isomorphism $\varphi_H$ in the following diagram:
\[ \xymatrix{ \cP(\bV) \ar[d]_{\cong}\ar@{.>}[rrr] &&& \cM(\bV)\ar[d]^{\cong} \\ \w\cD(\bV,H) \WO{\w\cD(\bU,H)} \cM(\bU) \ar[r]_(0.46){\cong} & \w\cD(\bV,H) \WO{\w\cD(\bU,H)} M(\bU,H) \ar[rr]^(0.6){\cong}_(0.6){\varphi_H} &&M(V,H).}\]
Here the vertical isomorphism on the left comes from \cite[Corollary 3.5.6]{EqDCap}, and the other two unlabelled isomorphisms come from Corollary \ref{NUHindepofH}. We define $\varphi(\bV) : \cP(\bV) \to \cM(\bV)$ be the unique dotted arrow in this diagram which makes it commutative. We leave to the reader the tasks of verifying that $\varphi(\bV)$ does not depend on the choice of $H$,  that it commutes with the restriction maps in the presheaves $\cP$ and $\cM$, and that it also commutes with the $J$-equivariant structures on these presheaves.
\end{proof}

\begin{cor}\label{Msheaf} $\cM$ is a \emph{sheaf} of $\cD$-modules on $\bX_w(\cT)$.
\end{cor}
\begin{proof} Apply Proposition \ref{LocDeterm} and \cite[Theorem 3.5.11]{EqDCap}.
\end{proof}

Recall that because $\bX$ is assumed to be smooth, $\bX_w(\cT)$ is a basis for the strong $G$-topology $\bX_{\rig}$ on $\bX$. It follows from Corollary \ref{Msheaf} and \cite[Theorem 9.1]{DCapOne} that $\cM$ extends uniquely to a $G$-equivariant sheaf $\ind_B^G(\cN)$ of $\cD$-modules on $\bX_{\rig}$. In fact, it is easy to see that it is a $G$-equivariant \emph{locally Fr\'echet} $\cD$-module on $\bX_{\rig}$ in the sense of \cite[Definition 3.6.1]{EqDCap}.

\begin{defn} We call $\ind_B^G(\cN)$ the \emph{induced $G$-equivariant $\cD$-module}.
\end{defn}

\begin{thm}\label{IndBGCoadm} For every $\cN \in \cC_{\bX/B}$, the induced $G$-equivariant $\cD$-module $\ind_B^G(\cN)$ is coadmissible: $\ind_B^G(\cN) \in \cC_{\bX/G}$.
\end{thm}
\begin{proof} Choose any $\bX_w(\cT)$-covering $\cU$ of $\bX$, fix $\bU \in \cU$ and choose a $\bU$-small subgroup $J$ of $G$ using \cite[Lemma 3.4.7]{EqDCap}.  Using \cite[Theorem 9.1]{DCapOne}, we see that the isomorphism $\varphi : \cP^{\w\cD(\bU,J)}_\bU(\cM(\bU)) \congs \cM_{|\bU_w}$ of $J$-equivariant presheaves of $\cD$-modules on $\bU_w$ from Proposition \ref{LocDeterm} extends uniquely to a continuous isomorphism $\Loc^{\w\cD(\bU,J)}_\bU(\cM(\bU)) \congs \ind_B^G(\cN)_{|\bU_{\rig}}$ of $J$-equivariant $\cD$-modules on $\bU_{\rig}$. It follows from \cite[Lemma 3.6.5]{EqDCap} that this isomorphism is continuous, so we see that $\ind_B^G(\cN)$ is $\cU$-coadmissible in the sense of \cite[Definition 3.6.7(a)]{EqDCap}.
\end{proof}

We have the following observation about supports.

\begin{lem}\label{IndSupp} $\ind_{B}^G$ defines a functor $\cC_{\bX/B}^{\bY} \longrightarrow \cC_{\bX/G}^{G\bY}$. 
\end{lem}
\begin{proof} By Theorem \ref{IndBGCoadm}, $\ind_{B}^G(\cN) \in \cC_{\bX/G}$ whenever $\cN \in \cC_{\bX/B}$. On the other hand, it follows from Definition \ref{DefnOfM} that $\ind_{B}^G(\cN)(\bU)$ is zero for any $\bU \in \bX_w(\cT)$ such that $\bU \cap G \bY = \emptyset$ whenever $\cN$ is supported on $\bY$ only.
\end{proof}

\begin{lem}\label{IndBGonSmallAffs} If $(\bX,G)$ is small, there is a canonical $\w\cD(\bX,G)$-linear isomorphism
\[ \Gamma(\bX, \ind_B^G(\cN)) \congs \w\cD(\bX,G) \WO{\w\cD(\bX,B)} \cN(\bX).\]
\end{lem}
\begin{proof} Because $\bX \in \bX_w(\cT)$, the map $\cM(\bX) \to \Gamma(\bX, \ind_B^G(\cN))$ is an isomorphism by Corollary \ref{Msheaf}. Now apply Corollary \ref{NUHindepofH} with $\bU = \bX$ and $H = G$.
\end{proof}

Fix an open subgroup $H$ of $G$. Using Lemma \ref{CompactGmodB}, we can find a finite set of representatives $\{s_1 = 1,\ldots, s_m\}$ for the $H,B$-double cosets in $G$. Recall the $s_i$-twist $[s_i] s_{i,\ast} \cN$ of $\cN$ from Definition \ref{sTwistDefn}; this is an object of $\cC_{\bX / {}^{s_i} B}$ by Proposition \ref{sTwist}, and we define $\cN_i := \Res^{{}^{s_i}B}_{H \cap {}^{s_i}B} \quad [s_i] s_{i,\ast} \cN$ be its restriction to $H \cap {}^{s_i} B$. 

We have the following sheaf-theoretic version of Mackey's restriction theorem.
\begin{lem}\label{SheafyMackey} Suppose $\bU$ is an affinoid subdomain of $\bX$ such that $(\bU,H)$ is small. Then $\ind_B^G(\cN)_{|\bU} \cong \bigoplus_{i=1}^m \ind_{H \cap {}^{s_i}B}^{\hsp\hsp H} ( \cN_{i | \bU} )$ in $\cC_{\bU/H}$.
\end{lem}
\begin{proof} By Theorem \ref{IndBGCoadm} and \cite[Proposition 3.6.10]{EqDCap}, both sheaves do lie in $\cC_{\bU/H}$. Because $(\bU,H)$ is small, it follows from  Lemma \ref{IndBGonSmallAffs} that $\Gamma(\bU, \ind_{H \cap {}^{s_i}B}^{\hsp\hsp H} ( \cN_{i|\bU} )) \cong \w\cD(\bU,H) \WO{\w\cD(\bU, H \cap {}^{s_i}B)} \cN_i(\bU)$. Applying Corollary \ref{Msheaf} and Definition \ref{DefnOfM}, we see that $\Gamma(\bU, \ind_B^G(\cN))$  is isomorphic to $\oplus_{i=1}^m\Gamma(\bU, \ind_{H \cap {}^{s_i}B}^{\hsp\hsp H} ( \cN_i ))$ as a $\w\cD(\bU,H)$-module.  The result now follows from \cite[Theorem 4.4.3]{EqDCap}.
\end{proof}

Now we specialise to the case $B = G_{\bY}$, and define $\cM_i := \ind_{H \cap {}^{s_i}B}^{\hsp\hsp H} ( \cN_i )$.

\begin{lem}\label{H0Comp4} Suppose that the $G$-orbit of $\bY$ is regular in $\bX$ and that $\cN \in \cC_{\bX/G_{\bY}}^{\bY}$. Then $\cH^0_{H s_i \bY}(\cM_i) = \cM_i$ and $\cH^0_{H s_j \bY}(\cM_i) = 0$ if $j \neq i$.
\end{lem}
\begin{proof} By Corollary \ref{sTwistSupp}, we know that $\Supp(\tilde{\cN}_i) \subseteq s_i \tilde{\bY}$. Hence $\Supp(\tilde{\cM}_i) \subseteq H s_i \tilde{\bY}$ for each $i$ by Lemma \ref{IndSupp}, so $\cH^0_{Hs_i\bY}(\cM_i) = \cM_i$.

Now, the sets $H s_i \bY$ are pairwise disjoint: if $h s_i y = h' s_j y'$ for some $h,h' \in H$ and $y, y' \in \bY$ then $(h s_i)^{-1} (h' s_j) \in G_{\bY}$ because the $G$-orbit of $\bY$ is regular in $\bX$, so $Hs_i G_{\bY} = Hs_j G_\bY$ and hence $i = j$ by our choice of the $s_i$'s.  It follows from Lemma \ref{YXtop}(d) that the sets $H s_i \tilde{\bY}$ are pairwise disjoint, so we can find pairwise disjoint open neighbourhoods $U_i$ of these sets in $\sP(\bX)$ by Theorem \ref{Separation}. Now if $j \neq i$ then $\tilde{\cM}_{i|U_j} = 0$ because $\Supp(\tilde{\cM}_i) \subseteq Hs_i \tilde{\bY}$ and $H s_i \tilde{\bY} \cap U_j = \emptyset$ if $j \neq i$. Hence $\cH^0_{Hs_j \bY}(\cM_i) = 0$ whenever $j \neq i$.
\end{proof}

\begin{prop}\label{H0HYindBG} Suppose that the $G$-orbit of $\bY$ is regular in $\bX$ and that $\cN \in \cC_{\bX/G_{\bY}}^{\bY}$. Suppose that $(\bU,H)$ is small. Then there is an isomorphism in $\cC_{\bX/H}$
\[\cH^0_{H(\bU \cap \bY)}(\ind_{G_{\bY}}^G(\cN)_{|\bU}) \cong \ind_{H_{\bY}}^H (\Res^{G_{\bY}}_{H_{\bY}} \cN_{|\bU})\]
\end{prop}
\begin{proof}By Lemma \ref{SheafyMackey}, $\ind_{G_{\bY}}^G(\cN)_{|\bU}$ is isomorphic to $\bigoplus_{i=1}^m \ind_{H \cap {}^{s_i}G_{\bY}}^{\hsp\hsp H} ( \cN_{i | \bU} )$ in $\cC_{\bU/H}$. Now $\cN_i \in \cC^{s_i\bY}_{\bX / H \cap {}^{s_i}G_{\bY}}$ by Corollary \ref{sTwistSupp}, so $\cN_{i|\bU} \in \cC^{s_i\bY \cap \bU}_{\bU/H \cap {}^{s_i}G_{\bY}}$. Hence $\cM'_i := \ind_{H \cap {}^{s_i}G_{\bY}}^{\hsp\hsp H} ( \cN_{i | \bU} ) \in \cC_{\bU/H}^{H(s_i \bY \cap \bU)}$ by Lemma \ref{IndSupp}. Because the $H(s_i \bY \cap \bU)$ are pairwise disjoint, a similar argument to the proof of Lemma \ref{H0Comp4} shows that $\cH^0_{H(s_i\bY \cap \bU)}(\cM'_i) = \cM'_i$ and $\cH^0_{H(s_j\bY \cap \bU)}(\cM'_i) = 0$ whenever $j \neq i$. Hence
\[ \cH^0_{H(s_i\bY \cap \bU)}(\ind_{G_{\bY}}^G(\cN)_{|\bU}) \cong \cM'_i = \ind_{H \cap {}^{s_i}G_{\bY}}^{\hsp\hsp H} ( \cN_{i | \bU} )\]
for each $i$, and the result follows by considering the case $i = 1$.
\end{proof}

\subsection{Structure of \ts{\ind_B^G(\cN)} and \ts{\cC_{\bX/G}^{G\bY}} on small affinoids}\label{H0IndBGsect} Throughout $\S \ref{H0IndBGsect}$, we assume that $(\bX,G)$ is small. We continue to assume that $B$ is a closed  subgroup of $G$ and that $\cN \in \cC_{\bX/B}$; note that this automatically implies that $G/B$ is compact because $G$ is itself compact. 

Because $(\bX,G)$ is small by assumption, we can find a $G$-stable affine formal model $\cA$ in $\cO(\bX)$ and a $G$-stable free $\cA$-Lie lattice $\cL$ in $\cT(\bX)$. Using \cite[Corollary 3.3.7]{EqDCap}, we fix a \emph{good chain} $(H_\bullet)$ for $\cL$; recall from \cite[Definitions 3.2.11 and 3.3.3]{EqDCap} that this means that $H_0 \geq H_1 \geq \cdots$ is a descending chain of open normal subgroups of $G$ with trivial intersection, such that $\rho(H_n) \leq \exp( p^\epsilon \pi^n \cL )$ for all $n\geq 0$, where $\rho : G \to \Aut_K(\cO(\bX))$ is the action of $G$ on $\cO(\bX)$. 

\begin{lem}\label{H0Comp3} Let $N$ be a coadmissible $\w\cD(\bX,B)$-module. Then the canonical map $N \longrightarrow \bigcap_{n=0}^\infty \w\cD(\bX, H_n B) \WO{\w\cD(\bX, B)} N$ is bijective.
\end{lem}
\begin{proof} Let $D_n := \hK{U(\pi^n \cL)}$, write $B_n := B \cap H_n$ and consider the following commutative diagram:
\[ \xymatrix{ N \ar[rr] \ar[ddd] && \lim\limits_{\stackrel{\longleftarrow}{k}} (D_k \underset{B_k}{\rtimes}{} B) \underset{\w\cD(\bX,B)}{\otimes}{} N \ar[d] \\
&& \lim\limits_{\stackrel{\longleftarrow}{k}} (D_k \underset{H_k}{\rtimes}{} H_kB) \underset{\w\cD(\bX,B)}{\otimes}{} N \ar[d] \\
&& \lim\limits_{\stackrel{\longleftarrow}{k}} \lim\limits_{\stackrel{\longleftarrow}{n \leq k}} (D_k \underset{H_k}{\rtimes}{} H_n B) \underset{\w\cD(\bX,B)}{\otimes}{} N \ar[d] \\
\lim\limits_{\stackrel{\longleftarrow}{n}} \w\cD(\bX, H_nB) \WO{\w\cD(\bX,B)} N \ar[rr] && \lim\limits_{\stackrel{\longleftarrow}{n}} \lim\limits_{\stackrel{\longleftarrow}{k \geq n}} (D_k \underset{H_k}{\rtimes}{} H_n B) \underset{\w\cD(\bX,B)}{\otimes}{} N. }\]
The map we wish to show is an isomorphism is the long vertical arrow on the left. The top horizontal arrow is an isomorphism because $N \in \cC_{\w\cD(\bX,B)}$ by assumption, and because $\invlim D_k \rtimes_{B_k} B = \w\cD(\bX,B)$ by \cite[Lemma 3.3.4]{EqDCap}. The first vertical arrow on the right is induced by the isomorphism $B / B_k \cong H_k B / H_k$. The second vertical arrow on the right is an isomorphism because for any fixed $k$ we have $\cap_{n = 0}^k H_nB = H_k B$. The third vertical arrow on the right is the isomorphism obtained by interchanging the order of inverse limits. Finally, for each fixed $n$, $\w\cD(\bX, H_nB)$ is isomorphic to $\lim\limits_{\stackrel{\longleftarrow}{k \geq n}} D_k \rtimes_{H_k} H_n B$ again by \cite[Lemma 3.3.4]{EqDCap}. The definition of $\w\otimes$ given in \cite[\S 7.3]{DCapOne} now implies that the bottom horizontal arrow is an isomorphism.
\end{proof}

\begin{prop}\label{H0Comp1} There is a continuous $\cD(\bX) \rtimes B$-linear isomorphism
\[  \alpha : H^0_{\bY}(\cM) \congs \invlim H^0_{H_n \bY}(\cM)\]
for every $G$-equivariant locally Fr\'echet $\cD$-module $\cM$ on $\bX$.
\end{prop}
\begin{proof} First, we claim that $\bigcap_{n=0}^\infty H_n \tilde{\bY} = \tilde{\bY}$. Suppose for a contradiction that $p \in \left(\bigcap_{n=0}^\infty H_n \tilde{\bY} \right)\backslash \tilde{\bY}$. Because $\tilde{\bY}$ is closed in $\sP(\bX)$ we can find $\bU \in \bX_w$ such that $p \in \tilde{\bU}$ and $\tilde{\bU} \cap \tilde{\bY} = \emptyset$. Since $G$ acts continuously on $\bX$, $G_{\bU}$ is an open subgroup of $G$ and therefore must contain one of the $H_n$'s because $\bigcap_{n=1}^\infty H_n$ is trivial and $G$ is compact. Since $p \in H_n \tilde{\bY}$, we can find some $h \in H_n$ such that $h^{-1} p \in \tilde{\bY}$. But then $h^{-1} p \in \tilde{\bU} \cap \tilde{\bY} = \emptyset$, a contradiction. Next, in the commutative diagram
\[\xymatrix{ 0 \ar[r] & H^0_{\bY}(\cM) \ar[r] \ar@{.>}[d]_{\alpha} & \cM(\bX) \ar[r]\ar@{=}[d] & \tilde{\cM}(\sP(\bX) \backslash \tilde{\bY}) \ar[d] \\ 0 \ar[r] & \invlim H^0_{H_n\bY}(\cM) \ar[r] & \cM(\bX) \ar[r] & \invlim \tilde{\cM}(\sP(\bX) \backslash H_n \tilde{\bY})}\]
the rightmost vertical arrow is a bijection because $\tilde{\cM}$ is a sheaf and because the complements of the closed sets $H_n \tilde{\bY}$ form an open covering of the complement of $\tilde{\bY}$. Therefore $\alpha$ is bijective by the Five Lemma. Now each $H_n \bY$ is $B$-stable because $H_n$ is normal in $G$. The continuity and $\cD(\bX) \rtimes B$-linearity of $\alpha$ now follow from Lemma \ref{TopOnLocCoh}.\end{proof}

\begin{lem}\label{H0HY} Suppose that the $G$-orbit of $\bY$ is regular in $\bX$ and that $\cN \in \cC_{\bX/G_{\bY}}^{\bY}$. Let $H$ be an open subgroup of $G$. Then the natural isomorphism from Lemma \ref{IndBGonSmallAffs} identifies $H^0_{H\bY}(\ind_{G_{\bY}}^G(\cN))$ with $\w\cD(\bX,HG_{\bY}) \WO{\w\cD(\bX,G_{\bY})} \cN(\bX)$.
\end{lem}
\begin{proof} By Proposition \ref{H0HYindBG} applied to $\bU = \bX$, we know that $\cH^0_{H\bY}(\ind_{G_{\bY}}^G(\cN)) \cong \ind_{H \cap G_{\bY}}^{\hsp \hsp H} \Res^{\hsp \hsp G_{\bY}}_{H \cap G_{\bY}} \cN$ in $\cC_{\bX/H}$. Taking global sections and applying Lemma \ref{IndBGonSmallAffs} gives a $\w\cD(\bX,H)$-linear isomorphism $H^0_{H\bY}(\ind_{G_{\bY}}^G(\cN)) \cong \w\cD(\bX,H) \WO{\w\cD(\bX,H \cap G_{\bY})} \cN(\bX)$. The result now follows from the isomorphism $(\ref{Mackey2})$. \end{proof}

We can now state and prove the first important result of $\S \ref{H0IndBGsect}$.

\begin{thm}\label{H0indBGsmall} If the $G$-orbit of $\bY$ is regular in $\bX$ and $\cN \in \cC_{\bX/G_{\bY}}^{\bY}$, then there is a natural $\w\cD(\bX,G_{\bY})$-linear isomorphism $\cN(\bX) \congs H^0_{\bY}(\ind_{G_{\bY}}^G(\cN))$.
\end{thm}
\begin{proof}Consider the following commutative square:
\[ \xymatrix{ \cN(\bX) \ar[r] \ar[d] & \bigcap\limits_{n=0}^\infty \w\cD(\bX,H_nG_{\bY}) \WO{\w\cD(\bX,G_{\bY})} \cN(\bX) \\
H^0_{\bY}(\ind_{G_{\bY}}^G(\cN)) \ar[r] & \bigcap\limits_{n=0}^\infty H^0_{H_n\bY}(\ind_{G_{\bY}}^G(\cN)). \ar[u]  
}\]
Because $\cN \in \cC_{\bX/G_{\bY}}$, we know that $\cN(\bX)$ is a coadmissible $\w\cD(\bX,G_{\bY})$-module by \cite[Theorem 4.4.3]{EqDCap}. Hence the top horizontal arrow is an isomorphism by Lemma \ref{H0Comp3}. The vertical arrow on the right is an isomorphism by Lemma \ref{H0HY}. Now $\ind_{G_{\bY}}^G(\cN) \in \cC_{\bX/G}$ by Theorem \ref{IndBGCoadm}, so the bottom horizontal arrow is an isomorphism by Proposition \ref{H0Comp1}. Therefore the vertical arrow on the left is also an isomorphism as required.\end{proof}

Next, we begin our study of the objects of $\cC_{\bX/G}^{G\bY}$.

\begin{thm}\label{Induced} Let $\cM \in \cC_{\bX/G}^{G\bY}$, suppose that the $G$-orbit of $\bY$ is regular in $\bX$ and let $H$ be an open normal subgroup of $G$. Then the natural $K[G]$-linear map
\[ \varphi : K[G] \underset{K[H G_{\bY} ]}{\otimes}{} H^0_{H \bY}(\cM) \longrightarrow \cM(\bX)\]
is an isomorphism.
\end{thm}
\begin{proof} By \cite[Proposition 2.3.5]{EqDCap},  $\cM(\bX)$ is a $K[G]$-module and $H^0_{H \bY}(\cM)$ is a $K[H G_{\bY}]$-submodule because $H \bY$ is a $H G_{\bY}$-stable subset of $\bX$. The map $\varphi$ is given by $\varphi(g \otimes m) = g^{\cM}(\bX)(m)$ for all $g \in G$ and $m \in H^0_{H\bY}(\cM)$. 

Choose $g_1,\ldots, g_m \in G$ such that $G = \coprod_{i=1}^m g_i H G_{\bY}$. Then $\varphi( g_i \otimes H^0_{H \bY}(\cM) ) = g_i^{\cM}(\bX) \left( H^0_{H \bY}(\cM) \right) = H^0_{g_iH \bY}(\cM)$ and we have to show that the direct sum of these subspaces of $\cM(\bX)$ equals $\cM(\bX)$.  Since $\cM$ is supported on $G \bY$, we know that $\Supp( \tilde{\cM} ) \subseteq \overline{ G \bY }$ by Lemma \ref{Supports}, so the restriction of $\tilde{\cM}$ to the complement of $\overline{G \bY}$ is zero. Because $\overline{G \bY} = G \tilde{\bY}$ by Corollary \ref{ClosureOfGY} it follows from Definition \ref{LocalCohoDefn} that $H^0_{G \bY}(\cM) = H^0_{G \tilde{\bY}}(\tilde{\cM}) = \tilde{\cM}(\sP(\bX)) = \cM(\bX).$ Now $G \bY$ is the union of the $g_iH \bY$, and these are pairwise disjoint by Lemma \ref{HYreg}. Hence, the closures $\overline{g_iH\bY}$ in $\sP(\bX)$ of the $g_i H\bY$ admit pairwise disjoint open neighbourhoods by Theorem \ref{Separation}. It follows from Lemma \ref{LocalCohoSeps} that the canonical map $H^0_{g_1 H \bY}(\cM) \oplus \cdots \oplus H^0_{g_m H \bY}(\cM) \longrightarrow H^0_{G \bY}(\cM)$ is an isomorphism.
\end{proof}

\begin{cor}\label{Induced2} With the notation of Theorem \ref{Induced}, 
\be \item $H^0_{H \bY}(\cM)$ is a coadmissible $\w\cD(\bX,H G_{\bY})$-module, and
\item the natural $\w\cD(\bX,G)$-linear map $\w\cD(\bX,G) \WO{\w\cD(\bX,H G_{\bY})} H^0_{H \bY}(\cM) \longrightarrow \cM(\bX)$ is an isomorphism.
\ee\end{cor}
\begin{proof} (a) By \cite[Theorem 4.4.3]{EqDCap}, $\cM(\bX)$ is a coadmissible $\w\cD(\bX,G)$-module. Because $H G_{\bY}$ is open in $G$ and $G$ is compact, $H G_{\bY}$ has finite index in $G$ so $\cM(\bX)$ is also a coadmissible $\w\cD(\bX,H G_{\bY})$-module. $H^0_{H \bY}(\cM)$ is a $\cD(\bX) \rtimes H G_{\bY}$-submodule of $\cM(\bX)$ by Theorem \ref{Induced}.  Now $H^0_{H \bY}(\cM)$ is closed in $\cM(\bX)$ by Lemma \ref{TopOnLocCoh}, so it is a $\w\cD(\bX,H G_{\bY})$-submodule because $\cD(\bX) \rtimes H G_{\bY}$ is dense in $\w\cD(\bX, H G_{\bY})$ by construction. 

Because $H$ is normal in $G$, $g H^0_{H \bY}(\cM)$ is a $\w\cD(\bX,H)$-submodule of $\cM(\bX)$ for each $g \in G$. By Theorem \ref{Induced}, $\cM(\bX)$ is a finite direct sum of such submodules. We see that $H^0_{H \bY}(\cM)$ is in fact a \emph{direct summand} of $\cM(\bX)$ as a $\w\cD(\bX,H)$-module, and is hence coadmissible. It is therefore per force coadmissible as a $\w\cD(\bX,H G_{\bY})$-module.

(b) The proof of \cite[Proposition 3.4.10(a)]{EqDCap} shows that the natural map 
\begin{equation}\label{GGU} K[G] \underset{K[H G_{\bY}]}{\otimes}{} \w\cD(\bX, H G_{\bY}) \longrightarrow \w\cD(\bX, G)\end{equation}
is an isomorphism of $K[G]-\w\cD(\bX,H G_{\bY})$-bimodules. Now apply part (a).
\end{proof}

Our next result forms the technical heart of the proof of Theorem \ref{InductionEquivalence}.

\begin{thm}\label{LocalCohoCoadm} Suppose that $(\bX,G)$ is small. Let $\bY$ be a Zariski closed subset of $\bX$ whose $G$-orbit is regular in $\bX$, and let $\cM \in \cC^{G\bY}_{\bX/G}$. Then
\be \item $H^0_{\bY}(\cM)$ is a coadmissible $\w\cD(\bX,G_{\bY})$-module, and
 \item there is a natural $\w\cD(\bX,G)$-module isomorphism 
 \[\w\cD(\bX,G) \WO{\w\cD(\bX,G_{\bY})} H^0_{\bY}(\cM)\congs \cM(\bX)\]
 which is functorial in $\cM$.
\ee\end{thm}

We now start working towards the proof of Theorem \ref{LocalCohoCoadm}. We assume that the $G$-orbit of $\bY$ is regular in $\bX$ and fix an object $\cM \in \cC^{G\bY}_{\bX/G}$ until the end of $\S \ref{H0IndBGsect}$. Recall that above we have chosen a $G$-stable affine formal model $\cA$ in $\cO(\bX)$, a $G$-stable free $\cA$-Lie lattice $\cL$ in $\cT(\bX)$ and a good chain $(H_\bullet)$ for $\cL$.

\begin{notn} Let $n \geq 0$. We set
\be \item $B := G_{\bY}$ and  $G_n := BH_n$,
\item $M := \cM(\bX)$ and $M(n) := H^0_{H_n \bY}(\cM)$,
\item $A := \w\cD(\bX,G)$ and $A(n) := \w\cD(\bX,G_n)$.
\item  $R_n := \hK{U(\pi^n \cL)} \rtimes_{H_n} G_n$ and $R := \invlim R_n$.
\ee\end{notn}
We now need some preliminary Lemmas.

\begin{lem} \label{Source1}The natural $A(n)$-linear map $\theta_n : A(n) \underset{A(n+1)}{\otimes}{} M(n+1) \tocong M(n)$ is an isomorphism for each $n \geq 0$.
\end{lem}
\begin{proof} Because $H_{n+1}\bY \subset H_n \bY$, we have $M(n+1) \subset M(n) \subset M$. Consider the following commutative diagram:
\[\xymatrix{ A \underset{A(n)}{\otimes}{} \left(A(n) \underset{A(n+1)}{\otimes}{} M(n+1)\right) \ar[rr]^(.6){1 \otimes \theta_n} \ar[d]_{\cong} && A \underset{A(n)}{\otimes}{} M(n) \ar[d] \\ A \underset{A(n+1)}{\otimes} M(n+1) \ar[rr] && M. }\]
The bottom horizontal arrow and the vertical arrow on the right are isomorphisms by Corollary \ref{Induced2}. Hence $1 \otimes \theta_n$ is an isomorphism. Now $(\ref{GGU})$ implies that $A$ is a free, and hence faithfully flat, right $A(n)$-module. Therefore $\theta_n$ is also an isomorphism.
\end{proof}
\begin{lem}\label{Source2} There is an isomorphism $\w\cD(\bX,B) \tocong R$.
\end{lem}
\begin{proof} Because $B / B \cap H_n  \cong G_n / H_n$, there are isomorphisms  of $K$-Banach algebras $\hK{U(\pi^n \cL)} \rtimes_{B\cap H_n} B \cong R_n$, compatible with variation in $n$. Since $(B \cap H_\bullet)$ is a good chain for $\cL$ in $B$, it follows from \cite[Lemma 3.3.4]{EqDCap} that there is an isomorphism $\w\cD(\bX,B) \tocong \invlim \hK{U(\pi^n \cL)} \rtimes_{B\cap H_n} B$.  Hence $\w\cD(\bX,B) \cong \invlim R_n$ as claimed.
\end{proof}

\begin{defn} $N_n := R_n \underset{A(n)}{\otimes}{} M(n)$ for each $n \geq 0$, and $N := \invlim N_n$.\end{defn}

\begin{lem}\label{Source3} $N$ is a coadmissible $R$-module. \end{lem}
\begin{proof}
Note that $N_n$ is a finitely generated $R_n$-module because $M(n) = H^0_{H_n\bY}(\cM)$ is a coadmissible $A(n)$-module by Corollary \ref{Induced2}. The inclusion $M(n+1) \subset M(n)$ together with functoriality induces a connecting $R_{n+1}$-linear map $N_{n+1} \to N_n$. This gives an $R_n$-linear map $\psi_n : R_n \otimes_{R_{n+1} } N_{n+1} \to N_n$ which features in the following commutative diagram:
\[ \xymatrix{ R_n \underset{R_{n+1}}{\otimes}{} N_{n+1} \ar@{=}[rr]  \ar[d]_{\psi_n} & &  R_n \underset{R_{n+1}}{\otimes}{} \left(R_{n+1} \underset{A(n+1)}{\otimes}{} M(n+1)\right)  \ar[d]^{\cong} 
\\ N_n \ar@{=}[d] & & R_n \underset{A(n+1)}{\otimes}{} M(n+1) \ar[d]^{\cong} \\
R_n \underset{A(n)}{\otimes}{} M(n)  & & R_n \underset{A(n)}{\otimes}{} \left( A(n) \underset{A(n+1)}{\otimes}{} M(n+1) \right) \ar[ll]^(0.63){1 \otimes \theta_n}}\]
Here $\theta_n : A(n) \underset{A(n+1)}{\otimes}{} M(n+1) \tocong M(n)$ is the $A(n)$-linear isomorphism from Lemma \ref{Source1}. Hence $\psi_n$ is an isomorphism and $N$ is a coadmissible $R$-module.
\end{proof}

\begin{lem}\label{H0Comp2} The natural map $\beta : \invlim M(n) \to \invlim N_n$ is a continuous $\w\cD(\bX,B)$-linear isomorphism. \end{lem}
\begin{proof} For each $m \geq n\geq 0$ we define $A(n)_m := \hK{U(\pi^m \cL)} \rtimes_{H_m} G_n$ so that $\lim\limits_{\stackrel{\longleftarrow}{m \geq n}} A(n)_m = A(n)$ for each $n \geq 0$. Because $M(n)$ is a coadmissible $A(n)$-module by Corollary \ref{Induced2}, we have an isomorphism $\sigma_n : M(n) \cong \lim\limits_{\stackrel{\longleftarrow}{m \geq n}} A(n)_m \underset{A(n)}{\otimes}{} M(n)$ of $A(n)$-modules.  The canonical maps $R$-linear maps $M(n) \to N_n = R_n \underset{A(n)}{\otimes}{} M(n)$ that send $v \in N_n$ to $1 \otimes v \in R_n \underset{A(n)}{\otimes}{} M(n)$ induce an $R$-linear map $\beta : \invlim M(n) \to \invlim N_n$, appearing in the following commutative diagram:
\[ \xymatrix{\lim\limits_{\stackrel{\longleftarrow}{n}} M(n) \ar[d]^{\cong}_(0.6){\invlim \sigma_n} \ar[r]^{\beta} & \lim\limits_{\stackrel{\longleftarrow}{n}} N_n \ar@{=}[r] & \lim\limits_{\stackrel{\longleftarrow}{m}} R_m \underset{A(m)}{\otimes}{} M_m \ar[d] \\ \lim\limits_{\stackrel{\longleftarrow}{n}} \lim\limits_{\stackrel{\longleftarrow}{m \geq n}} A(n)_m \underset{A(n)}{\otimes}{} M(n) \ar[rr] && \lim\limits_{\stackrel{\longleftarrow}{m}} \lim\limits_{\stackrel{\longleftarrow}{n \leq m}} A(n)_m \underset{A(n)}{\otimes}{} M(n). }\]
The bottom horizontal arrow is the isomorphism coming from swapping the order of limits, and the vertical arrow on the right is an isomorphism because for $m$ fixed, the partially ordered set $\{0 \leq n \leq m\}$ has $n = m$ as its largest member and $A(m)_m = R_m$. We conclude that $\beta$ is an $R$-linear bijection. It is continuous because each of the maps $M(n) \to N_n$ is continuous.
\end{proof}

\begin{proof}[Proof of Theorem \ref{LocalCohoCoadm}]  Note that $N = \invlim N_n$ is a coadmissible $R$-module by Lemma \ref{Source3}, so $N$ is a coadmissible $\w\cD(\bX,B)$-module by Lemma \ref{Source2}. Consider the following commutative triangle:
\[ \xymatrix{ H^0_{\bY}(\cM) \ar[rr]^{\alpha} \ar[dr]_{\gamma} & & \invlim M(n) \ar[dl]^{\beta} & \\ & \invlim N_n. & }\]
The maps $\alpha$ and $\beta$ are continuous $\cD(\bX)\rtimes B$-linear bijections by Proposition \ref{H0Comp1} and Lemma \ref{H0Comp2}, respectively. Thus, $\gamma$ is a continuous bijection between two Fr\'echet spaces, so $\gamma^{-1}$ is also a continuous bijection by the Open Mapping Theorem \cite[Proposition 8.6]{SchNFA}.  It now follows from \cite[Lemma 4.4.4]{EqDCap} that $H^0_{\bY}(\cM)$ is a coadmissible $\w\cD(\bX,B)$-module and that $\gamma$ is a $\w\cD(\bX,B)$-linear isomorphism.

The definitions given in \cite[\S 7.3]{DCapOne} imply that there is an isomorphism
\[ A \hsp \WO{R} \hsp N \cong \invlim A \underset{R_n}{\otimes}{} N_n \cong \invlim A \underset{A(n)}{\otimes}{} M(n).\]
But $A \otimes_{A(n)} M(n)$ is isomorphic to $M = \cM(\bX)$ by Theorem \ref{Induced} for all $n\geq 0$.   Thus, the natural $A$-linear map $A \WO{R} N \longrightarrow M$ is an isomorphism.
\end{proof}

\subsection{Proof of the induction equivalence}\label{PfIndEq} We can now start working on the proof of Theorem \ref{InductionEquivalence}. We will assume throughout $\S\ref{PfIndEq}$ that:
\begin{itemize}
\item $\bX$ is a smooth rigid analytic space,
\item $G$ is a $p$-adic Lie group acting continuously on $\bX$,
\item $\bY$ is a Zariski closed subset of $\bX$ whose stabiliser $G_{\bY}$ is co-compact in $G$.
\end{itemize}
Our first aim will be to show that $\cH^0_{\bY}(\cM) \in \cC_{\bX/G_{\bY}}$ whenever $\cM \in \cC^{G\bY}_{\bX/G}$ and the $G$-orbit of $\bY$ is regular in $\bX$. Recall from Lemma \ref{CompactGmodB} that the set of double cosets $J \backslash G / G_{\bY}$ is finite whenever $J$ is an open subgroup of $G$.

\begin{lem}\label{MsupponGY} Let $\cM \in \cC^{G\bY}_{\bX/G}$ and suppose that the $G$-orbit of $\bY$ is regular in $\bX$. 
Let $(\bU,J)$ be small and let $s_1,\ldots,s_n$ be representatives for the $J, G_{\bY}$-double cosets in $G$. Then
\be \item $G \bY \cap \bU = J (s_1 \bY \cap \bU) \cup \cdots \cup J(s_n \bY \cap \bU)$,
\item there is an isomorphism of locally Fr\'echet $J$-equivariant $\cD$-modules on $\bU$
\[\cM_{|\bU} \cong \bigoplus_{i=1}^n \cH^0_{J (s_i \bY \cap \bU)}(\cM_{|\bU}),\]
\item $\cH^0_{J (s \bY \cap \bU)}(\cM_{|\bU}) \in \cC_{\bU/J}^{J(s \bY \cap \bU)}$ for all $s \in G$.
\ee \end{lem}
\begin{proof} (a) This is clear. (b) Since $\cM$ is supported on $G\bY$, its restriction to $\bU$ is supported on $G \bY \cap \bU$.  Because the $G$-orbit of $\bY$ is regular in $\bX$, the $J$-orbit of $s \bY \cap \bU$ is regular in $\bU$ for any $s \in G$, and $J(s_i \bY \cap \bU) \cap J(s_j \bY \cap \bU) =\emptyset$ if $i \neq j$.  Applying Theorem \ref{Separation} and Lemma \ref{LocalCohoSeps} to the Zariski closed subsets $s_1\bY\cap \bU, \cdots, s_n \bY\cap \bU$ of $\bU$ and the compact group $J$, we obtain a $\cD(\bU) \rtimes J$-linear isomorphism 
\begin{equation}\label{MUdecomp} \cM(\bU) \tocong H^0_{G\bY \cap \bU}(\cM_{|\bU}) \tocong \bigoplus_{i=1}^n H^0_{J (s_i \bY \cap \bU)}(\cM_{|\bU}). \end{equation}
In a similar manner, using Lemma \ref{TopOnLocCoh} we obtain a direct sum decomposition 
\begin{equation}\label{SheafyMUdecomp}\cM_{|\bU} \tocong \cH^0_{G\bY \cap \bU}(\cM_{|\bU}) \tocong \bigoplus_{i=1}^n \cH^0_{J (s_i \bY \cap \bU)}(\cM_{|\bU})\end{equation}
of $J$-equivariant locally Fr\'echet $\cD$-modules on $\bU$.

(c) Lemma \ref{TopOnLocCoh} implies that the direct summands of the $\cD(\bU) \rtimes J$-module $\cM(\bU)$ appearing in equation (\ref{MUdecomp}) are closed. Since $\cM(\bU)$ is a coadmissible $\w\cD(\bU,J)$-module, they must be coadmissible $\w\cD(\bU,J)$-submodules of $\cM(\bU)$. Because the functor $\Gamma(\bU,-)$ sends $(\ref{SheafyMUdecomp})$ to $(\ref{MUdecomp})$, \cite[Theorem 3.6.11]{EqDCap} implies that the corresponding finitely many direct summands $\cM_{|\bU}$ appearing in equation $(\ref{SheafyMUdecomp})$ all lie in $\cC_{\bU/J}$. 

Finally, $J(s \bY \cap \bU) = J(s_i \bY \cap \bU)$ whenever $s \in J s_i G_{\bY}$, so $\cH^0_{J (s \bY \cap \bU)}(\cM_{|\bU}) = \cH^0_{J (s_i \bY \cap \bU)}(\cM_{|\bU}) \in \cC_{\bU/J}$. By construction, $\cH^0_{J (s \bY \cap \bU)}(\cM_{|\bU})$ is supported on $J(s\bY \cap \bU)$ only.
\end{proof}

\begin{lem}\label{PhiUJ} Suppose that $(\bX,G)$ is small. Let $\cM \in \cC^{G\bY}_{\bX/G}$ and suppose that the $G$-orbit of $\bY$ is regular in $\bX$. Then whenever $(\bU,J)$ is small, there is a natural $\w\cD(\bU,J_{\bY})$-linear isomorphism
\[ \varphi(\bU,J) \hsp : \hsp \w\cD(\bU,J_{\bY}) \WO{\w\cD(\bX,J_{\bY})} H^0_{\bY}(\cM) \tocong H^0_{\bY\cap \bU}(\cM_{|\bU}).\]
\end{lem}
\begin{proof} Because $J$ is open in $G$, it follows from Theorem \ref{LocalCohoCoadm}(a) that $H^0_{\bY}(\cM)$ is a coadmissible $\w\cD(\bX, J_{\bY})$-module and $H^0_{\bY \cap \bU}(\cM_{|\bU})$ is a coadmissible $\w\cD(\bU,J_{\bY})$-module. Using the universal property of $\w\otimes$ given in \cite[Lemma 7.3]{DCapOne}, we see that the restriction map $H^0_{\bY}(\cM) \to H^0_{\bY \cap \bU}(\cM_{|\bU})$ induces a $\w\cD(\bU,J_{\bY})$-linear map
\[ \varphi(\bU,J) \hsp : \hsp \w\cD(\bU,J_{\bY}) \WO{\w\cD(\bX,J_{\bY})} H^0_{\bY}(\cM) \longrightarrow H^0_{\bY\cap \bU}(\cM_{|\bU}).\]
We must show $\varphi(\bU,J)$ is an isomorphism. Now $\cH^0_{J(\bY\cap \bU)}(\cM_{|\bU}) \in \cC_{\bU/J}^{J(\bY\cap \bU)}$ by Lemma \ref{MsupponGY}(c), so by Theorem \ref{LocalCohoCoadm}(b) there is a $\w\cD(\bU,J)$-linear isomorphism
\begin{equation}\label{PreCoh} \w\cD(\bU,J)\WO{\w\cD(\bU,J_{\bY})} H^0_{\bY \cap \bU}(\cM_{|\bU}) \tocong H^0_{J (\bY \cap \bU)}(\cM_{|\bU}). \end{equation}
Next, $H^0_\bY(\cM)$ is a coadmissible $\w\cD(\bX,G_{\bY})$-module by Theorem \ref{LocalCohoCoadm}(a), so we may form the localisation $\cN := \Loc^{\w\cD(\bX,G_{\bY})}_{\bX}( H^0_\bY(\cM) ) \in \cC_{\bX/G_{\bY}}$. Then we have the following commutative diagram:
\[ \xymatrix{ \w\cD(\bU,J) \WO{\w\cD(\bX,J)}  \left( \w\cD(\bX,G) \WO{\w\cD(\bX,G_{\bY})} H^0_{\bY}(\cM) \right) \ar@{.>}[r]^(0.71){\cong}\ar[d]_{\cong} & \cM(\bU) \ar[d]^{\cong} \\  \bigoplus\limits_{i=1}^n \w\cD(\bU,J) \WO{\w\cD(\bU,J \cap {}^{s_i}G_{\bY})} [s_i]\cN(s_i^{-1}\bU)  \ar@{.>}[r] &  \bigoplus\limits_{i=1}^n H^0_{J (s_i \bY \cap \bU)}(\cM_{|\bU}). }\]
Here the isomorphism on the left comes from the Mackey decomposition, Proposition \ref{Mackey}, whereas the arrow on the right is the decomposition appearing in equation (\ref{MUdecomp}). The top horizontal arrow is obtained by combining the isomorphisms
\[ \w\cD(\bX,G) \WO{\w\cD(\bX,G_{\bY})} H^0_{\bY}(\cM) \tocong \cM(\bX) \qmb{and} \w\cD(\bU,J) \WO{\w\cD(\bX,J)} \cM(\bX) \tocong \cM(\bU)\]
coming from Theorem \ref{LocalCohoCoadm}(b) and \cite[Theorem 4.4.3]{EqDCap}, respectively, and it induces the bottom horizontal arrow in the diagram. Because this last arrow respects the direct sum decomposition, we obtain $\w\cD(\bU,J)$-linear isomorphisms
\begin{equation}\label{DUJisos}\w\cD(\bU,J) \WO{\w\cD(\bU,J \cap {}^{s_i}G_{\bY})} [s_i]\cN(s_i^{-1}\bU)  \tocong H^0_{J (s_i \bY \cap \bU)}(\cM_{|\bU})\end{equation}
for all $i=1,\ldots, n$. Now, consider the following commutative triangle:
\[\xymatrix{ \w\cD(\bU,J) \WO{\w\cD(\bU, J_{\bY})} \left(\w\cD(\bU,J_{\bY}) \WO{\w\cD(\bX,J_{\bY})} H^0_{\bY}(\cM) \right) \ar[r]^(0.71){\cong}\ar[d]_{1 \hsp \w\otimes \hsp \varphi(\bU,J)} & H^0_{J (\bY \cap \bU)}(\cM_{|\bU}) \\\w\cD(\bU,J)\WO{\w\cD(\bU,J_{\bY})} H^0_{\bY \cap \bU}(\cM_{|\bU}).  \ar[ur]_{\cong} &
 }\]
Here the horizontal isomorphism comes from equation (\ref{DUJisos}) by taking $i=1$, and the diagonal isomorphism comes from equation (\ref{PreCoh}). We can now finally apply Lemma \ref{cfflat} below to conclude that $\varphi(\bU,J)$ is an isomorphism. \end{proof}

We refer the reader to \cite[Definition 7.5]{DCapOne} for the notion of \emph{faithfully c-flatness}.

\begin{lem}\label{cfflat} Suppose that $(\bX,G)$ is small, and let $H$ be a closed subgroup of $G$. Then $\w\cD(\bX,G)$ is a faithfully c-flat $\w\cD(\bX,H)$-module on both sides.
\end{lem}
\begin{proof} Fix a $G$-stable affine formal model $\cA$ in $\cO(\bX)$ and a $G$-stable free $\cA$-Lie lattice $\cL$ in $\cT(\bX)$. Choose a good chain $(N_\bullet)$ for $\cL$ in $G$, so that $\rho(N_n) \leq \exp(p^\epsilon \pi^n \cL)$ for all $n\geq 0$ where $\rho$ is the action of $G$ on $\cO(\bX)$. Then $(N_\bullet \cap H)$ is a good chain for $\cL$ in $H$, and by \cite[Lemma 3.3.4]{EqDCap} we have presentations of Fr\'echet-Stein algebras
\[  \w\cD(\bX,H) \cong \invlim \hK{U(\pi^n \cL)} \underset{N_n \cap H}{\rtimes}{} H \qmb{and} \w\cD(\bX,G) \cong \invlim \hK{U(\pi^n \cL)} \underset{N_n}{\rtimes}{} G.\]
Now $\hK{U(\pi^n \cL)} \rtimes_{N_n} G$ is a free $\hK{U(\pi^n \cL)} \rtimes_{N_n} N_n H$-module on both sides, in view of \cite[Lemma 2.2.4(b)]{EqDCap}. However this latter algebra is naturally isomorphic to $\hK{U(\pi^n \cL)} \rtimes_{N_n \cap H} H$, and the result now follows from \cite[Proposition 7.5(b,c)]{DCapOne}.
\end{proof}

\begin{lem}\label{LocOfH0} Suppose that $(\bX,G)$ is small. Let $\bY$ be a Zariski closed subset of $\bX$ whose $G$-orbit is regular in $\bX$ and let $\cM \in \cC_{\bX/G}^{G \bY}$. Then there is an isomorphism of $G_{\bY}$-equivariant locally Fr\'echet $\cD$-modules on $\bX$
\[\Loc^{\w\cD(\bX,G_{\bY})}_{\bX}( H^0_{\bY}(\cM) ) \tocong \cH^0_{\bY}(\cM).\]
\end{lem}
\begin{proof} Note that $H^0_{\bY}(\cM)$ is a coadmissible $\w\cD(\bX, G_{\bY})$-module by Theorem \ref{LocalCohoCoadm}(a). In view of \cite[Definitions 3.5.1, 3.5.3 and 3.5.12]{EqDCap}, we have to exhibit a $\cD(\bU)$-linear isomorphism
\[ \varphi(\bU) : \lim\limits_{\stackrel{\longleftarrow}{H}} H^0_{\bY}(\cM)(\bU,H) \quad \longrightarrow \quad \cH^0_{\bY}(\cM)(\bU)\]
which commutes with the restriction maps and $G_{\bY}$-equivariant structure maps on both sides. Here the inverse limit on the left hand side runs over the set $\mathcal{S}$ of $\bU$-small open subgroups $H$ of $G_{\bY}$. It follows from \cite[Lemma 3.4.5]{EqDCap} that the open subgroups of the form $J_{\bY}$ where $J$ is a $\bU$-small subgroup of $G$ form a cofinal family in $\mathcal{S}$, so we may take the inverse limit over this smaller family. By \cite[Definition 3.5.1]{EqDCap}, we have
\[H^0_{\bY}(\cM)(\bU,J_{\bY}) = \w\cD(\bU,J_{\bY}) \WO{\w\cD(\bX,J_{\bY})} H^0_{\bY}(\cM).\]
Unwinding the definitions, we can now see that we can take $\varphi(\bU)$ to be the inverse limit of the isomorphisms $\varphi(\bU,J)$ that were constructed in Lemma \ref{PhiUJ}.
\end{proof}
We would like to show the local cohomology sheaf $\cH^0_{\bY}(\cM)$ lies in $\cC_{\bX/G_{\bY}}$. Towards this, we have the following result.
\begin{prop}\label{H0nearlyCoad} Suppose that the $G$-orbit of $\bY$ is regular in $\bX$, and let $\cM \in \cC^{G \bY}_{\bX/\bG}$. Then whenever $(\bU,J)$ is small, we have
\[ \cH^0_{\bY}(\cM)_{|\bU}  \in \cC_{\bU / J_{\bY \cap \bU}}.\]
\end{prop}
\begin{proof} Define $\cM' := \cH^0_{J (\bY \cap \bU)}(\cM_{|\bU})$, 
 an object of $\cC_{\bU/J}^{J(\bY \cap \bU)}$ by Lemma \ref{MsupponGY}(c). Since $(\bU, J)$ is small and since $J$-orbit of $\bU \cap \bY$ is regular in $\bU$, applying Lemma \ref{LocOfH0} to $\cM' \in \cC_{\bU/J}^{J(\bY \cap \bU)}$ gives an isomorphism of $J_{\bY \cap \bU}$-equivariant locally Fr\'echet $\cD$-modules on $\bU$
\[\Loc^{\w\cD(\bU,J_{\bY\cap \bU})}_{\bX}( H^0_{\bY \cap \bU}(\cM') ) \tocong \cH^0_{\bY \cap \bU}(\cM').\]
However, because $\bY \cap \bU \subset J(\bY \cap \bU)$, using the definition of $\cM'$ we have
\[\cH^0_{\bY \cap \bU}( \cM'  )= \cH^0_{\bY \cap \bU}(\cM_{|\bU}) = \cH^0_{\bY}(\cM)_{|\bU}.\]
The result follows by combining the last two displayed equations. \end{proof}
At this point, in order to proceed we must introduce the following condition.
\begin{defn}\label{LSC} We say that $(\bX, \bY, G)$ satisfies the \emph{Local Stabiliser Condition (LSC)}\label{LSCdefn} if for all $\bU \in \bX_w(\cT)$ such that $\bY \cap \bU \neq \emptyset$, there is a $\bU$-small compact open subgroup $J$ of $G$ such that 
\[J_{\bY \cap \bU} = J_{\bY}.\]
\end{defn}
Note that in practice it could easily happen that $\bY \cap \bU$ is the empty set so $J_{\bY \cap \bU} = J$ regardless of what $J$ is, and yet $J_{\bY}$ will in general be a proper subgroup of $J$. For an explicit example, take $\bX := \bP^{1,\an}$, $G := \SL_2(L)$ for some finite extension $L$ of $\Qp$, $\bY := \{\infty \}$ and $\bU := \Sp K \langle x \rangle$ the closed unit disc in $\bP^{1,\an}$.  

\begin{thm}\label{SheafyLocCohCoadm} Suppose that $(\bX, \bY, G)$ satisfies the LSC, and that the $G$-orbit of $\bY$ is regular in $\bX$. Then $\cH^0_{\bY}$ defines a functor $\cH^0_{\bY} : \cC_{\bX/G}^{G \bY} \to \cC_{\bX/G_{\bY}}^{\bY}.$
\end{thm}
\begin{proof}
Let $\cM \in \cC_{\bX/G}^{G \bY}$. By Lemma \ref{TopOnLocCoh}, we know that $\cH^0_{\bY}(\cM)$ is a $G_{\bY}$-equivariant $\cD$-module on $\bX$ supported only on $\bY$, and the functorial nature of $\cH^0_{\bY}$ is clear. So, we only have to show that $\cH^0_{\bY}(\cM) \in \cC_{\bX/G_{\bY}}$.

Let $\bU \in \bX_w(\cT)$. If $\bY \cap \bU = \emptyset$, then there is nothing to do because in this case $\cH^0_{\bY}(\cM)_{|\bU} = 0$. Assume therefore that $\bY \cap \bU \neq \emptyset$. Because $(\bX,\bY,G)$ satifies the LSC, there exists a $\bU$-small subgroup  $J$ of $G$ such that $J_{\bY \cap \bU} = J_{\bY}$. Now use \cite[Definition 3.6.7]{EqDCap} together with  Proposition \ref{H0nearlyCoad} to deduce $\cH^0_{\bY}(\cM) \in \cC_{\bX/G_{\bY}}$. \end{proof}

We can now prove the following globalisation of Theorem \ref{H0indBGsmall}.

\begin{prop}\label{UnitIsom} Suppose that $(\bX, \bY, G)$ satisfies the LSC, and that the $G$-orbit of $\bY$ is regular in $\bX$. Then there is a natural isomorphism 
\[\eta : 1_{\cC_{\bX/G_{\bY}}^{\bY}} \congs \cH^0_{\bY} \circ \ind_{G_{\bY}}^G.\] 
\end{prop}
\begin{proof}Let $\cN \in \cC_{\bX/G_{\bY}}^{\bY}$ and let $\cM := \ind_{G_{\bY}}^G(\cN)$. Then $\cM$ is an object in $\cC_{\bX/G}^{G\bY}$ by Lemma \ref{IndSupp} and $\cH^0_{\bY}(\cM)$ is an object in $\cC^{\bY}_{\bX/G_{\bY}}$ by Theorem \ref{SheafyLocCohCoadm}. Fix $\bU \in \bX_w(\cT)$ and choose a $\bU$-small compact open subgroup $H$ of $G$. Now, $\cH^0_{\bY}(\cM)(\bU) = H^0_{\bU \cap \bY}( \cM_{|\bU} ) = H^0_{\bU \cap \bY}\left( \cH^0_{H(\bU \cap \bY)}(\cM_{|\bU}) \right)$. By Proposition \ref{H0HYindBG}, we have an isomorphism 
$\cH^0_{H(\bU \cap \bY)}(\cM_{|\bU}) \cong \ind_{H_{\bY}}^H \left(\Res^{G_{\bY}}_{H_{\bY}} \cN_{|\bU}\right)$ in $\cC_{\bU/H}$. Because the $G$-orbit of $\bY$ is regular in $\bX$, the $H$-orbit of $\bU \cap \bY$ is regular in $\bU$. Applying Theorem \ref{H0indBGsmall} to $\Res^{G_{\bY}}_{H_{\bY}} \cN_{|\bU} \in \cC^{\bU \cap \bY}_{\bU/H_{\bY}}$ gives a natural $\w\cD(\bU,H_{\bY})$-linear isomorphism $\cN(\bU) \congs H^0_{\bU \cap \bY}(\ind_{H_{\bY}}^H(\Res^{G_{\bY}}_{H_{\bY}}\cN_{|\bU}) )$. Combining all these facts gives us a continuous $\cD(\bU)$-linear isomorphism
\[ \eta_{\cN}(\bU) : \cN(\bU) \congs \cH^0_{\bY}(\cM)(\bU).\]
It is straightforward to check that this isomorphism does not depend on the choice of $H$. We leave to the reader the tasks of verifying that $\eta_{\cN}$ commutes with the restriction maps and the $G_{\bY}$-equivariant structure of the sheaves $\cN$ and $\cH^0_{\bY}(\ind_{G_{\bY}}^G(\cN))$, and that $\eta_{\cN}$ is functorial in $\cN$.
\end{proof}

Our next goal will be to exhibit a functorial isomorphism 
\[\epsilon : \ind_{G_{\bY}}^G \hsp \circ \hsp \cH_{\bY}^0 \congs 1_{\cC_{\bX/G}^{G\bY}}.\]
We are again unable to do this without the LSC. First we will show that the LSC implies the following stronger equivariant version.

\begin{prop}\label{LSCb}Suppose that $(\bX,\bY,G)$ satisfies the LSC. Then for all $\bU \in \bX_w(\cT)$ there is a $\bU$-small compact open subgroup $J$ of $G$ such that whenever $s \bY \cap \bU \neq \emptyset$ for some $s \in G$, we have $J_{s \bY \cap \bU} = J_{s \bY}$.
\end{prop}
\begin{proof} Let $Z := \{s \in G : s \bY \cap \bU \neq \emptyset\}$. Note that $Z$ and $G \backslash Z$ are both stable under left-multiplication by $G_{\bU}$ and right-multiplication by $G_{\bY}$. Since $G_{\bU}$ is an open subgroup of $G$ by \cite[Definition 3.1.8]{EqDCap}, it follows that $Z$ is a clopen subset of $G$. Since $Z \cdot G_{\bY} = Z$, the image $Z/G_{\bY}$ of $Z$ in $G / G_{\bY}$ is also clopen, and hence compact since we're assuming that $G_{\bY}$ is co-compact in $G$.

For each $s \in Z$, the intersection $\bY \cap s^{-1}\bU$ is non-empty. We can therefore apply the LSC to $s^{-1}\bU \in \bX_w(\cT)$ to find an $s^{-1}\bU$-small compact open subgroup $H(s)$ of $G$ such that $H(s)_{\bY \cap s^{-1}\bU} = H(s)_{\bY}$. Then $J(s) := s H(s) s^{-1}$ is $\bU$-small, and $J(s)_{s \bY \cap \bU} = J(s)_{s \bY}$. Note that $J(s) s \subset Z$ for each $s \in Z$ because $J(s) \leq G_{\bU}$ and $G_{\bU} Z = Z$. The images of the open cosets $\{J(s)s : s \in G\}$ in $Z / G_{\bY}$ form an open covering, which has a finite subcovering since $Z/G_{\bY}$ is compact.  Choose $s_1,\cdots,s_n \in G$ such that $G = \bigcup_{i=1}^n J(s_i) s_i G_{\bY}$ and let $J := \bigcap_{i=1}^n J(s_i)$. Then $J$ is still a $\bU$-small compact open subgroup of $G$. 

Suppose now that $s \in G$ is such that $s \bY \cap \bU \neq \emptyset$; it will be enough to show that $J_{s \bY \cap \bU} \leq G_{s \bY}$. Then $s \in Z$, so $s \in J(s_i) s_i G_{\bY}$ for some $i = 1,\ldots, n$. Write $s = gs_ih$ for some $g \in J(s_i)$ and $h \in G_{\bY}$; since $s\bY = gs_i\bY$ we may assume that $h = 1$. Suppose $x \in J_{s \bY \cap \bU}$. Now $g \in J(s_i) \leq G_{\bU}$ implies that $s \bY \cap \bU = g s_i \bY \cap \bU = g(s_i \bY \cap \bU)$, and hence $g^{-1}xg$ preserves $s_i \bY \cap \bU$. But $x \in J$ and $g$ both lie in $J(s_i)$, so $g^{-1}x g \in J(s_i)_{s_i \bY \cap \bU} = J(s_i)_{s_i \bY}$ by construction. Hence $g^{-1}x g$ preserves $s_i \bY$ and therefore $x$ preserves $gs_i \bY = s \bY$ as required.
\end{proof}

\begin{thm}\label{CounitIsom}Suppose that $(\bX, \bY, G)$ satisfies the LSC, and that the $G$-orbit of $\bY$ is regular in $\bX$.  Then there is a natural isomorphism 
\[\epsilon : \ind_{G_{\bY}}^G \hsp \circ \hsp \cH^0_{\bY}  \stackrel{\cong}{\longrightarrow} 1_{\cC_{\bX/G}^{G\bY}}.\]
\end{thm}
\begin{proof} Let $\cM \in \cC^{G\bY}_{\bX/G}$ and let $\cN := \cH^0_{\bY}(\cM)$, an object in $\cC_{\bX/G_{\bY}}^{\bY}$ by Theorem \ref{SheafyLocCohCoadm}. Fix $\bU \in \bX_w(\cT)$ and choose a $\bU$-small compact open subgroup $J$ of $G$ satisfying the conclusion of Proposition \ref{LSCb}. Then $\ind_{G_{\bY}}^G(\cN)(\bU)$ is a coadmissible $\w\cD(\bU,J)$-module by \cite[Proposition 3.6.10, Theorem 4.4.3]{EqDCap} and Theorem \ref{IndBGCoadm}. We will first construct a continuous $\cD(\bU)$-linear isomorphism
\[ \epsilon_{\cM}(\bU) : \ind_{G_{\bY}}^G(\cN)(\bU) \congs \cM(\bU).\]
Let $s \in G$. Then $\cH^0_{J (s \bY \cap \bU)}(\cM_{|\bU}) \in \cC_{\bU/J}^{J(s\bY \cap \bU)}$ by Lemma \ref{MsupponGY}(c), so Theorem \ref{LocalCohoCoadm}(b) gives us a $\w\cD(\bU,J)$-linear isomorphism
\[ \w\cD(\bU,J) \WO{\w\cD(\bU, J_{s \bY \cap \bU})} H^0_{s\bY \cap \bU}(\cM_{|\bU}) \tocong H^0_{J(s \bY \cap \bU)}(\cM_{|\bU}).\]
Because of our choice of $J$, Proposition \ref{LSCb} tells us that if $s \bY \cap \bU \neq \emptyset$, then
\[J_{s\bY \cap \bU} = J_{s \bY} = {}^sG_{\bY} \cap J. \]
 Also, $\cH^0_{\bY}(\cM)(s^{-1}\bU) = H^0_{s^{-1}\bU \cap \bY}(\cM_{|s^{-1}\bU})$, and there is a natural $\w\cD(\bU, {}^sG_{\bY} \cap J)$-linear isomorphism $[s] H^0_{s^{-1}\bU \cap \bY}(\cM_{|s^{-1}\bU}) \tocong H^0_{s\bY \cap \bU}(\cM_{|\bU})$ induced by $s^{\cM}(s^{-1}\bU) : \cM(s^{-1}\bU) \to \cM(\bU)$.  Putting all these facts together gives a natural $\w\cD(\bU,J)$-linear isomorphism for each\footnote{If $s \bY \cap \bU = \emptyset$ then both terms are zero, and we define $\epsilon_{\cM}^{J,s}$ to be the zero map.} $s \in G$:
\[ \epsilon_{\cM}^{J,s}(\bU) \quad :\quad  \w\cD(\bU,J) \WO{\w\cD(\bU, {}^sG_{\bY} \cap J)} [s] \cH^0_{\bY}(\cM)(s^{-1}\bU) \tocong H^0_{J(s\bY\cap \bU)}(\cM_{|\bU}).\]
This map is given is by $a \w\otimes [s] m \mapsto a \cdot s^{\cM}(m)$ for $a \in \w\cD(\bU,J), m \in H^0_{s^{-1}\bY\cap \bU}(\cM_{s^{-1}\bU})$. Using this description, we can verify that if $t = h s b^{-1}$ for some $h \in J$ and $b \in G_{\bY}$, then  $\epsilon^{J,s}_{\cM}(\bU) = \epsilon^{J,t}_{\cM}(\bU) \circ \eta^J_{(h,b)}$ where $\eta^J_{(h,b)}$ are the maps appearing in the proof of Proposition \ref{MUHfunctor}. Thus we can take the limit of these maps to obtain an $\w\cD(\bU,J)$-linear isomorphism
\[ \epsilon^{J,Z}_{\cM}(\bU) \quad : \quad \lim\limits_{s\in Z} \w\cD(\bU,J) \WO{\w\cD(\bU, {}^sG_{\bY} \cap J)} [s] \cH^0_{\bY}(\cM)(s^{-1}\bU) \tocong H^0_{J(Z\bY\cap \bU)}(\cM_{|\bU})\]
for every $J, G_{\bY}$-double coset $Z$ in $G$. Taking the direct sum over all double cosets in $J \backslash G / G_{\bY}$ and using Lemma \ref{MsupponGY}(b) gives a $\w\cD(\bU,J)$-linear isomorphism
\[ \epsilon^J_{\cM}(\bU) : M(U,J) \congs \cM(\bU)\]
using the notation from Definition \ref{DefnOfMUH}(b). Finally, let $J'$ be an open subgroup of $J$, and recall the $\w\cD(\bU,J')$-linear isomorphism $\alpha^J_{J'} : M(U,J') \to M(U,J)$ constructed in Proposition \ref{NUHtoJ}. Once we show that 
\begin{equation}\label{EpsCompat} \epsilon^{J}_{\cM}(\bU) \circ \alpha^J_{J'} = \epsilon^{J'}_{\cM}(\bU),\end{equation} 
it will follow that the maps $\epsilon^J_{\cM}$ assemble correctly to give the required continuous $\cD(\bU)$-linear isomorphism
\[ \epsilon_{\cM}(\bU) := \lim_{J} \epsilon^J_{\cM}(\bU) : \ind_{G_{\bY}}^G(\cN)(\bU) \congs \cM(\bU).\]
To see that $(\ref{EpsCompat})$ holds, after unwinding the definitions we find that it is enough to check that for each $y \in J$ and $s \in G$, the following diagram is commutative:
\[ \xymatrix{ \w\cD(\bU,J') \WO{\w\cD(\bU, {}^{ys}G_{\bY} \cap J')} [ys] \cN( (ys)^{-1}\bU) \ar[drr]^(0.6){\epsilon^{J', ys}_{\cM}(\bU)} \ar[d]_{\alpha^y} && \\ \w\cD(\bU,J) \WO{\w\cD(\bU, {}^{s}G_{\bY} \cap J)} [s] \cN( s^{-1}\bU) \ar[rr]_(0.6){\epsilon^{J, s}_{\cM}(\bU)} && \cM(\bU)}\]
Here the vertical arrow $\alpha^y$ comes from the proofs of Propositions \ref{NUHtoJ} and \ref{Mackey} and is given by $\alpha^y( a \w\otimes [ys] m) = a \gamma(y) \w\otimes [s]m$ for every $a \in \w\cD(\bU,J')$ and $m \in \cN((ys)^{-1}\bU)$, because $y^{-1}\bU = \bU$ since $y \in J \leq G_{\bU}$. We can compute that
\[ \epsilon^{J,s}_{\cM}(\bU)(\alpha^y(a \w\otimes [ys] m)) = a \gamma(y) \cdot s^{\cM}(m) = a \cdot (ys)^{\cM}(m) =\epsilon^{J',ys}_{\cM}(\bU)(a \w\otimes [ys] m)\]
which shows that the diagram commutes.

To prove that $\epsilon_{\cM} : \ind_{G_{\bY}}^G(\cH^0_{\bY}(\cM)) \to \cM$ is an isomorphism in $\cC_{\bX/G}$, it remains to show that $\epsilon_{\cM}$ commutes with restriction maps in the sheaves $\ind_{G_{\bY}}^G(\cH^0_{\bY}(\cM))$ and $\cM$, and also that it commutes with the $G$-equivariant structure on both of these sheaves. To complete the proof, we must also check that $\epsilon_{\cM}$ is functorial in $\cM$. We leave these three straightforward verifications to the reader.
\end{proof}

\begin{cor}\label{TechIndEq}  Let $G$ be a $p$-adic Lie group acting continuously on the smooth rigid analytic space $\bX$, and let $\bY$ be a Zariski closed subset of $\bX$. Suppose that
\be 
\item the stabiliser $G_{\bY}$ of $\bY$ is co-compact in $G$,
\item the $G$-orbit of $\bY$ is regular in $\bX$, and
\item $(\bX, \bY, G)$ satisfies the LSC.
\ee Then the functors
\[ \cH^0_{\bY} : \cC_{\bX/G}^{G \bY} \to \cC_{\bX / G_{\bY}}^{\bY} \qmb{and} \ind_{G_{\bY}}^G :  \cC_{\bX / G_{\bY}}^{\bY} \to \cC_{\bX/G}^{G \bY} \]
are mutually inverse equivalences of categories.
\end{cor}
\begin{proof} The functor $\cH^0_{\bY}$ sends $\cC_{\bX/G}^{G\bY}$ into $\cC_{\bX/G_{\bY}}^{\bY}$ by Theorem \ref{SheafyLocCohCoadm}, and the functor $\ind_{G_{\bY}}^G$ sends $\cC_{\bX / G_{\bY}}^{\bY}$ into $\cC_{\bX/G}^{G \bY}$ by Lemma \ref{IndSupp}. The result now follows from Proposition \ref{UnitIsom} and Theorem \ref{CounitIsom}.
\end{proof}
It seems to the author that the three conditions in Corollary \ref{TechIndEq} are close to being necessary in order for this equivalence to hold. Here is an example which shows that the LSC condition from Definition \ref{LSC} is not vacuous.

\begin{ex}\label{NoLSC} Let $\bX = \Sp K\langle x,y\rangle$, let $\bY := V(xy)$ and $G = \overline{\langle g \rangle} \cong \Zp$ be a pro-$p$-cyclic group acting on $\bX$ by the rule
\[ g^\lambda \cdot (a,b) = (a + p^2 \lambda, b) \qmb{for all} \lambda \in \Zp.\]
Then $(\bX,\bY,G)$ does \emph{not} satisfy the LSC.
\end{ex}
\begin{proof} Let $\bU := X(p/x)$; this is an affinoid subdomain of $\bX$. The group $G$ preserves the locus $\{(a,b) \in \bX : |a| < |p|\}$ and therefore it also preserves $\bU$. We can compute that $\bY \cap \bU = \Sp K \langle x , y , p / x \rangle / (xy)$ is the locus $\{(a,0) \in \bU \}$, and this is again preserved by $G$. Therefore $J_{\bY \cap \bU} = J$ for any open subgroup $J$ of $G$. However if $g^\lambda \in J_{\bY}$ for some $\lambda \in \Zp$, then since $(0,1) \in \bY$, we must have $g^\lambda(0,1) = (p^2\lambda,1) \in \bY$, which is only possible if $\lambda = 0$. Thus $J_{\bY}$ is trivial and never equal to $J_{\bY \cap \bU}$.
\end{proof}
Note that the Zariski closed subset $\bY$ is connected in this example, so we need to impose a stronger condition on $\bY$ to in order to be able to show that the LSC holds. Recall the notion of \emph{irreducible} rigid analytic spaces from \cite[Definition 2.2.2]{ConradIrreducible}, and note that the subset $\bY$ in Example \ref{NoLSC} is \emph{not} irreducible. In the remainder of $\S \ref{PfIndEq}$, all rigid analytic varieties are understood to be reduced.

\begin{lem}\label{Irr} Let $\bZ$ be an irreducible rigid analytic variety.
\be
\item Every non-empty connected affinoid subdomain of $\bZ$ is irreducible.
\item Suppose that $\bZ$ is affinoid. Then $\cO(\bZ)$ is an integral domain.
\ee \end{lem}
\begin{proof} (a) Suppose for a contradiction that $\bV$ is an affinoid subdomain of $\bZ$ which is not irreducible. Let $A$ be the Noetherian algebra $\cO(\bV)$ and let $P_1, \cdots, P_n$ be the distinct minimal primes of $A$.  It follows from \cite[Lemma 2.2.3]{ConradIrreducible} that $n \geq 2$. Now $P_1 \cap \cdots \cap P_n$ is the zero ideal in $A$ because $\bV$ is reduced. Let $P := P_1$ and $Q := P_2 \cap \cdots \cap P_n$; then one checks easily that $\Ann_A(P) = Q$ and $\Ann_A(Q) = P$. Let $\bV_1$ and $\bV_2$ be the Zariski closed subsets of $\bX$ defined by the vanishing of the ideals $P$ and $Q$, respectively. Because $\bV$ is connected, the intersection $\bV_1 \cap \bV_2$ is non-empty. Take $x \in \bV_1 \cap \bV_2$ and let $\mathfrak{m}$ be its maximal ideal in $A$; we will show that the local ring $A_{\mathfrak{m}}$ is not a domain. If the localisation $P_{\mathfrak{m}}$ is zero, then we can find $s \in A \backslash \mathfrak{m}$ such that $P \cdot s = 0$ and then $s \in \Ann_A(P) = Q$. Since $Q$ vanishes on $\bV_2 \ni x$, $Q$ is contained in $\mathfrak{m}$ which implies $s \in \mathfrak{m}$, a contradiction. Similarly, the ideal $Q_{\mathfrak{m}}$ of $A_{\mathfrak{m}}$ is non-zero because $\Ann_A(Q) = P$. Now $P_{\mathfrak{m}} \cdot Q_{\mathfrak{m}} = (PQ)_{\mathfrak{m}}$ is zero because $PQ \subseteq P \cap Q = P_1 \cap \cdots \cap P_n = 0$, so $A_{\mathfrak{m}}$ has zero-divisors as claimed.

By \cite[Proposition 7.3.2/3]{BGR}, there is an injective map $A_{\mathfrak{m}} \hookrightarrow \cO_{\bV, x}$, so $\cO_{\bV,x}$ is not a domain. Because $\bV$ is an affinoid subdomain of $\bZ$, the restriction map $\cO_{\bZ,x} \to \cO_{\bV,x}$ is an isomorphism. However, it was observed in the paragraph following \cite[Definition 2.2.2]{ConradIrreducible} that \cite[Lemma 2.1.1]{ConradIrreducible} implies that $\Spec(\cO_{\bZ,x})$ is irreducible if and only if there is a unique irreducible component of $\bZ$ passing through $x$. Since $\bZ$ is reduced and irreducible by assumption, this implies that $\cO_{\bZ,x}$ \emph{is} an integral domain, which gives the required contradiction.

(b) This follows from \cite[Lemma 2.2.3]{ConradIrreducible}.
\end{proof}

\begin{lem}\label{AffinoidStabs} Suppose that $\bX$ is affinoid and that $\bY$ is irreducible. Let $\bU$ be an affinoid subdomain of $\bX$ such that $\bY \cap \bU \neq \emptyset$. Then 
\[ G_{\bY \cap \bU} \hsp \cap \hsp G_{\bU} \leq G_{\bY}.\]
\end{lem}
\begin{proof} Since $\bY$ is Zariski closed in the affinoid variety $\bX$, it is itself affinoid. Since it is also irreducible, its coordinate ring $\cO(\bY)$ is an integral domain by Lemma \ref{Irr}(b). Since $\bY \cap \bU$ is a non-empty affinoid subdomain of $\bY$, the proof of \cite[Proposition 4.2]{ABB} now shows that the restriction map $\cO(\bY) \to \cO(\bY \cap \bU)$ is injective.

Let $\cI$ be the coherent subsheaf of $\cO_{\bX}$ consisting of rigid analytic functions vanishing on $\bY$; it follows that the kernel of the restriction map $\cO(\bX) \to \cO(\bY \cap \bU)$ is equal to $\cI(\bX)$. Now if $g \in G_{\bY \cap \bU} \cap G_{\bU}$ then the action of $g$ on $\cO(\bX)$ must preserve this kernel. Hence $g$ preserves $\cI(\bX)$ and therefore also $\bY$.
\end{proof}

We can now give some sufficient conditions on $(\bX,\bY,G)$ that ensure that the LSC is satisfied.
\begin{prop}\label{LocalStabs} Suppose that the following conditions hold:
\be
\item $\bX$ is separated,
\item $\bY$ is irreducible,
\item there is an admissible affinoid covering $\{\bX_i\}$ of $\bX$ such that 
\begin{enumerate}[{(}i{)}]
\item $\bY_i := \bY \cap \bX_i$ is connected for each $i$, and
\item $\bigcap\limits_{\bY_i \neq \emptyset} G_{\bX_i} $ is open in $G$.
\end{enumerate}
\ee
Then $(\bX,\bY,G)$ satisfies the LSC.
\end{prop}
\begin{proof} Note that the assumptions (b) and (c)(i) together with Lemma \ref{Irr}(a) imply that every non-empty $\bY_i$ is irreducible.

Let $\bU \in \bX_w(\cT)$ be such that $\bY \cap \bU \neq \emptyset$. Choose any $\bU$-small open subgroup $H$ of $G$, and define
\[ J := H \cap \bigcap\limits_{\bY_i \neq \emptyset} G_{\bX_i} .\]
Assumption (c)(ii) ensures that $J$ is still an open subgroup of $G$; it is therefore also $\bU$-small by \cite[Lemma 3.4.5]{EqDCap}. With this choice of $J$, every member $\bX_i$ of our covering for which $\bY_i \neq \emptyset$ is $J$-stable.

We have to show that $g \bY \subseteq \bY$ for every $g \in J_{\bY \cap \bU}$. Write $\bU_i := \bU \cap \bX_i$ for each $i$; this is an affinoid subdomain of $\bX_i$ because $\bX$ is assumed to be separated. Since $\bY \cap \bU$ is non-empty by assumption, there is some index $i$ such that $\bY_i \cap \bU_i \neq \emptyset$. Since $\bX_i$ is $J$-stable, $g$ also preserves $\bY_i \cap \bU_i = \bY \cap \bU \cap \bX_i$. Since $\bY_i$ is irreducible, $g$ preserves $\bY_i$ by Lemma \ref{AffinoidStabs} applied to $\bX_i, \bY_i$ and $\bU_i$.

Because $\bY$ is assumed to be irreducible, it is connected. Since $\bY = \cup_i \bY_i$, it will now suffice to show whenever $g\bY_j \subseteq \bY_j$ and $\bY_j \cap \bY_k \neq \emptyset$, we also have $g\bY_k \subseteq \bY_k$. Now $\bY_k$ is also irreducible and $\bX_j \cap \bX_k$ is an affinoid subdomain of $\bX_k$ because $\bX$ is separated. We may apply Lemma \ref{AffinoidStabs} to $\bX_k$, $\bY_k$ and $\bX_j \cap \bX_k$ to deduce that $g$ preserves $\bY_k$, because $\bY_k \cap (\bX_j \cap \bX_k) = \bY_j \cap \bY_k \neq \emptyset$ by assumption.\end{proof}

Before we can give the proof of Theorem \ref{InductionEquivalence}, we need two results from rigid analysis that are probably well-known but for which we could not find a reference.

\begin{lem}\label{FineCover} Suppose that $\bX$ is affinoid. Then there is a finite admissible affinoid covering $\{\bX_i\}$ of $\bX$ such that each $\bY \cap \bX_i$ is connected.
\end{lem}
\begin{proof} Since $\bY$ is Zariski closed in the affinoid variety $\bX$, it is also affinoid, and therefore admits a decomposition $\bY = \bY_1 \cup \cdots \cup \bY_n$ into pairwise disjoint connected components. By induction on $n$, we will now construct an affinoid covering $\{\bX_i\}$ of $\bX$ such that $\bY \cap \bX_i = \bY_i$ for each $i$. 

Fix $0 \neq \pi \in K$ such that $|\pi| < 1$. Let $e \in \cO(\bX)$ be such that its image $\overline{e}$ in $\cO(\bY)$ is an idempotent which is $0$ on $\bY_1$ and $1$ on $\bY \backslash \bY_1$. Then $\bX = \bX(e/\pi) \cup \bX(\pi/e)$ is an affinoid covering of $\bX$, $\bX(e/\pi) \cap \bY = \bY(\overline{e}/\pi) = \bY_1$ and $\bX(\pi/e) \cap \bY = \bY \backslash \bY_1$. By induction, we can find an affinoid covering $\{\bX_2,\ldots,\bX_n\}$ of $\bX(\pi/e)$ such that $\bX_i \cap \bY = \bY_i$ for each $i \geq 2$. Letting $\bX_1 := \bX(e/\pi)$ completes the induction.
\end{proof}

\begin{lem}\label{QCY} Suppose that the Zariski closed subspace $\bY$ of $\bX$ is quasi-compact. Then there is a quasi-compact admissible open subspace $\bU$ of $\bX$ containing $\bY$ such that $\{\bU, \bX \backslash \bY\}$ is an admissible open covering of $\bX$.
\end{lem}
\begin{proof} By considering the restriction of an admissible affinoid covering of $\bX$ to the quasi-compact subspace $\bY$, we see that $\bY$ is contained in a finite union $\bU$ of affinoid subdomains of $\bX$. Certainly $\bX \backslash \bY$ is an admissible open subset of $\bX$; this follows from \cite[Corollary 9.1.4/7]{BGR}. We must show that $\{\bU, \bX \backslash \bY\}$ is an admissible covering. Let $\bX'$ be an affinoid subdomain of $\bX$, let $\bU' := \bX' \cap \bU$ and $\bY' := \bX' \cap \bY$; it will be enough to show that $\{ \bU', \bX' \backslash \bY'\}$ is an admissible covering of $\bX'$. Now $\bU'$ is an admissible open subspace of $\bX'$ containing the Zariski closed subspace $\bY'$ of $\bX'$; in this situation, \cite[Lemma 2.3]{Kisin99} implies that the covering $\{\bU', \bX' \backslash \bY'\}$ of $\bX'$ admits a finite affinoid refinement and is therefore admissible.
\end{proof}

Combining these two results we obtain the following sufficient condition.

\begin{cor}\label{EasyLSC} $(\bX, \bY, G)$ satisfies the LSC whenever $\bX$ separated, and $\bY$ is irreducible and quasi-compact. 
\end{cor}
\begin{proof} It is enough to verify condition (c) of Proposition \ref{LocalStabs}.

By Lemma \ref{QCY}, we can find a quasi-compact admissible open subspace $\bU$ of $\bX$ containing $\bY$ such that $\{\bU, \bX \backslash \bY\}$ is an admissible covering of $\bX$. Choose a finite admissible affinoid covering $\{\bU_1,\cdots,\bU_n\}$ of $\bU$; by Lemma \ref{FineCover}, we may replace this covering by a finite affinoid refinement and thereby assume that each $\bU_i \cap \bY$ is connected. Choose an admissible affinoid covering $\{\bV_j\}$ of $\bX \backslash \bY$ so that $\{\bU_1,\ldots,\bU_n\} \cup \{\bV_j\}$ is an admissible affinoid covering of $\bX$. Then the intersection of $\bY$ with every member of this covering is connected (and possibly empty) by construction, showing that this covering satisfies condition (c)(i). Because only finitely many intersections $\bU_i \cap \bY$ are possibly non-empty and because each $G_{\bU_i}$ is open in $G$ by \cite[Definition 3.1.8]{EqDCap}, we conclude that this covering also satisfies condition (c)(ii).\end{proof}

\begin{proof}[Proof of Theorem \ref{InductionEquivalence}]
Combine Corollary \ref{TechIndEq} with Corollary \ref{EasyLSC}.
\end{proof}
In practice, the quasi-compactness condition on $\bY$ is still too strong to be necessary, and we can establish the LSC in certain other cases, as follows.
\begin{lem} Suppose that $(\bX, \bY, G)$ satisfies the LSC and that $\bX'$ is a $G'$-stable admissible open subset of $\bX$ for some open subgroup $G'$ of $G$. If $\bY' := \bX' \cap \bY$, then $(\bX', \bY', G')$ also satisfies the LSC.
\end{lem}
\begin{proof} Let $\bU \in \bX_w'(\cT)$ be such that $\bU \cap \bY' \neq \emptyset$. Then also $\bU \in \bX_w(\cT)$ because $\bX'$ is an admissible open subset of $\bX$, so because $(\bX,\bY,G)$ satisfies the LSC, we can find a $\bU$-small compact open subgroup $J$ of $G$ such that $J_{\bY \cap \bU} = J_{\bY}$. Since $G'$ is open in $G$ by assumption, $J' := J \cap G'$ is open in $J$. Then $J'$ is $\bU$-small by \cite[Lemma 3.4.5]{EqDCap}, and it will be enough to show that $J'_{\bY' \cap \bU} \leq J'_{\bY'}$. Let $g \in J'_{\bY' \cap \bU}$; then $g \in J' \leq G'$ stabilises $\bX'$ by assumption, and $g$ stabilises $\bY' \cap \bU = \bY \cap \bX' \cap \bU = \bY \cap \bU$ since $\bU \subseteq \bX'$. So $g$ also stabilises $\bX' \cap \bY = \bY'$.
\end{proof}

\section{The equivariant Kashiwara equivalence \ts{\cC_{\bX/G}^{\bY}\cong \cC_{\bY/G}}}

\subsection{Side-switching operations}\label{LRSwitch}
Here we extend the material in \cite[\S 3]{DCapTwo} and show that there is a canonical equivalence of categories between the categories of coadmissible $G$-equivariant left $\cD$-modules and of coadmissible $G$-equivariant right $\cD$-modules on any smooth rigid analytic space $\bX$ equipped with a continuous $G$-action.

We begin with some algebraic generalities. Let $R \to A$ be a homomorphism of commutative rings, let $L$ be an $(R,A)$-Lie algebra, and let $G$ be a group. Suppose that we are given an action of $G$ on $L$ as defined at \cite[Definition 2.1.5]{EqDCap}. Then by \cite[Corollary 2.1.9]{EqDCap}, $G$ acts on $U(L)$ by $R$-algebra automorphisms and we may form the skew-group ring $U(L) \rtimes G$. 
\begin{lem}\label{OmegaL} Suppose that $L$ is a finitely generated and projective as an $A$-module of constant rank $d$. Then $\Omega_L := \Hom_A(\bigwedge^d_AL, A)$ a right $U(L) \rtimes G$-module.
\end{lem}
\begin{proof} Recall that for any $n\geq 0$, the \emph{Lie derivative} gives a left action $x \mapsto \Lie_x$ of the $R$-Lie algebra $L$ on the $R$-module $\Hom_A(\bigwedge^n_AL, A)$ of alternating $n$-forms on $L$ with values in $A$. This action is given explicitly as follows:
\begin{equation}{\label{LieDer}} \Lie_x(\phi)(v_1 \wedge \cdots \wedge v_n) = x\cdot \phi\left(v_1 \wedge \cdots \wedge v_n\right) - \phi\left( \sum_{i=1}^n v_1 \wedge \cdots \wedge [x,v_i] \wedge \cdots \wedge v_n \right)\end{equation}
whenever $\phi \in \Hom_A(\wedge^n_AL, A)$ and $x,v_1,\ldots,v_n \in L$. Now if in addition $L$ is finitely generated and projective as an $A$-module of constant rank $d$, then
\[ \omega \cdot a := a \omega \qmb{and} \omega \cdot x := - \Lie_x(\omega) \qmb{for all} a \in A, x \in L, \omega \in \Omega_L\]
defines a right $U(L)$-module structure on $\Omega_L$. See \cite[Proposition 2.8]{Hue99} for more details. On the other hand, there is a right $R[G]$-module structure on $\Omega_L$ given by 
\[ (\omega \cdot g)(v_1 \wedge \cdots \wedge v_d) := g^{-1} \cdot \omega(g\cdot v_1 \wedge \cdots \wedge g \cdot v_d)\]
where $\omega \in \Omega_L, g \in G$ and $v_1,\ldots, v_d \in L$. A straightforward calculation shows that
\[ ((\omega \cdot g) \cdot x) \cdot g^{-1} = \omega \cdot (g \cdot x) \qmb{and} ((\omega \cdot g) \cdot a) \cdot g^{-1} = \omega \cdot (g \cdot a) \]
hold whenever $\omega \in \Omega_L, g \in G, x \in L$ and $a \in A$. These formulas mean that the right actions of $U(L)$ and $R[G]$ on $\Omega_L$ combine into a right $U(L)\rtimes G$-action. 
\end{proof}
We will assume from now on that $L$ is a finitely generated projective $A$-module of constant rank $d$, and write $S := U(L) \rtimes G$ for brevity. We can now use $\Omega_L$ to convert left $S$-modules into right $S$-modules, and vice versa, as follows.

\begin{lem}\label{LRswitch}Let $M$ be a left $S$-module, and let $N$ be a right $S$-module. 
\be \item $\Omega_L \otimes_A M$ is a right $S$-module, with the right actions of $L$ and $G$ given by
\begin{equation}\label{SactionLM} (\omega \otimes m) \cdot x = \omega x \otimes m - \omega \otimes x m \qmb{and} (\omega \otimes m) \cdot g = \omega g \otimes g^{-1}  m\end{equation}
for all $\omega \in \Omega_L, m \in M, x \in L$ and $g \in G$.
\item $\Hom_A(\Omega_L, N)$ is a left $S$-module, with the left actions of $L$ and $G$ given by
\begin{equation}\label{SactionHom} (x \cdot \phi)(\omega) = \phi(\omega x) - \phi(\omega) x \qmb{and} (g \cdot \phi)(\omega) = \phi(\omega g) g^{-1}\end{equation}
for all $\phi \in \Hom_A(\Omega_L,N), \omega \in \Omega_L, x \in L$ and $g \in G$.
\item The canonical $A$-linear evaluation and co-evaluation maps, namely \newline $\Omega_L \otimes_A \Hom_A(\Omega_L,N) \to N$ and  $M \to \Hom_A(\Omega_L, \Omega_L \otimes_A M)$, are $S$-linear.
\ee\end{lem}
\begin{proof} Note that the actions of $U(L)$ agree with the standard ones from \cite[\S 3.1, (6) and (5)]{DCapTwo}; note also that $\Omega_L$ is a projective $A$-module of constant rank $1$, i.e. an invertible $A$-module, so the two maps appearing in part (c) are in fact isomorphisms. We have to show that formulas $(\ref{SactionLM})$ and $(\ref{SactionHom})$ define actions of the $R$-Lie algebra $L$ and the group $G$ on $\Omega_L \otimes_A M$ and $\Hom_A(\Omega_L, N)$, and that these actions are compatible: whenever $\omega \in \Omega_L, m \in M, \phi \in \Hom_A(\Omega_L, N), x \in L$ and $g \in G$, 
\[ (( (\omega \otimes m) \cdot g) \cdot x) \cdot g^{-1} = (\omega \otimes m) \cdot (g \cdot x) \qmb{and} g \cdot (x \cdot (g^{-1} \cdot \phi)) = (g \cdot x) \cdot \phi.\]
These are straightforward verifications so we omit the details. 
\end{proof}

Armed with this result, we immediately obtain the following

\begin{cor}\label{ULGswitch} The functors $M \mapsto \Omega_L \otimes_A M$ and $N \mapsto \Hom_A(\Omega_L, N)$ are mutually inverse exact equivalences of categories between the category of left $S$-modules and the category of right $S$-modules.
\end{cor}

We will next need to know what these side-switching functors do to finitely presented $S$-modules; to that end, we study $\Omega_L \otimes_A S$ closely. There are two natural ways to view $\Omega_L \otimes_A S$ as a right $S$-module: one of these comes from the left action of $S$ on itself and Lemma \ref{LRswitch}(a); the other, which we'll denote by $\circ$, comes from the right action of $S$ on itself. Following \cite[\S 3.2]{DCapTwo}, we will write $\Omega_L \oslash_A S$ to denote the right $A$-module $\Omega_L \otimes_A S$ equipped with the $\circ$-action of $S$ on the right. 

\begin{lem}\label{Exch} There is an $S$-linear isomorphism $\alpha : \Omega_L\oslash_A S \tocong \Omega_L\otimes_A S$ given by $\alpha( \omega \otimes s ) = (\omega \otimes 1)s$ for all $\omega \in \Omega_L, s \in S$, which satisfies $\alpha^2 = 1$.
\end{lem}
\begin{proof} The canonical filtration on $U(L)$ is $G$-stable. It follows that there is a natural exhaustive positive filtration $F_\bullet$ on $S$ such that $F_0 = A \rtimes G$ and $F_n = L \cdot F_{n-1}+F_{n-1}$ for all $n \geq 1$. Let $\omega \in \Omega_L, s \in S, a \in A$ and $g \in G$. Using (\ref{SactionLM}) we see that $\alpha( (\omega \otimes s) (ag)) = \alpha ( (\omega a \otimes s) \cdot g) = \alpha( (\omega a) g \otimes g^{-1}s ) =  ( (\omega a) g \otimes 1) \cdot g^{-1} s = (\omega a \otimes g) \cdot s = (\omega \otimes ag)s.$ Now the proof of \cite[Lemma 3.2]{DCapTwo} works with straightforward modifications.
\end{proof}

\begin{cor}\label{FPswitch} The functors appearing in Corollary \ref{ULGswitch} preserve the categories of finitely presented $S$-modules.
\end{cor}
\begin{proof} The cokernel of a morphism between two finitely generated projective modules is finitely presented. Because $F := \Omega_L \otimes_A -$ and $G := \Hom_A(\Omega_L,-)$ are exact, it will be enough to verify that $F(S)$ and $G(S)$ are finitely generated and projective. 

Since $\Omega_L$ is a finitely generated projective $A$-module, $\Omega_L \oslash_A S$ is a finitely generated projective right $S$-module. Hence the right $S$-module $F(S) = \Omega_L \otimes_A S$ is finitely generated projective by Lemma \ref{Exch}. 

Since $\Hom_A(\Omega_L,A)$ is a finitely generated projective $A$-module, it is a direct summand of a finitely generated free $A$-module. Tensoring by $\Omega_L$ and using the fact that $\Omega_L$ is an invertible $A$-module, we can find an $A$-module $Q$ and an integer $m > 0$ such that $A \oplus Q \cong \Omega_L^m$. Hence $S \oplus (Q \oslash_A S) \cong (\Omega_L \oslash_A S)^m$ and therefore $G(S)$ is a direct summand of $G(\Omega_L \oslash_A S)^m$. Since $F$ and $G$ are mutually inverse and $F(S) \cong \Omega_L \oslash_A S$ by Lemma \ref{Exch}, we see that $G(S)$ is a direct summand of $S^m$. Hence the left $S$-module $G(S)$ is finitely generated and projective. \end{proof}

We now return to rigid analytic geometry, and suppose that $(\bX, G)$ is small. Thus the group $G$ is compact $p$-adic analytic and we can find a $G$-stable affine formal model $\cA$ in $\cO(\bX)$ as well as a $G$-stable free $\cA$-Lie lattice $\cL$ in $\cT(\bX)$, of rank $d$ say. Let $\Omega_{\cL}:= \Hom_{\cA}(\bigwedge^d\cL,\cA)$. This is a free $\cA$-module of rank $1$ as well as a right $U(\cL) \rtimes G$-module, by Lemma \ref{OmegaL} above. Recall also the open normal subgroup $G_{\cL} = \rho^{-1}(\exp(p^\epsilon \cL))$ of $G$ from \cite[Definition 3.2.11]{EqDCap}, where $\rho : G \to \Aut_K \cO(\bX)$ is the action of $G$ on $\cO(\bX)$, and the crossed product $\h{U(\cL)} \rtimes_N G$ from \cite[\S 3.2]{EqDCap}.

\begin{lem}\label{OmegaUHGmod} Let $H$ be an open normal subgroup of $G$ contained in $G_{\cL}$. Then $\Omega_{\cL}$ is a right $\h{U(\cL)} \rtimes_H G$-module.
\end{lem}
\begin{proof} Because $\Omega_{\cL}$ is $\pi$-adically complete, it follows from Lemma \ref{OmegaL} that it is naturally a right module over the $\pi$-adic completion of $U(\cL) \rtimes G$. The canonical map from $\h{U(\cL)}$ to this $\pi$-adic completion factors through $\h{U(\cL)} \rtimes G$, hence $\Omega_{\cL}$ is also a right $\h{U(\cL)} \rtimes G$-module. It remains to show that this action of $\h{U(\cL)} \rtimes G$ factors through $\h{U(\cL)} \rtimes_H G$. To that end, let $g \in H$ and let $x \in \cL$ be such that $\rho(g) = \exp( p^\epsilon x )$; we must show that $\exp(p^\epsilon \iota(x)) - g$ kills $\Omega_{\cL}$ on the right, where $\iota : \cL \to \h{U(\cL)}$ is the canonical map.   Let $m \geq 0$, $\omega \in \Omega_{\cL}$ and $v_1,\ldots, v_d \in \cL$; then it follows from (\ref{LieDer}) that $(\omega \cdot \iota(x)^m)(v_1\wedge \cdots \wedge v_d)$ is equal to
\[ \sum \binom{m}{i_0,i_1,i_2,\ldots,i_d} (-x)^{i_0} \cdot \omega\left( \ad(x)^{i_1}(v_1) \wedge \cdots \wedge \ad(x)^{i_d}(v_d) \right)\]
where the sum is taken over all $(d+1)$-tuples $(i_0,\ldots,i_d)$ of non-negative integers such that $i_0+ \cdots + i_d = m$. Now $\rho(g^{-1}) = \exp( -p^\epsilon x )$, and 
\[ \exp( p^\epsilon \ad(x) )(v) = \exp(p^\epsilon x) \hsp v \hsp \exp(-p^\epsilon x) = \rho(g) \hsp v \hsp \rho(g)^{-1} = g \cdot v\]
for any $v \in \cL$ by \cite[Exercise 6.12]{DDMS}. It follows from these facts that
\[\begin{array}{lll} & & (\omega \cdot \exp(p^\epsilon \iota(x))(v_1 \wedge \cdots \wedge v_d) = \\
&=& \sum\limits_{m=0}^\infty \frac{p^{\epsilon m}}{m!} \sum\limits_{\sum i_j = m} \binom{m}{i_0,i_1,i_2,\ldots,i_d} (-x)^{i_0} \cdot \omega\left( \ad(x)^{i_1}(v_1) \wedge \cdots \wedge \ad(x)^{i_d}(v_d) \right)  = \\
& = & \sum\limits_{i \in \N^{d+1}} \frac{(-p^\epsilon x)^{i_0}}{i_0!} \cdot \omega\left( \frac{ \ad(p^\epsilon x)^{i_1} }{i_1!}(v_1) \wedge \cdots \wedge \frac{ \ad(p^\epsilon x)^{i_d} }{i_d!}(v_d) \right)  = \\
&=& g^{-1} \cdot \omega( g \cdot v_1 \wedge \cdots \wedge g \cdot v_d ) = (\omega \cdot g)(v_1\wedge \cdots \wedge v_d).\end{array}\]
Thus $\Omega_{\cL} \cdot (\exp(p^\epsilon \iota(x)) - g) = 0$, as required.
\end{proof}

\begin{lem}\label{GplikeAction} Let $\cM$ be a $\pi$-adically complete $\h{U(\cL)}\rtimes G$-module. Then $\Omega_{\cL} \otimes_{\cA} \cM$ is a right $\h{U(\cL)} \rtimes G$-module, and whenever $x \in \cL, \omega \in \Omega_{\cL}$ and $m \in \cM$, we have
\[ (\omega \otimes m) \cdot \exp(p^\epsilon \iota(x)) = \omega \cdot \exp(p^\epsilon \iota(x))  \otimes  \exp(-p^{\epsilon} \iota(x)) \cdot m .\]
\end{lem}
\begin{proof} Note that $\Omega_{\cL} \otimes_{\cA} \cM$ is $\pi$-adically complete by \cite[Lemma 3.3]{DCapTwo}. Now the first statement follows from Lemma \ref{OmegaL} and Lemma \ref{LRswitch}(b). Using  (\ref{SactionLM}), we have
\[ \begin{array}{lll}(\omega \otimes m) \cdot \exp(p^\epsilon \iota(x)) &=& \sum\limits_{n= 0}^\infty \frac{p^{\epsilon n}}{n!} (\omega \otimes m) \cdot \iota(x)^n = \\
&=& \sum\limits_{n= 0}^\infty \frac{p^{\epsilon n}}{n!} \sum\limits_{j=0}^n \binom{n}{j} \omega \cdot \iota(x)^j \otimes (-\iota(x))^{n-j} \cdot m = \\
&=&\sum\limits_{j =0}^\infty \sum\limits_{i = 0}^\infty \omega \cdot \frac{(p^\epsilon \iota(x))^j}{j!} \otimes \frac{(-p^\epsilon \iota(x))^{i}}{i!} \cdot m = \\
&=&\omega \cdot \exp(p^\epsilon \iota(x))  \otimes  \exp(-p^{\epsilon} \iota(x)) \cdot m. \hfill \qedhere \end{array}\]
\end{proof}

Let $H$ be an open subgroup of $G$ contained in $G_{\cL}$, and write $\cS := \h{U(\cL)} \rtimes_HG$.

\begin{cor}\label{hUHGsideswitch}  Let $\cM$ be a $\pi$-adically complete $\cS$-module.
\be\item If $\cM$ is a left $\cS$-module, then $\Omega_{\cL} \otimes_{\cA} \cM$ is a right $\cS$-module.
\item If $\cN$ is a right $\cS$-module, then $\Hom_{\cA}(\Omega_{\cL}, \cN)$ is a left $\cS$-module.
\ee\end{cor}
\begin{proof} Let $g \in H$ and let $x \in \cL$ be such that $\rho(g) = \exp( p^\epsilon x )$. By Lemma \ref{OmegaUHGmod}, we have $\omega \cdot \exp(p^\epsilon \iota(x)) = \omega \cdot g$ for all  $\omega \in \Omega_{\cL}$. On the other hand, $\exp(p^\epsilon \iota(x)) \cdot m = g \cdot m$ for all $m \in \cM$ because $\cM$ is a left $\cS$-module. Using Lemma \ref{GplikeAction}, we have
\[ (\omega \otimes m) \cdot \exp(p^\epsilon \iota(x)) = \omega \cdot \exp(p^\epsilon \iota(x))  \otimes  \exp(-p^{\epsilon} \iota(x)) \cdot m = \omega \cdot g \otimes g^{-1} \cdot m = (\omega \otimes m) \cdot g.\]
Hence the $\h{U(\cL)} \rtimes G$-action on $\Omega_{\cL} \otimes_{\cA} \cM$ factors through $\cS$, and part (a) follows. Part (b) is proved in a similar manner.
\end{proof}

Recall the trivialisation $\beta : H \to \h{U(\cL)}^\times$ of the action of $H$ on $\h{U(\cL)}$ from \cite[Theorem 3.2.12]{EqDCap}. Our next technical result will help us to identify the image of $\cS = \h{U(\cL)} \rtimes_HG$ under the side-switching functor $\Omega_{\cL} \otimes_{\cA} -$ from Corollary \ref{hUHGsideswitch}(b).

\begin{prop}\label{Spres} Let $\overline{\cR}$ be any proper homomorphic image of $\cR$, write $A := \overline{\cR} \otimes_{\cR} \cA$,  let $L$ be the $(\overline{\cR}, A)$-Lie algebra $\overline{\cR} \otimes_{\cR} \cL$, and write $S  := U(L) \rtimes_H G$. Then there is a right $S$-linear isomorphism $\Omega_L \oslash_A S \cong \Omega_L \otimes_A S$.
\end{prop}
\begin{proof} Write $T := U(L) \rtimes G$, and let $I$ be the kernel of the canonical surjection $T \twoheadrightarrow S$. It follows from the proof of \cite[Lemma 2.2.4]{EqDCap} that $\{ \beta(g) g^{-1} - 1 : g \in H\}$ generates $I$ as a right $T$-module, and also that $\{g \beta(g)^{-1} - 1 : g \in H\}$ generates $I$ as a left $T$-module.  For each $g \in H$, let  $r_{\beta(g) g^{-1}}$ denote the left $T$-linear endomorphism of $T$ given by right-multiplication by $\beta(g) g^{-1}$, and let $\ell_{g \beta(g)^{-1}}$ denotes the right $T$-linear endomorphism of $T$ given by left-multiplication by $g \beta(g)^{-1}$. This gives us the following presentations of $S$ as a left, respectively, right, $T$-module:
\[ \bigoplus\limits_{g \in H} T \stackrel{ \sum\limits_{g \in H} r_{\beta(g) g^{-1} - 1} }{\longrightarrow} T \longrightarrow S \to 0, \quad \mbox{and} \quad
 \bigoplus\limits_{g \in H} T \stackrel{ \sum\limits_{g \in H} \ell_{g \beta(g)^{-1} - 1} }{\longrightarrow} T \longrightarrow S \to 0.\]
Because $\beta$ is a trivialisation of the $G$-action of $U(L)$, $g \beta(g)^{-1} - 1$ commutes with $U(L)$ inside $T$, and hence $\ell_{g \beta(g)^{-1} - 1}$ is also left $U(L)$-linear. Therefore the second presentation can be regarded as an exact sequence of $A-T$-bimodules. Now consider the following diagram:
\[ \xymatrix{  \bigoplus\limits_{g \in H} \Omega_L \oslash_A T \ar[rrrr]^{\sum\limits_{g \in H} 1 \oslash \ell_{g \beta(g)^{-1} - 1} }\ar[d]_\alpha & &&&  \Omega_L \oslash_AT \ar[r] \ar[d]^\alpha & \Omega_L \oslash_AS \ar[r] \ar@{.>}[d] & 0 \\ \bigoplus\limits_{g \in H} \Omega_L \otimes_A T \ar[rrrr]_{ \sum\limits_{g \in H} 1 \otimes r_{\beta(g) g^{-1} - 1} }  &&& & \Omega_L \otimes_A  T \ar[r]& \Omega_L \otimes_A S \ar[r] & 0}\]
Here the top row is obtained by applying the twisting functor $\Omega_A \oslash_A -$ to the presentation of $S$ as a right $T$-module, and the bottom row is obtained by applying the side-switching functor $\Omega_L \otimes_A -$ to the presentation of $S$ as a left $T$-module. We will show that the diagram commutes.

Let $g \in H$ and write $\beta(g) = \exp(p^\epsilon \iota(x))$ where $p^\epsilon x = \log \rho(g) \in p^\epsilon \cL$.  Let $\omega \in \Omega_L$ and $u \in T$, and define $z := \beta(g) g^{-1}$.  Then using Lemma \ref{GplikeAction}, equation (\ref{SactionLM}) and Lemma \ref{OmegaUHGmod} we see that 
\[ (\omega \otimes u) \cdot z = (\omega \cdot \beta(g)\otimes \beta(g)^{-1} u) \cdot g^{-1} = \omega \cdot \beta(g) g^{-1} \otimes g \beta(g)^{-1} u = \omega \otimes z^{-1}u.\]
Using this equation together with Lemma \ref{Exch}, we calculate
\[ \begin{array}{lll} \left((1 \otimes r_{z-1}) \circ \alpha\right)( \omega \otimes u) &=& \alpha(\omega \otimes u) \circ (z-1) = \alpha( (\omega \otimes u) \cdot (z-1)) = \\
&=& \alpha( \omega \otimes (z^{-1}-1) u) = \left(\alpha \circ (1 \otimes \ell_{z^{-1}-1})\right)( \omega \otimes u).\end{array}\]
Hence the diagram commutes. Because its rows are exact and because $\alpha$ is an isomorphism by Lemma \ref{Exch}, by the Five Lemma the map $\alpha$ descends to a right $T$-module isomorphism $\Omega_L \oslash_A S \cong \Omega_L \otimes_A S$, appearing as a dotted arrow in the diagram. Since the action of $T$ on these modules factors through $S$, it is in fact a right $S$-linear isomorphism.
\end{proof}

\begin{cor}\label{AlphaOnCurlyS} There is a right $\cS$-module isomorphism $\Omega_{\cL} \oslash_{\cA} \cS \cong \Omega_{\cL} \otimes_{\cA} \cS.$
\end{cor}
\begin{proof} Fix $n \geq 0$, and write $\cR_n := \cR / \pi^n \cR$, $\cA_n := \cA \otimes_{\cR} \cR_n$, $\cL_n := \cL \otimes_{\cR} \cR_n$ and $\cS_n = \cS \otimes_{\cR} \cR_n = U(\cL_n) \rtimes_HG$. Proposition \ref{Spres} gives an isomorphism $\alpha_n : \Omega_{\cL_n} \oslash_{\cA_n} \cS_n \cong \Omega_{\cL_n} \otimes_{\cA_n} \cS_n$ of right $\cS$-modules. Because the $\cR$-module $\Omega_{\cL} \otimes_{\cA} \cS$ is $\pi$-adically complete by \cite[Lemma 3.3]{DCapTwo} and the $\alpha_n$'s are compatible, the result follows by passing to the inverse limit.
\end{proof}

\begin{prop}\label{CurlySswitch} Suppose that $[\cL,\cL]\subseteq \pi \cL$ and $\cL \cdot \cA \subseteq \pi \cA$. Then the functors $\cM \mapsto \Omega_{\cL} \otimes_{\cA} \cM$ and $\cN \mapsto \Hom_{\cA}(\Omega_{\cL}, \cN)$ are mutually inverse equivalences of categories between the categories of finitely presented left $\cS$-modules and finitely presented right $\cS$-modules.
\end{prop}
\begin{proof} Because $\Omega_{\cL}$ is a finitely generated projective $\cA$-module, the given functors are exact mutually inverse equivalences between the categories of left $\cA$-modules and right $\cA$-modules. It follows from \cite[Lemma 3.3]{DCapTwo} that the functors send $\pi$-adically complete left $\cA$-modules to $\pi$-adically complete right $\cA$-modules. Corollary \ref{ULGswitch} and Corollary \ref{hUHGsideswitch} now imply that the functors restrict to exact mutually inverse equivalences between $\pi$-adically complete left $\cS$-modules and $\pi$-adically complete right $\cS$-modules. 

By construction,  $\cS = \cU \rtimes_H G$ is a free $\cU := \h{U(\cL)}$-module of finite rank, so every finitely generated $\cS$-module is also finitely generated as a $\cU$-module. Under the given assumptions on $\cL$, it follows from \cite[Lemma 4.1.9 and Proposition 4.1.6(c)]{EqDCap} that every finitely generated $\cS$-module is $\pi$-adically complete. Now, $\Omega_{\cL} \otimes_{\cA} \cS$ is a finitely generated projective right $\cS$-module by Corollary \ref{AlphaOnCurlyS}. Following the proof of Corollary \ref{FPswitch}, we see that $\Hom_{\cA}(\Omega_{\cL}, \cS)$ is a finitely generated projective left $\cS$-module. The result now follows from the exactness of the two functors.
\end{proof}

Recall that $\Omega(\bX)$ denotes the $\cO(\bX)$-module $\Hom_{\cO(\bX)}(\bigwedge^d_{\cO(\bX)} \cT(\bX), \cO(\bX))$, and that $\cA$ denotes some fixed $G$-stable affine formal model in $\cO(\bX)$.  If $U$ is a left (respectively, right) Noetherian ring, we will write $\Coh(U)$ (respectively, ${}^r\Coh(U)$) to denote the abelian category of finitely generated left (respectively, right) $U$-modules. 

\begin{thm}\label{LRequivBan} Let $\bX$ be a smooth $K$-affinoid variety. Suppose that
\be\item $\cT(\bX)$ admits a free $\cA$-Lie lattice $\cL$ for some affine formal model $\cA \subset \cO(\bX)$ which satisfies $[\cL,\cL]\subseteq \pi \cL$ and $\cL \cdot \cA \subseteq \pi \cA$, and
\item $G$ is a compact $p$-adic Lie group which acts on $\bX$ continuously and preserves $\cA \subset \cO(\bX)$ and $\cL \subset \cT(\bX)$.
\ee
Let $H$ be an open normal subgrop of $G$ contained in $G_{\cL}$, and let $U :=  \hK{U(\cL)} \rtimes_H G$. Then there is an equivalence of categories 
\[ \Coh(U) \quad\quad\stackrel{\cong}\quad\quad {}^r \Coh(U)\]
given by the functors $\Omega(\bX) \underset{\cO(\bX)}{\otimes}{}-$ and $\Hom_{\cO(\bX)}(\Omega(\bX),-)$.
\end{thm}
\begin{proof} Note first that under the given assumptions on $\cL$, $S := \hK{U(\cL)} \rtimes_H G$ is left and right Noetherian, by \cite[Lemma 4.1.9 and Theorem 4.1.4]{EqDCap}, so every finitely generated $S$-module is automatically finitely presented. 

Let $M$ be a finitely generated left $S$-module. Then there is an $\cS$-linear map $\epsilon : \cS^b \to M$ for some $b \geq 1$ whose image $\cM$ spans $M$ as a $K$-vector space. Note that $\cS^b$ is a free $\h{U(\cL)}$-module of finite rank, so $\ker \epsilon$ is a finitely generated $\cS$-module by \cite[Lemma 4.1.9 and Theorem 4.1.4]{EqDCap}, because of the given assumptions on $\cL$. Thus $\cM$ is a finitely presented left $\cS$-module, so $\Omega_{\cL} \otimes_{\cA} \cM$ is a finitely presented right $\cS$-module by Proposition \ref{CurlySswitch}. It follows that $\Omega(\bX) \otimes_{\cO(\bX)} M = K \otimes_{\cR} \left(\Omega_{\cL} \otimes_{\cA} \cM\right)$ is a finitely generated right $S$-module. The same argument shows that $\Hom_{\cO(\bX)}(\Omega(\bX),N)$ is a finitely generated left $S$-module whenever $N$ is a finitely generated right $S$-module. Because $\Omega(\bX)$ is a free $\cO(\bX)$-module of rank $1$, the two functors are mutually inverse adjoint equivalences of categories between left $\cO(\bX)$-modules and right $\cO(\bX)$-modules, and it follows from Lemma \ref{LRswitch}(c) that the unit and counit morphisms for the adjunction are in fact $S$-linear on finitely generated $S$-modules. So the restrictions of these functors to finitely generated $S$-modules are mutually inverse equivalences.
\end{proof}

\begin{lem}\label{AlphaLevelComp} Let $\cL_2 \leq \cL_1$ be two $G$-stable free $\cA$-Lie lattices in $\cT(\bX)$. Suppose that $H_2 \leq H_1$ are open normal subgroups of $G$ such that $H_i \leq G_{\cL_i}$ for $i=1,2$. Then there is a commutative diagram of right $\hK{U(\cL_2)} \rtimes_{H_2}G$-modules
\[ \xymatrix{ \Omega(\bX)\underset{\cO(\bX)}{\oslash}{} \left(\hK{U(\cL_2)} \rtimes_{H_2} G\right) \ar[r] \ar[d]_{\alpha_K}^\cong & \Omega(\bX)\underset{\cO(\bX)}{\oslash}{} \left(\hK{U(\cL_1)} \rtimes_{H_1} G\right) \ar[d]^{\alpha_K}_\cong \\ 
                                                                 \Omega(\bX)\underset{\cO(\bX)}{\otimes}{} \left(\hK{U(\cL_2)}\rtimes_{H_2}G\right) \ar[r]        & \Omega(\bX)\underset{\cO(\bX)}{\otimes}{} \left(\hK{U(\cL_1)}\rtimes_{H_1}G.\right) }\]
\end{lem}
\begin{proof} The vertical arrows come from Corollary \ref{AlphaOnCurlyS}, and the horizontal arrows come from \cite[Proposition 3.2.15]{EqDCap} and functoriality.
\end{proof}

\begin{thm}\label{LRswitchAffs} The pair $\Omega(\bX)\otimes_{\cO(\bX)} -$ and $\Hom_{\cO(\bX)}(\Omega(\bX),-)$ define mutually inverse equivalences of categories between coadmissible left $\w\cD(\bX,G)$-modules and coadmissible right $\w\cD(\bX,G)$-modules.
\end{thm}
\begin{proof} Let $\cL$ be a $G$-stable free $\cA$-Lie lattice in $\cT(\bX)$ such that $[\cL,\cL] \subseteq \pi^2 \cL$ and $\cL \cdot \cA \subseteq \pi \cA$ and note that each $\pi$-power multiple of $\cL$ also satisfies these conditions. Choose a good chain $(H_\bullet)$ for $\cL$ in $G$, in the sense of \cite[Definition 3.3.3]{EqDCap}. Write $S_n=\hK{U(\pi^n\cL)} \rtimes_{H_n} G$ and $S=\w\cD(\bX,G)$, so that we have a standard presentation $S = \invlim S_n$ as a Fr\'echet-Stein algebra, by the proof of \cite[Theorem 3.4.8]{EqDCap}. 

Suppose that $(M_\bullet)$ is a family of finitely generated left $S_\bullet$-modules and let $N_\bullet:=\Omega(\bX)\otimes_{\cO(\bX)} M_\bullet$. By Theorem \ref{LRequivBan}, $(N_\bullet)$ is a family of finitely generated right $S_\bullet$-modules. We will verify that $(N_\bullet)$ is coherent if and only if $(M_\bullet)$ is coherent, that is, that there are isomorphisms $N_{n+1}\otimes_{S_{n+1}}S_n\cong N_n$ for each $n\ge 0$ if and only if there are isomorphisms $S_n\otimes_{S_{n+1}}M_{n+1}\cong M_n$ for each $n\ge 0$.  If $Q$ is a finitely generated left $S_{n+1}$-module, define a right $S_n$-linear map
\[\theta_{Q}\colon (\Omega(\bX)\otimes_{\cO(\bX)} Q)\otimes_{S_{n+1}} S_n\to \Omega(\bX)\otimes_{\cO(\bX)} (S_n\otimes_{S_{n+1}}Q)\]
by setting $\theta_{Q} \left((\omega\otimes m)\otimes r\right) = (\omega\otimes 1\otimes m)r$. Then $\theta$ is a natural transformation between two right exact functors, and it follows from Lemma \ref{AlphaLevelComp} that $\theta_{S_{n+1}}$ is an isomorphism. Hence $\theta_Q$ is an isomorphism for all $Q$ by the Five Lemma. Thus $N_{n+1}\otimes_{S_{n+1}}S_n\cong \Omega(\bX)\otimes_{\cO(\bX)} (S_n\otimes_{S_{n+1}}M_{n+1})$. So $N_n\cong N_{n+1}\otimes_{S_{n+1}}S_n$ if and only if $M_n\cong S_n\otimes_{S_{n+1}}M_{n+1}$.

It now follows from Theorem \ref{LRequivBan} that $(M_\bullet) \mapsto (N_\bullet)$ is an equivalence of categories between coherent sheaves of left $S_\bullet$-modules and coherent sheaves of right $S_\bullet$-modules. Finally, since $\Omega(\bX)$ is a direct summand of a free $A$-module, for every coadmissible left $S$-module $M$ there are canonical isomorphisms
\[ \Omega(\bX) \underset{\cO(\bX)}{\otimes}{} M \cong \Omega(\bX) \underset{\cO(\bX)}{\otimes}{}(\invlim S_n \otimes_S M) \cong \invlim \Omega(\bX) \underset{\cO(\bX)}{\otimes}{} (S_n \otimes_S M) \]
of left $\cO(\bX)$-modules. Using these isomorphisms we can define a right $S$-module structure on $\Omega(\bX) \otimes_{\cO(\bX)} M$. Similarly, the canonical isomorphisms
\[ \Hom_{\cO(\bX)}(\Omega(\bX), N) \cong \Hom_{\cO(\bX)}(\Omega(\bX), N \otimes_S S_n) \cong S_n\otimes_S \Hom_{\cO(\bX)}(\Omega(\bX),N)\]
induce a left $S$-module structure on $\Hom_{\cO(\bX)}(\Omega(\bX), N)$ for every coadmissible right $S$-module $N$. The result now follows from \cite[Corollary 3.3]{ST}.
\end{proof}

We now come to the main result of $\S \ref{LRSwitch}$: it is an equivariant generalisation of \cite[Theorem 3.5]{DCapTwo}. Let $G$ be a $p$-adic Lie group acting continuously on the smooth rigid analytic variety $\bX$; we will write ${}^r \cC_{\bX/G}$ to denote the category of \emph{right} $G$-equivariant coadmissible $\cD$-modules on $\bX$.

\begin{thm}\label{MainLRswitch}  The functors $\Omega_{\bX}\otimes_{\cO_{\bX}}-$ and $\mathpzc{Hom}_{\cO_{\bX}}(\Omega_{\bX},-)$ are mutually inverse equivalences of categories between $\cC_{\bX/G}$ and ${}^r \cC_{\bX/G}$.
\end{thm}
\begin{proof} By working locally and using Lemma \ref{LRswitch}, we see that if $\cM$ is a left $G$-$\cD$-module on $\bX$, then $\Omega_{\bX} \otimes_{\cO_{\bX}} \cM$ is a right $G$-$\cD$-module on $\bX$; similarly if $\cM$ is a right $G$-$\cD$-module on $\bX$ then $\mathpzc{Hom}_{\cO_{\bX}}(\Omega_{\bX},\cM)$ is a left $G$-$\cD$-module on $\bX$. It is straightforward to verify that these functors convert locally Fr\'echet left $G$-equivariant $\cD$-modules into locally Fr\'echet right $G$-equivariant $\cD$-modules, and vice versa. Because these functors are naturally quasi-inverse on the level of sheaves of $\cO_{\bX}$-modules, it remains to check that they preserve coadmissibility. We will show that if $\cM \in \cC_{\bX/G}$, then $\Omega_{\bX} \otimes_{\cO_{\bX}} \cM \in {}^r\cC_{\bX,G}$, and leave the corresponding statement about $\mathpzc{Hom}_{\cO_{\bX}}(\Omega_{\bX},-)$ to the reader. 

Let $(\bX,G)$ be small and let $M$ be a coadmissible left $\w\cD(\bX,G)$-module, so that $\Omega(\bX) \otimes_{\cO(\bX)} M$ is a coadmissible right $\w\cD(\bX,G)$-module by Theorem \ref{LRswitchAffs}. In view of \cite[Definition 3.6.7]{EqDCap}, it remains to verify that there is a continuous right $G$-$\cD$-linear isomorphism
\begin{equation}\label{LocLRswitch} \Loc^{\w\cD(\bX,G)}_{\bX}\left( \Omega(\bX) \underset{\cO(\bX)}{\otimes}{} M\right) \quad \cong \quad \Omega_{\bX} \underset{\cO_X}{\otimes}{} \Loc^{\cD(\bX,G)}_{\bX}(M).\end{equation}
Let $\bY$ be an affinoid subdomain of $\bX$, and let $H$ be a $\bY$-small open subgroup of $G$. We will first exhibit an isomorphism of right $\w\cD(\bY,H)$-modules
\[ \tau(\bY,H) : \left(\Omega(\bX) \underset{\cO(\bX)}{\otimes}{} M\right) \underset{\w\cD(\bX,H)}{\w\otimes}{} \w\cD(\bY, H) \stackrel{\cong}{\longrightarrow} \Omega(\bY)\underset{\Omega(\bY)}{\otimes}{} \left( M \underset{\w\cD(\bX,H)}{\w\otimes}{} \w\cD(\bY,H) \right)\]
which is natural in $\bY$ and $H$. To this end, fix an $H$-stable affine formal model $\cA$ in $\cO(\bX)$ and an $H$-stable $\cA$-Lie lattice $\cL$ in $\cT(\bX)$. By rescaling $\cL$ if necessary and applying \cite[Lemma 7.6(b)]{DCapOne} together with \cite[Lemma 4.3.5]{EqDCap}, we may assume that $\cO(\bY)$ admits an $\cL$-stable and $H$-stable affine formal model $\cB$. We may also assume that $[\cL, \cL] \subseteq \pi \cL$ and $\cL \cdot \cA \subseteq \pi \cA$. Choose a good chain $(H_\bullet)$ in $H$ for $\cL$ using \cite[Lemma 3.3.6]{EqDCap}, and write $U_n:=\hK{U(\pi^n\cL)} \rtimes_{H_n} H$ and $V_n:= \hK{U(\cB\otimes_\cA\pi^n\cL)} \rtimes_{H_n} H$, so that $\w\cD(\bX,H) = \invlim U_n$ and $\w\cD(\bY,H) = \invlim V_n$ by \cite[Lemma 3.3.4]{EqDCap}. For each finitely generated left $U_n$-module $Q$, define a right $V_n$-linear map
\[ \tau_Q :  \left(\Omega(\bX)\underset{\cO(\bX)}{\otimes}{} Q\right)\underset{U_n}{\otimes}{}V_n\to  \Omega(\bY)\underset{\cO(\bY)}{\otimes}{} \left(V_n\underset{U_n}{\otimes}{}Q\right)\]
by setting $\tau_Q\left((\omega\otimes m)\otimes v\right)= (\omega|_{\bY}\otimes (1 \otimes m))v$. Then $\tau$ is a natural transformation between two right exact functors, and we see that $\tau_{U_n}$ is an isomorphism using Corollary \ref{AlphaOnCurlyS} together with the proof of \cite[Lemma 3.5]{DCapTwo}. Hence $\tau_{Q}$ is an isomorphism for all $Q$ by the Five Lemma. Passing to the limit as $n \to \infty$, we obtain the required $\w\cD(\bY,H)$-linear isomorphism $\tau(\bY,H)$ as the limit of the isomorphisms $\tau_{M_n}$ where $M_n := U_n \otimes_{\w\cD(\bX,H)} M$ for each $n \geq 0$. We leave it to the reader to verify that $\tau(\bY,H)$ is natural in both $\bY$ and $H$.

Finally, let $g \in G$ and write $N := \Omega(\bX) \otimes_{\cO(\bX)} M$. Using the definitions given above, we can check that the following diagram is commutative:
\[\xymatrix{ N \underset{\w\cD(\bX,H)}{\w\otimes}{} \w\cD(\bY, H) \ar[rr]^(0.4){\tau(\bY,H)} \ar[d]_{g^N_{\bY,H}} &&  \Omega(\bY)\underset{\Omega(\bY)}{\otimes}{} \left( \w\cD(\bY,H) \underset{\w\cD(\bX,H)}{\w\otimes}{} M \right) \ar[d]^{g^{\Omega}(\bY) \otimes g^M_{\bY,H}}\\
N \underset{\w\cD(\bX,H)}{\w\otimes}{} \w\cD(g\bY, gHg^{-1})  \ar[rr]_(0.4){\tau(g\bY,gHg^{-1})} &&  \Omega(g \bY)\underset{\Omega(g \bY)}{\otimes}{} \left( \w\cD(g\bY,gHg^{-1}) \underset{\w\cD(\bX,H)}{\w\otimes}{} M \right) }\]
where the vertical maps come from \cite[Proposition 3.5.7(a)]{EqDCap}. Now the maps $\tau(\bY,H)$ assemble together to form the continuous $G$-$\cD$-linear isomorphism in (\ref{LocLRswitch}). \end{proof} 

\subsection{Microlocal Kashiwara equivalence} \label{MicroKash}

In $\S \ref{MicroKash}$, we will extend the results from \cite[\S 4]{DCapTwo} to our setting, where the ground field $K$ is no longer assumed to be discretely valued, or equivalently, when its valuation subring $\cR$ is not Noetherian. As in \cite[\S 4.2]{DCapTwo}, we begin by fixing an affine formal model $\cA$ in a $K$-affinoid algebra $A$ and a $(\cR,\cA)$-Lie algebra $\cL$ which is finitely generated and projective as an $\cA$-module. We will work with \emph{right} $\hK{U(\cL)}$-modules throughout $\S \ref{MicroKash}$. 

Whenever $M$ is an $A$-module and $F$ is a subset of $A$, we denote the \emph{submodule of $F$-victims in $M$} by $M[F] := \{m \in M : mf = 0 \hsp \forall f \in F\}$. If, in addition, $M$ is a Banach $A$-module, we have at our disposal the $A$-submodule
\[ M_{\Dp}(F) := \{m \in M : \lim\limits_{n\to\infty} m \frac{f^n}{n!} = 0 \quad \forall f \in F\}\]
of \emph{divided-power topological $F$-torsion} in $M$. This is the case whenever $M$ is a finitely generated $\hK{U(\cL)}$-module, as $M$ then is naturally a Banach $A$-module.

When $F = \{f_1,\ldots, f_r\}$ is finite, we write $\cL \cdot (f_1,\ldots,f_r)$ for the $\cA$-submodule $\{(y \cdot f_1,\ldots, y\cdot f_r) : y \in \cL\}$ of $\cA^r$ and $C_{\cL}(F) := \{ y \in \cL : y \cdot f = 0$ for all $f \in F\}$ for the \emph{centraliser} of $F$ in $\cL$. The main result of $\S \ref{MicroKash}$ is 

\begin{thm}\label{MainMicroKash} Let $F = \{f_1,\ldots, f_r\} \subset A$ be such that $\cL \cdot (f_1,\ldots,f_r) = \cA^r$, and let $\cC := C_{\cL}(F)$. Suppose that $[\cL,\cL] \subseteq \pi \cL$ and $\cL \cdot \cA \subseteq \pi \cA$. Then there is an equivalence of categories
\[ \left\{ N \in {}^r\Coh(\hK{U(\cC)}) : N \cdot F = 0\right\} \quad\cong \quad \left\{ M \in {}^r\Coh(\hK{U(\cL)}) : M = M_{\Dp}(F)\right\}\]
given by the functors $N \mapsto N \underset{\hK{U(\cC)}}{\otimes}{} \hK{U(\cL)}$ and $M \mapsto M[F]$.
\end{thm}
Note that in general $F$ is not contained in $\cA$. Note also that the hypothesis that $\cL \cdot (f_1,\ldots, f_r) = \cA^r$ is equivalent to the existence of elements $x_1,\ldots, x_r \in \cL$ with the property that $x_i \cdot f_j = \delta_{ij}$ for all $i,j=1,\ldots, r$. Theorem \ref{MainMicroKash} is a direct generalisation of \cite[Theorem 4.10]{DCapTwo} to the setting where $K$ is not necessarily discretely valued and $\cL$ satisfies the additional restrictions $[\cL,\cL]\subseteq \pi \cL$ and $\cL \cdot \cA \subseteq \pi \cA$. Note that these restrictions are particularly mild because $\pi$ is an unspecified non-zero non-unit element of $\cR$, whose valuation, whilst positive, could be arbitrarily small.

\begin{lem}\label{UVfflat} Let $F = \{f_1,\ldots, f_r\} \subset A$ be such that $\cL \cdot (f_1,\ldots,f_r) = \cA^r$, and let $\cC = C_{\cL}(F)$. Suppose that $[\cL,\cL]\subseteq \pi \cL$ and $\cL \cdot \cA \subseteq \pi \cA$. Then $\hK{U(\cL)}$ is a faithfully flat $\hK{U(\cC)}$-module on both sides.
\end{lem}
\begin{proof} When $K$ is discretely valued, this follows from \cite[Corollary 4.3]{DCapTwo}. The only step of the proof of \cite[Corollary 4.3]{DCapTwo} which uses the Noetherianity of $\cR$ is the second application of \cite[Chapter II, Proposition 1.2.2]{LVO} in the proof of \cite[Proposition 4.3]{DCapTwo}, which requires that good filtrations on $\hK{U(\cL)}$-modules are separated, and that the $\pi$-adic filtration on $\hK{U(\cC)}$ has the left Artin-Rees property. It follows from \cite[Lemma 4.1]{DCapTwo} that our hypotheses on $\cL$ imply that $\cC$ also satisfies $[\cC,\cC]\subseteq \pi\cC$ and $\cC \cdot \cA \subseteq \pi \cA$. Now in view of \cite[Lemma 4.1.9]{EqDCap}, we see that both of these facts have already been verified in the proof of \cite[Proposition 4.1.7(a)]{EqDCap}.
\end{proof}

Until the end of $\S \ref{MicroKash}$, we will assume that $[\cL, \cL]\subseteq \pi \cL$ and $\cL \cdot \cA \subseteq \pi \cA$, and define $\cU := \h{U(\cL)}$, $U := \hK{U(\cL)}$ and $V := \hK{U(\cC)}$. Unfortunately, the next step in the proof of \cite[Theorem 4.10]{DCapTwo}, namely \cite[Proposition 4.4]{DCapTwo} seems unlikely to hold as stated. The problem is that the definition of \emph{$\cU$-lattice} used in \cite{DCapTwo} seems to be too strong. Fortunately, we do not need the full strength of \cite[Proposition 4.4]{DCapTwo} to continue with the proof.

\begin{defn} Let $M$ be a finitely generated $U$-module, and let $\cM$ be a $\cU$-submodule of $M$. We say that $\cM$ is a \emph{weak $\cU$-lattice} if it is the unit ball in $M$ with respect to some Banach norm on $M$ which induces the canonical Banach topology on $M$. If $f \in A$, then we say that $\cM$ is \emph{stable under divided powers of $f$} if $\cM \cdot f^n/n!\subseteq \cM$ for all $n \geq 0$.

\end{defn}
In more algebraic terminology, we require $\cM$ to span $M$ as a $K$-vector space, to satisfy $\cM_0 \subseteq \cM \subseteq \pi^{-t} \cM_0$ for some $t \geq 0$ and for some finitely generated $\cU$-submodule $\cM_0$ of $M$, and to be closed in the $\pi$-adic topology on $M$ defined by $\cM_0$. Note that any such weak $\cU$-lattice is automatically $\pi$-adically separated and complete as an $\cR$-module. Our modified version of \cite[Proposition 4.4]{DCapTwo} now reads as follows.

\begin{prop}\label{WeakLat} Let $f \in A$ be such that $\cL \cdot f \subseteq \cA$ and let $M$ be a finitely generated $U$-module. Then there is at least one weak $\cU$-lattice $\cM$ in $M_{\Dp}(f)\cdot U$ which is stable under divided powers of $f$.
\end{prop}
\begin{proof} Note that $M_{\Dp}(f)\cdot U$ is a finitely generated submodule of $M$ because $U$ is Noetherian by \cite[Corollary 4.1.10]{EqDCap}. The proof of \cite[Proposition 4.4]{DCapTwo} now carries over in a straightforward manner.
\end{proof}

\begin{prop} \label{DpUSubmod} Suppose that $F\subseteq A$ is such that $\cL\cdot F\subseteq \cA$, and $M$ is a finitely generated $U$-module. Then $M_{\Dp}(F)$ is a $U$-submodule of $M$. \end{prop}
\begin{proof} The proof of \cite[Proposition 4.5]{DCapTwo} works, provided we use Proposition \ref{WeakLat} to construct and use a weak $\cU$-lattice $\cM$ in $M_{\Dp}(f)$ for any $f \in F$.
\end{proof}

Note that if $M$ is a finitely generated $U$-module then $M[f]$ is a closed $V := \hK{U(\cC)}$-submodule of $M$, where $\cC := C_{\cL}(f)$.
\begin{lem}\label{qjxj} Let $f \in A$ and $x \in \cL$ satisfy $x \cdot f = 1$, and let $M$ be a $U$-module. Then for every $q_0,\ldots,q_n \in M[f]$, we have $\sum_{j=0}^n q_j x^j \cdot \frac{f^n}{n!} = q_n$.
\end{lem}
\begin{proof} This was shown in the proof of \cite[Lemma 4.6]{DCapTwo}.
\end{proof}
Our next statement agrees precisely with \cite[Proposition 4.7]{DCapTwo} when $K$ is discretely valued. Unfortunately the proof of \cite[Proposition 4.7]{DCapTwo} as written relies heavily on the Noetherianity of $\cU$ and therefore does not extend to our current setting. Fortunately, it is possible to rewrite the proof in a different way that \emph{does} extend.
\begin{prop} Let $f\in A$ be such that $\cL\cdot f=\cA$ and let $M$ be a finitely generated $U$-module. Then the natural $U$-linear map $\epsilon_M\colon M[f]\otimes_VU\to M$ is injective, and hence $M[f]$ is finitely generated as a $V$-module.
\end{prop}
\begin{proof} Let $\xi = \sum_{i=1}^d \xi_i \otimes u_i$ lie in the kernel of $\epsilon_M$ for some $\xi_i \in M[f]$ and $u_i \in U$. Choose $x \in \cL$ such that $x \cdot f = 1$. Now, $\cL = \cC \oplus \cA x$ by \cite[Lemma 4.1]{DCapOne}, so $U(\cL)$ is isomorphic to $U(\cC)[x]$ as a left $U(\cL)$-module. Therefore we can write $u_i = \sum_{j=0}^\infty v_{ij} x^j$ for some $v_{ij} \in V$ such that $v_{ij} \to 0$ as $j\to \infty$ for all $i=1,\ldots,d$. Define $q_j := \sum_{i=1}^d \xi_i v_{ij} \in M[f]$ for each $j \geq 0$, and note that  $q_j \to 0$ as $j \to \infty$. Then $\sum_{j=0}^n q_j x^j = \sum_{i=1}^d \xi_i \cdot \sum_{j=0}^n v_{ij} x^j \to \sum_{i=1}^d \xi_i \cdot u_i = \epsilon_M(\xi) = 0$ as $n \to \infty$. We will show that $q_k = 0$ for any $k \geq 0$. 

Using Proposition \ref{DpUSubmod} we see that the image $M[f]\cdot U$ of $\epsilon_M$ is contained in the $U$-module $M_{\Dp}(f)$.  Using Proposition \ref{WeakLat}, choose a weak $\cU$-lattice $\cM$ in $M_{\Dp}(f) = M_{\Dp}(f)\cdot U$ which is stable under divided powers of $f$. Fix $r \geq 0$; because $\lim_{n\to\infty} \sum_{j=0}^n q_j x^j = 0$, we can find $n > k$ such that $\sum_{j=0}^n q_j x^j \in \pi^r \cM$. Since $\pi^r \cM$ is stable under divided powers of $f$ and the action of $x$, we see that 
\[\sum_{j=0}^n q_j x^j \cdot \left(1 - \frac{f^n}{n!}x^n\right) \cdots \left(1 - \frac{f^{k+1}}{(k+1)!}x^{k+1}\right)\cdot \frac{f^k}{k!}\in \pi^r \cM.\] 
It follows from Lemma \ref{qjxj} that the expression on the left hand side equals $q_k$. Hence $q_k \in \pi^r \cM$ for all $k, r \geq 0$. Since $\cM$ is $\pi$-adically separated, $q_k = 0$ for all $k \geq 0$ as claimed.

Let $\cN$ be the $\cU$-submodule of $M[f] \otimes_V U$ generated by $\{\xi_1 \otimes 1,\ldots, \xi_d \otimes 1\}$. Then for any $n \geq 0$ we have
\[\xi = \sum_{i=1}^d \xi_i \otimes u_i = \sum_{i=1}^d \xi_i \otimes \left( \sum_{j=0}^n v_{ij} x^j + \sum_{j > n} v_{ij} x^j \right) = \sum_{j=0}^n q_j \otimes x^j + \sum_{i=1}^n \xi_i \otimes \sum_{j > n} v_{ij} x^j.\]
Since $q_j = 0$ for all $j$ and since $\sum_{j > n} v_{ij} x^j \to 0$ in $\cU$ as $n \to \infty$, we conclude that $\xi \in \pi^r \cN$ for all $r \geq 0$. Because $\cN$ is a finitely generated $\cU$-module, it is $\pi$-adically separated by \cite[Proposition 4.1.6(b)]{EqDCap}. Hence $\xi = 0$ and $\epsilon_M$ is injective as claimed. 

Finally, if $N_1 \leq N_2 \leq \cdots $ is an ascending chain of $V$-submodules of $M[f]$, then $\epsilon_M(N_1 \otimes_V U) \leq \epsilon_M(N_2 \otimes_V U) \leq \cdots$ is an ascending chain of $U$-submodules of $M$. This chain must terminate because $M$ is a finitely generated module over the Noetherian ring $U$ --- see \cite[Corollary 4.1.10]{EqDCap}. Since $\epsilon_M$ is $U$-linear and injective by the above, and since $U$ is a faithfully flat $V$-module by Proposition \ref{UVfflat}, we see that the original chain $N_1 \leq N_2 \leq \cdots$ must terminate. Hence the $V$-module $M[f]$ is finitely generated.
\end{proof}

\begin{proof}[Proof of Theorem \ref{MainMicroKash}]
The remaining proofs in \cite[\S 4.7, \S4.8, \S4.9, \S4.10]{DCapTwo} do not rely on the Noetherianity of $\cU$. The proof of \cite[Theorem 4.10]{DCapTwo} now works after making straightforward modifications.
\end{proof}

\subsection{Kashiwara equivalence for small \ts{(\bX,G)}}\label{EqMKE}
In $\S \ref{EqMKE}$, we specialise slightly to the setting of untwisted $\cD$-modules, and extend the material in $\S \ref{MicroKash}$ to the equivariant setting. Thus, let $\bX$ be a smooth $K$-affinoid variety and let $\bY$ be a smooth, Zariski closed subvariety of $\bX$ defined by a coherent ideal sheaf $\sI \subset \cO_{\bX}$. We assume that $\cT(\bX)$ admits a free $\cA$-Lie lattice $\cL = \cA \partial_1 \oplus \cdots \oplus \cA \partial_d$ for some affine formal model $\cA \subset \cO(\bX)$, and that $I := \sI(\bX)$ admits a generating set $F := \{f_1,\ldots, f_r\}$ such that $\partial_i(f_j) = \delta_{ij}$ whenever $1 \leq i \leq d$ and $1 \leq j \leq r$. 

Let $\cI := I \cap \cA$ and let $N_{\cL}(\cI) := \{v \in \cL : v \cdot \cI \subseteq \cI\}$ be its normaliser in $\cL$. 

\begin{lem}\label{Normaliser} We have $N_{\cL}(\cI) = \cI \partial_1 \oplus \cdots \oplus \cI \partial_r \oplus \cA \partial_{r+1} \oplus \cdots \oplus \cA \partial_d$.
\end{lem}
\begin{proof} If $a \in \cI$, then $a = \sum_{j=1}^r a_j f_j$ for some $a_j \in \cO(\bX)$ so $\partial_i( a ) = \sum_{j=1}^r \partial_i(a_j) f_j \in I$ whenever $i > r$. Since $a \in \cA$ and $\partial_i \in \cL$, we see that $\partial_i(a) \in \cA$. So, $\partial_i \in N_{\cL}(\cI)$ whenever $i > r$. On the other hand, $\cI \partial_i(\cI) \subset \cI \cA = \cI$ for all $i$, and we have established the $\supseteq$ inclusion. Now let $v \in N_{\cL}(\cI)$ and write $v = \sum_{i=1}^d a_i \partial_i$ for some $a_i \in \cA$. Choose a sufficiently large integer $n$ such that $\pi^n f_j \in \cI$ for all $j=1,\ldots, r$. Then $v(\pi^n f_j) \in \cI$ implies that $\pi^n a_j \in I$ for all $j = 1,\ldots, r$. Since $I$ is a $K$-vector space, we conclude that $a_1,\ldots, a_r \in I \cap \cA = \cI$, as required for $\subseteq$ inclusion.\end{proof}

We write $\overline{\cA} := \cA / \cI$ and $\cN$ for the quotient $(\cR, \overline{\cA})$-Lie algebra $N_{\cL}(\cI) / \cI\cL$. Note that it follows from Lemma \ref{Normaliser} that the image of $\{\partial_{r+1},\ldots, \partial_d\}$ in $\cN$ forms a basis for this $\overline{\cA}$-module. 

Now let $G$ be a compact $p$-adic Lie group acting continuously on $\bX$, and assume that $\bY \subset \bX$, $\cA \subset \cO(\bX)$ and $\cL \subset \cT(\bX)$ are all $G$-stable. Note that $\overline{\cA}$ is a $G$-stable formal model in $\cO(\bY) = \cO(\bX)/I$, and by the functoriality of the normaliser construction, $\cN$ is a $G$-stable $\overline{\cA}$-Lie lattice in $\cT(\bY) = \cO(\bY) \overline{\partial_{r+1}} \oplus \cdots \oplus \cO(\bY) \overline{\partial_d}$. We note in passing that even given the explicit description of normaliser $N_{\cL}(\cI)$ by Lemma \ref{Normaliser}, it is not particularly nice as an $\cA$-module so the structure of its enveloping algebra is not clear. However, it does admit a $G$-action which induces a well-defined $G$-action on $\cN$. In contrast, because it is not reasonable to expect the generating set $F$ for $I$ to be $G$-stable in general, there is no $G$-action on the centraliser $C_{\cL}(F)$ which played a key role in $\S\ref{MicroKash}$.

Let $\rho : G \to \Aut_K(\cO(\bX))$ denote the action of $G$ on $\cO(\bX)$, and recall the subgroup $G_{\cL} = \rho^{-1}(\exp(p^\epsilon \cL))$ from \cite[Definition 3.2.11]{EqDCap}. By \cite[Theorem 3.2.12]{EqDCap}, the map $\beta : G_{\cL} \to U^\times$ given by $\beta(g) = \exp(\iota(\log\rho(g)))$ for $g \in G_{\cL}$ is a $G$-equivariant trivialisation of the $G$-action on $U := \hK{U(\cL)}$; here $\iota$ denotes the inclusion of $\cL$ into the $K$-Banach algebra $U$. Let $\overline{\rho} : G \to \Aut_K\cO(\bY)$ denote the action of $G$ on $\cO(\bY)$, let $W := \hK{U(\cN)}$ and let $\overline{\beta} : G_{\cN} \to W^\times$ denote the corresponding $G$-equivariant trivialisation; if $a \in \cO(\bX)$ we will write $\overline{a}$ to denote its image in $\cO(\bY)$.

\begin{lem}\label{GLGN} We have $G_{\cL} \leq G_{\cN}$. If $g \in G_{\cL}$ and $a \in \cO(\bX)$, then
\[ \overline{\beta}(g) \cdot \overline{a} = \overline{ \beta(g) \cdot a }.\]
\end{lem}
\begin{proof}  Let $g \in G_{\cL}$ and let $u \in \cL$ be such that $\log \rho(g) = p^\epsilon u$. Because $\bY$ is $G$-stable by assumption, the automorphism $\rho(g) : \cO(\bX) \to \cO(\bX)$ stabilises the ideal $I$. Since this ideal is closed, it follows that 
\[p^\epsilon u = \sum\limits_{m=1}^\infty \frac{ (-1)^{m+1}}{m}(\rho(g) - 1)^m\]
also stabilises $I$. Since $I$ is a $K$-vector space, we conclude that $u \cdot I \subseteq I$, and hence $u \in N_{\cL}(\cI)$. Now, the kernel of the natural map $N_{\cL}(\cI) \to \cT(\bY)$ is equal to $\cI\cL$, so the image of $N_{\cL}(\cI)$ in $\cT(\bY)$ is isomorphic to $\cN = N_{\cL}(\cI) / \cI \cL$. Hence the image $\overline{u}$ of $u$ in $\cT(\bY)$ lies in $\cN$, and on the other hand it satisfies $\log \overline{\rho(g)} = p^\epsilon \overline{u}$. So, $\log \overline{\rho}(g) = \log \overline{\rho(g)} \in p^\epsilon \cN$ which implies that $g \in G_{\cN}$. Finally, if $g \in G_{\cL}$ and $a \in \cO(\bX)$ then $\overline{\beta}(g) \cdot \overline{a} = \overline{\rho}(g)(\overline{a}) = \overline{\rho(g)(a)} = \overline{\beta(g) \cdot a}$ as required.
\end{proof}

We will use the following abbreviations. 
\begin{notn}\label{WUST} Fix an open normal subgroup $H$ of $G$ contained in $G_{\cL}$. 
\be \item  $W = \hK{U(\cN)} \hookrightarrow S := W \rtimes^{\overline{\beta}}_H G$, and
\item $U = \hK{U(\cL)} \hookrightarrow T := U \rtimes^\beta_H G$.
\ee\end{notn}
Note that it follows from \cite[Lemma 5.1(d)]{DCapTwo} that $\cO(\bY) \underset{\cO(\bX)}{\otimes}{} U = U / I U$ is an $W-U$-bimodule; hence 
\[ \cO(\bY) \underset{\cO(\bX)}{\otimes}{} T \cong  (U / I U) \otimes_U T \cong T / IT\]
is naturally a $W-T$-bimodule.

\begin{prop}\label{TITbimodule} The left $W$-action on $(U / IU) \otimes_U T$ by right $T$-linear endomorphisms extends to a left $S$-action. 
\end{prop}
\begin{proof} Recall that by \cite[Definition 2.2.3 and Lemma 2.2.4]{EqDCap}, $T$ is a factor ring of the skew-group ring $U \rtimes G$. The group $G$ acts on $U \rtimes G$ by left multiplication, and this action therefore descends to a left $G$-action on $T$ by right $T$-module endomorphisms. The ideal $IT$ is stable under this action because $g \cdot (f t) = (g \cdot f) g t$ for all $g \in G$, $f \in I$ and $t \in T$, and because the ideal $I$ in $\cO(\bX)$ is $G$-stable by assumption. Thus we obtain a left $G$-action on $T / IT$ by right $T$-module endomorphisms. On the other hand, the left $W$-action on $U / IU$ is $G$-equivariant, in the sense that
\[ g \cdot (w \cdot \overline{u}) = (g \cdot w) \cdot \overline{g \cdot u} \qmb{for all} g \in G, w \in W, u \in U.\] 
An easy calculation now implies that the right $T$-linear $W$ and $G$-actions on $T/IT$ combine to form a right $T$-linear $W \rtimes G$-action. It remains to check that this action descends to the factor ring $S = W \rtimes^{\overline{\beta}}_HG$ of $W \rtimes G$.

Fix $h \in H$, choose $v \in N_{\cL}(\cI)$ such that $\log \rho(h) = p^\epsilon v$ so that $\overline{\beta}(h) = \exp(p^\epsilon \iota(\overline{v})) \in W$ and $\beta(h) = \exp(p^\epsilon \iota(v)) \in U$. The $\overline{v}$-action on $T/IT$ is given by $\overline{v}\cdot (t + IT) = [\iota(v),t] + t \hsp \iota(v) + IT$. An easy induction then implies that 
\[\overline{v}^m \cdot(t + IT) = \sum\limits_{i=0}^m \binom{m}{i} \ad(\iota(v))^i(t) \hsp \iota(v)^{m-i} + IT\]
for all $m \geq 0$ and all $t \in T$, and therefore
\[ \begin{array}{lll} \exp(p^\epsilon \iota(\overline{v})) \cdot (t + IT) &=& \sum\limits_{m=0}^\infty \sum\limits_{i=0}^m \frac{p^{\epsilon m}}{m!}\binom{m}{i} \ad(\iota(v))^i(t) \hsp \iota(v)^{m-i} + IT = \\
&=& \sum\limits_{i=0}^\infty \sum\limits_{j=0}^\infty  \frac{(p^\epsilon \ad(\iota(v)))^i}{i!}(t) \frac{(p^\epsilon \iota(v))^j}{j!} + IT = \\
&=& \exp(p^\epsilon \ad(\iota(v)))(t) \hsp \exp(p^\epsilon \iota(v)) + IT.\end{array}\]
Now if $t = a \in \cO(\bX)$, then $\ad(\iota(v))(a) = v(a)$ and $\exp(p^\epsilon \ad(\iota(v)))(a) = \rho(h)(a) = h \cdot a$, so $\overline{\beta}(h) \cdot (a + IT) = (h \cdot a) \beta(h) + IT = h (a + IT)$. We conclude that $\overline{\beta}(h) - h \in W \rtimes G$ kills $(\cO(\bX) + IT) / IT$ for all $h \in H$. Because this set generates $T/IT$ as a right $T$-module, and because the action of $W \rtimes G$ on $T/IT$ is right $T$-linear, we see that $\overline{\beta}(h) - h$ kills all of $T/IT$. It follows that the left $W \rtimes G$-action on $T / IT$ descends to a well-defined $S = W \rtimes_H^{\overline{\beta}}G$-action, as required. \end{proof}

It follows immediately from Proposition \ref{TITbimodule} that $T / IT$ is naturally an $S - T$-bimodule. Next, recall from \cite[\S 2.2]{EqDCap} that $\gamma : G \to S^\times$ denotes the canonical group homomorphism, and that by construction, $S$ is a free left and right $W$-module with basis $\{\gamma(c) : c \in C\}$ whenever $C$ is a set of coset representatives for the open normal subgroup $H$ of $G$. Similar statements hold for the crossed product $T = U \rtimes^\beta_HG$.
\begin{prop}\label{SUbimod} There is an isomorphism of $S - U$-bimodules
\[ \xi : S \underset{W}{\otimes}{} \frac{U}{IU} \hsp \stackrel{\cong}{\longrightarrow} \hsp \frac{T}{IT}\]
given by $\xi( w \gamma(g) \otimes \overline{u} ) = \overline{ w \cdot (g \cdot u) \gamma(g)}$.
\end{prop}
\begin{proof} By Proposition \ref{TITbimodule}, there is a well-defined $S-T$-bilinear action map $S \times T/IT \to T/IT$, given by $(w \gamma(g), \overline{u \gamma(h)}) \mapsto \overline{w \cdot (g \cdot u) \hsp \gamma(gh)}$ for $w \in W, g,h \in H$ and $u \in U$. By considering this formula in the case $g = h = 1$, we see that the restriction of this map to $S \times U / IU$ descends to $S \otimes_W U/IU$, and induces an $S-U$-bilinear map
\[ \xi : S \underset{W}{\otimes}{} \frac{U}{IU} \stackrel{\cong}{\longrightarrow} \frac{T}{IT}\]
as claimed. Let $C$ be a set of coset representatives for $H$ in $G$. If we again identify $T/IT$ with $U/IU \otimes_U T$, then it follows from \cite[Lemma 2.2.4(b)]{EqDCap} that
\[ S \underset{W}{\otimes}{} \frac{U}{IU} = \bigoplus\limits_{c \in C} \gamma(c) \otimes \frac{U}{IU} \qmb{and} \frac{T}{IT} = \bigoplus\limits_{c \in C} \frac{U}{IU} \otimes \gamma(c). \]
For each $c \in C$, the restriction of $\xi$ to $\gamma(c) \otimes \frac{U}{IU}$ sends it bijectively onto $\frac{U}{IU} \otimes \gamma(c)$. Therefore $\xi$ is a bijection.
\end{proof}

Now, the $S-T$-bimodule $T / IT$ from Proposition \ref{TITbimodule} induces a functor $N \mapsto N \otimes_S T/IT$ from finitely generated right $S$-modules to finitely generated right $T$-modules. Using the identification of the restriction of this bimodule to an $S-U$-bimodule given by Proposition \ref{SUbimod}, we can prove the main result of $\S \ref{EqMKE}$.

\begin{thm}\label{MainEqMKE} Let $\bX$ be a smooth $K$-affinoid variety and let $\bY = \Sp \cO(\bX)/I$ be a smooth, Zariski closed subvariety of $\bX$. Assume that: 
\be\item$\cT(\bX)$ admits a free $\cA$-Lie lattice $\cL = \cA \partial_1 \oplus \cdots \oplus \cA \partial_d$ for some affine formal model $\cA \subset \cO(\bX)$, satisfying $[\cL,\cL] \subseteq \pi \cL$ and $\cL \cdot \cA \subseteq \pi \cA$,
\item the ideal $I$ admits a generating set $F := \{f_1,\ldots, f_r\}$ such that $\partial_i(f_j) = \delta_{ij}$ whenever $1 \leq i \leq d$ and $1 \leq j \leq r$,
\item $G$ is a compact $p$-adic Lie group which acts continuously on $\bX$ and which preserves $\bY \subset \bX$, $\cA \subset \cO(\bX)$ and $\cL \subset \cT(\bX)$.
\ee
Let $H$ be an open normal subgroup of $G$ contained in $G_{\cL}$ and write 
\[S := \hK{U(\cN)}\rtimes^{\overline{\beta}}_HG \qmb{and} T := \hK{U(\cL)} \rtimes^\beta_HG,\]
where $\cN = \frac{N_{\cL}(I \cap \cA)}{(I \cap \cA)\cL} \subset \cT(\bY)$. Then there is an equivalence of categories
\[{}^r\Coh(S) \quad\cong\quad \left\{M \in {}^r\Coh(T) : M = M_{\Dp}(F)\right\}\]
given by the functors $N \mapsto N \underset{S}{\otimes}{} T / IT$ and $M \mapsto \Hom_T(T/IT,M)$.
\end{thm}
\begin{proof} Because $G/H$ is a finite group, it follows from \cite[Lemma 2.2.4(b)]{EqDCap} that $S$ is a finitely generated $W$-module on both sides, and $T$ is a finitely generated $U$-module on both sides. Hence we obtain the following diagram of abelian categories and right-exact functors, where the vertical arrows are the restriction maps along the inclusion of rings $W \hookrightarrow S$ and $U \hookrightarrow T$:
\[\xymatrix{ {}^r\Coh(S)  \ar[rrrr]^{-\otimes_S T/IT} \ar[d] &&&& {}^r\Coh(T) \ar[d] \\ {}^r\Coh(W) \ar[rrrr]_{-\otimes_W U/IU} &&&& {}^r\Coh(U).}\]
It follows from Proposition \ref{SUbimod} and the Five Lemma that this diagram commutes up to natural isomorphism: for every finitely generated right $S$-module $N$, there is a right $U$-linear isomorphism 
\[ N \underset{S}{\otimes}{} \frac{T}{IT} \quad\cong\quad N \underset{W}{\otimes}{} \frac{U}{IU}\]
which is natural in $N$. Let $\cC := C_{\cL}(F) = \cA \partial_{r+1} \oplus \cdots \oplus \cA \partial_d$; it follows from Lemma \ref{Normaliser} that the natural map $\cC / \cI \cC \to \cN$ is an isomorphism. Let $V$ denote the $K$-Banach algebra $\hK{U(\cC)}$ from $\S \ref{MicroKash}$; by the proof of \cite[Lemma 5.8]{DCapTwo}, there is a natural isomorphism $V / I V \cong W$. So, if $N \in {}^r\Coh(S)$ then we may regard its restriction to $W$ as a finitely generated $V$-module killed by $I$, and then the right $U$-module
\[ M := N \underset{S}{\otimes}{} \frac{T}{IT} \quad \cong \quad N \underset{W}{\otimes}{} \frac{U}{IU} \quad\cong\quad N \underset{V}{\otimes}{} U\]
lies in ${}^r\Coh(U)$ and satisfies $M = M_{\Dp}(F)$ by Theorem \ref{MainMicroKash}. 

Now suppose that $M \in {}^r\Coh(T)$ satisfies $M = M_{\Dp}(F)$. On the one hand, its restriction to $U$ satisfies the same condition because it is phrased entirely in terms of the action of $F \subset \cO(\bX) \subset U$ on $M$, and therefore $\Hom_T(T/IT, M) = M[I]$ is a finitely generated right $V$-module killed by $I$ by Theorem \ref{MainMicroKash}. In other words, it is a finitely generated right $W$-module.  On the other hand, it is a right $S$-module because $T/IT$ is a left $S$-module; hence it is per force finitely generated as an $S$-module. So, the functors $-\otimes_S (T/IT)$ and $\Hom_T(T/IT,-)$ restrict to an adjunction between ${}^r\Coh(S)$ and $\left\{M \in {}^r\Coh(T) : M = M_{\Dp}(F)\right\}$, and it remains to show that the unit and counit morphisms of this adjunction are isomorphisms.

To this end, let $M \in {}^r\Coh(T)$ satisfy $M = M_{\Dp}(F)$ and let $\underline{M}$ denote its restiction to $U$. In the commutative diagram
\[\xymatrix{ \underline{M}[I] \underset{W}{\otimes}{} \frac{U}{IU} \ar[rrrr]^{\epsilon_{\underline{M}}}\ar[d]_{\cong} &&&& M \\
M[I] \underset{S}{\otimes}{} \left( S \underset{W}{\otimes}{} \frac{U}{IU} \right) \ar[rrrr]_{1_{M[I]} \otimes \xi} &&&& M[I] \underset{S}{\otimes}{} \frac{T}{IT}\ar[u]_{\epsilon_M}}\]
the bottom horizontal arrow is an isomorphism by Proposition \ref{SUbimod}, whereas the top horizontal arrow is an isomorphism by Theorem \ref{MainMicroKash}. So the counit morphism $\epsilon_M$ is an isomorphism; and a similar argument using Proposition \ref{SUbimod} and Theorem \ref{MainMicroKash} shows that the unit morphism $\eta_N : N \to (N \otimes_S \frac{T}{IT})[I]$ is also an isomorphism for each $N \in {}^r\Coh(S)$. \end{proof}

For future use, we record that the functor $M \mapsto M[I]$ is compatible with localisation to appropriately invariant affinoid subdomains $\bX' \subset \bX$.

\begin{prop}\label{LocBanIvict} Suppose that $\bX, \bY, I, \cA, \cL, G$ satisfy the hypotheses of Theorem \ref{MainEqMKE}. Let $\bX'$ be a $G$-stable affinoid subdomain of $\bX$ such that $\cO(\bX')$ admits a $G$-stable and $\cL$-stable affine formal model $\cA'$, and let $\cL' := \cL \otimes_{\cA} \cA'$. 

Let $H$ be an open normal subgroup of $G$ contained in $G_{\cL}$, let $S,T$ be defined as in Theorem \ref{MainEqMKE} and define  $S' := \hK{U(\cN')} \rtimes_H^{\overline{\beta'}} G \qmb{and} T' := \hK{U(\cL')} \rtimes_H^{\beta'} G$. Then the natural map 
\[ \alpha : M[I] \utimes{S} S' \quad\longrightarrow\quad (M \utimes{T} T')[I]\]
is an isomorphism for every $M \in {}^r\Coh(T)$ such that $M = M_{\Dp}(F)$. 
\end{prop}
\begin{proof} Note that $\cL'$ is an $\cA'$-Lie lattice in $\cO(\bX')$ and $G_{\cL} \leq G_{\cL'}$ by \cite[Proposition 4.3.6]{EqDCap}. Let $\cI' = \sI(\bX')\cap \cA'$ and $\cN' = N_{\cL'}(\cI')/\cI' \cL'$; then it follows from Lemma \ref{Normaliser} that $\cN' \cong (\cA' / \cI') \otimes_{\cA/\cI} \cN$ so again $G_{\cN} \leq G_{\cN'}$  by \cite[Proposition 4.3.6]{EqDCap}. This means that the crossed products $S'$ and $T'$ are well defined. 

Write $W := \hK{U(\cN)}$, $W' := \hK{U(\cN')}$, $U := \hK{U(\cL)}$ and $U' := \hK{U(\cL')}$. By the right-module version of  \cite[Proposition 4.3.11]{EqDCap}, there are natural isomorphisms $T' \cong T \otimes_U U' $ of $T-U'$-bimodules and $S' \cong S \otimes_W W'$ of $S-W'$-bimodules. These give isomorphisms $\gamma : M[I] \otimes_SS' \stackrel{\cong}{\longrightarrow} M[I]\otimes_WW'$ right $W'$-modules and $\delta : M \otimes_T T' \stackrel{\cong}{\longrightarrow} M \otimes_U U'$ of right $U'$-modules. Now the map $\alpha$ in question appears in the following commutative diagram:
\[\xymatrix{ \left(M[I] \utimes{S} S'\right) \utimes{W'} \frac{U'}{IU'} \ar[rr]^{\alpha \otimes 1} \ar[d]_{\gamma \otimes 1} && \left(M \utimes{T}T'\right)[I] \utimes{W'} \frac{U'}{IU'} \ar[d]^{\delta \otimes 1}\\
\left(M[I] \utimes{W} W'\right) \utimes{W'} \frac{U'}{IU'} \ar[d]_{\cong} && \left(M \utimes{U}U'\right)[I] \utimes{W'} \frac{U'}{IU'} \ar[d]^{\epsilon_{M \otimes_U U'}}\\
\left(M[I] \utimes{W} \frac{U}{IU}\right) \utimes{U} U' \ar[rr]_{\epsilon_M \otimes 1} && M \utimes{U} U'.
}\]
Here the second vertical arrow on the left is an isomorphism obtained by contracting tensor products. Let $M' := M \otimes_U U'$; because $M = M_{\Dp}(F)$, it follows from Proposition \ref{DpUSubmod} that also $M' = M'_{\Dp}(F)$. So the maps $\epsilon_M$ and $\epsilon_{M'}$ are isomorphisms by Theorem \ref{MainEqMKE}. Since $\gamma$ and $\delta$ are also isomorphisms, we conclude that the top horizontal map $\alpha \otimes 1$ is an isomorphism. Finally, the functor $-\otimes_{W'} U'/IU'$ is isomorphic to $-\otimes_{V'}U'$ where $V' = \hK{U(\cC_{\cL'}(F))}$, and $-\otimes_{V'}U'$ reflects isomorphisms because $U'$ is a faithfully flat left $V'$-module by Proposition \ref{UVfflat}. So $\alpha$ is an isomorphism as required.
\end{proof}

Now let $U_n := \hK{U(\pi^n \cL)}$ and $W_n := \hK{U(\pi^n \cN)}$, and choose a good chain $H_\bullet$ for $H$ using \cite[Lemma 3.3.6]{EqDCap}, so that $H_n$ is an open normal subgroup of $G$ contained in $G_{\pi^n \cL}$ for all $n \geq 0$. Then also $H_n \leq G_{\pi^n \cN}$ for all $n \geq 0$ by Lemma \ref{GLGN}, so we may form the crossed products 
\[T_n := U_n \rtimes_{H_n} G\qmb{and}S_n := W_n \rtimes_{H_n}G\]
as in Notation \ref{WUST} above; we have omitted the trivialisations $\beta_n : H_n \to U_n^\times$ and $\overline{\beta_n} : H_n \to W_n^\times$ in the notation in the interests of clarity. Then by \cite[Lemma 3.3.4]{EqDCap} we have Fr\'echet-Stein presentations
\[ \w\cD(\bX,G) \cong \invlim T_n \qmb{and} \w\cD(\bY,G) \cong \invlim S_n.\]

\begin{lem}\label{LocalTransBimod} $\cO(\bY) \underset{\cO(\bX)}{\otimes}{} \w\cD(\bX,G)$ is a $\w\cD(\bX,G)$-coadmissible $(\w\cD(\bY,G), \w\cD(\bX,G))$-bimodule.
\end{lem}
\begin{proof} Note that $\cO(\bY) \underset{\cO(\bX)}{\otimes}{} \w\cD(\bX,G) = \w\cD(\bX,G) / I \w\cD(\bX,G)$ is finitely presented as a right $\w\cD(\bX,G)$-module, because the ideal $I \subset \cO(\bX)$ is finitely generated. Hence it is coadmissible as a right $\w\cD(\bX,G)$-module by \cite[Corollary 3.4v]{ST}. We saw in the proof of \cite[Proposition 5.3]{DCapTwo} that $\pi^n \cN$ is naturally isomorphic to $N_{\pi^n \cL}(\cI) / (\cI (\pi^n \cL))$. Hence $T_n / I T_n$ is a $(S_n,T_n)$-bimodule by Proposition \ref{TITbimodule}. Since $\frac{\w\cD(\bX,G)}{I \hsp \w\cD(\bX,G)} = \invlim T_n / I T_n$, we conclude as in the proof of \cite[Proposition 5.3]{DCapTwo} that it is a $\w\cD(\bX,G)$-coadmissible $(\w\cD(\bY,G), \w\cD(\bX,G))$-bimodule in the sense of \cite[Definition 7.3]{DCapOne}.
\end{proof}
Lemma \ref{LocalTransBimod} together with \cite[Lemma 7.3]{DCapOne} allows us to write down the equivariant pushforward functor
\[ \iota_+ : {}^r\cC_{\w\cD(\bY,G)} \longrightarrow {}^r\cC_{\w\cD(\bX,G)}, \quad\quad N \mapsto N \WO{\w\cD(\bY,G)} \left(\cO(\bY) \underset{\cO(\bX)}{\otimes}{} \w\cD(\bX,G)\right).\] 
For future use, we record the following statement.
\begin{prop}\label{BimodVariesWithG} Let $N$ be an open normal subgroup of $G$. Then there is an isomorphism of $\w\cD(\bY,G) - \w\cD(\bX,N)$-bimodules
\begin{equation}\label{HNbimod} \w\cD(\bY, G) \quad\wtimes{\w\cD(\bY, N)}\quad \frac{ \w\cD(\bX,N)}{\cI(\bX) \w\cD(\bX,N)} \quad\congs\quad \frac{ \w\cD(\bX,G)}{\cI(\bX) \w\cD(\bX,G).}\end{equation}
\end{prop}
\begin{proof} Let $H$ be an open normal subgroup of $G$ contained in $N \cap G_{\cL}$, and write $W := \hK{U(\cN)}$, $U := \hK{U(\cL)}$ and $I := \cI(\bX)$ as above. 

Note that  $(U \rtimes_HN) / I(U \rtimes_HN)$ is naturally a sub $(W \rtimes_HN, U \rtimes_HN)$-bimodule of $(U \rtimes_H G)/I(U \rtimes_HG)$. This gives us a natural map $\theta_{H,G}$ of $(W \rtimes_HG, U \rtimes_HN)$-bimodules which appears as the top horizontal map of the following diagram:
\[ \xymatrix{ (W \urtimes{H} G)  \utimes{W \urtimes{H} N} \frac{ \left(U \urtimes{H} N\right) }{ I \left( U \urtimes{H} N\right) } \ar[rr]^{\theta_{H,G}} & & \frac{ \left(U \urtimes{H} G\right)}{I \left(U \urtimes{H} G\right)}  \\
(W \urtimes{H} G)  \utimes{W \urtimes{H} N} \left( (W \urtimes{H}N) \utimes{W} \frac{U}{IU} \right) \ar[rr]_(0.6){\cong} \ar[u]_{\cong}^{1 \otimes \xi_{H,N}}&& (W \urtimes{H} G) \utimes{W} \frac{U}{IU}.\ar[u]^{\cong}_{\xi_{H,G}}}\]
Here the vertical isomorphism on the left comes from Proposition \ref{SUbimod} applied to the groups $H \leq N$, the vertical isomorphism on the right comes from Proposition \ref{SUbimod} applied to the groups $H \leq G$ and the horizontal isomorphism on the bottom comes from contracting tensor products. It is straightforward to verify that the diagram is commutative, and therefore $\theta_{H,G}$ is an isomorphism.

Next, because $N$ is an open subgroup of $G$, we can assume that the open normal subgroups $H_n$ of $G$ that were chosen above just before Lemma \ref{LocalTransBimod} are all contained in $N$, by passing to a subsequence.  Then the above produces isomorphisms of $(W_n \rtimes_{H_n}G, U_n \rtimes_{H_n}N)$-bimodules
\[ \theta_{H_n, G} \quad:\quad (W_n \urtimes{H_n} G) \utimes{W_n \urtimes{H_n} N} \frac{ (U_n \urtimes{H_n} N)}{I  (U_n \urtimes{H_n} N)} \quad\congs\quad  \frac{ (U_n \urtimes{H_n} G)}{I  (U_n \urtimes{H_n} G)}\]
that are compatible with variation in $n$.  Passing to the limit as $n\to \infty$ we obtain the required isomorphism (\ref{HNbimod}). \end{proof}
Next, we extend \cite[Definition 5.5]{DCapTwo} to an appropriate more general setting.

\begin{defn} Let $A$ be a commutative Fr\'echet algebra, let $M$ be a Fr\'echet $A$-module and let $m \in M$. We say that $s \in A$ \emph{acts topologically nilpotently on $m$} if $m s^k \to 0$ as $k \to \infty$. If $S \subseteq A$, we let $M_\infty(S)$ denote the subset of $M$ consisting of all vectors $m \in M$ such that each $s \in S$ acts topologically nilpotently on $m$.
\end{defn}

With this definition, \cite[Lemma 5.5, Corollary 5.5 and Proposition 5.6]{DCapTwo} are in fact valid for any Fr\'echet $\cO(\bX)$-module.

\begin{prop}\label{MinftyCoad} Let $M$ be a coadmissible right $\w\cD(\bX,G)$-module. Then $M_\infty(I)$ is also a coadmissible right $\w\cD(\bX,G)$-module.
\end{prop}
\begin{proof} This is a direct generalisation of \cite[Corollary 5.6]{DCapTwo}. The key point is that if $M_n = M \otimes_{\w\cD(\bX,G)} T_n$, then $M_\infty(I) = \invlim M_n[I]$ by \cite[Proposition 5.6]{DCapTwo}, and each $M_n[I] \cong \Hom_{T_n}(T_n / I T_n, M_n)$ is a $T_n$-submodule of $M_n$ by Proposition \ref{TITbimodule}. Now $M_n$ is a finitely generated $T_n$-module because $M$ is coadmissible, and $T_n$ is Noetherian, being a crossed product of the Noetherian ring $U_n$ --- see \cite[Corollary 4.1.10]{EqDCap} --- with the finite group $G / H_n$. So $M_n[I]$ is a closed submodule of $M_n$ by \cite[Lemma 1.2.3]{FvdPut} for each $n \geq 0$, and therefore $M_\infty(I)$ is a closed submodule of $M$. It is hence coadmissible by \cite[Lemma 3.6]{ST}.
\end{proof}

Our next result, a generalisation of \cite[Theorem 5.7]{DCapTwo}, allows us to write down a right adjoint $\iota^\natural$ to the pushforward functor $\iota_+$.

\begin{prop}\label{Ivictims} Let $M$ be a coadmissible right $\w\cD(\bX,G)$-module. Then $M[I]$ is a coadmissible right $\w\cD(\bY,G)$-module.
\end{prop}
\begin{proof}  Because $M[I] = M_\infty(I)[I]$ and $M_\infty(I) \in \cC_{\w\cD(\bX,G)}$ by Proposition \ref{MinftyCoad}, we may assume that $M = M_\infty(I)$. Let $M_n = M \otimes_{\w\cD(\bX,G)} T_n$ for each $n \geq 0$. As in the proof of \cite[Theorem 5.7]{DCapTwo}, using Proposition \ref{DpUSubmod} we see that $M = M_\infty(I)$ implies that $M_n = (M_n)_{\Dp}(F / \pi^n)$. Therefore by Theorem \ref{MainEqMKE} applied to $F/\pi^n$ and $\pi^n \cL$, the right $S_n$-module $M_n[I]$ is finitely generated, and the counit morphism
\[ \epsilon_n : M_n[I] \hsp \underset{S_n}{\otimes}\hsp \frac{T_n}{IT_n} \longrightarrow M_n\]
is an isomorphism for all $n \geq 0$. Next, the $T_{n+1}$-linear map $M_{n+1} \to M_n$ induces an $S_{n+1}$-linear map $M_{n+1}[I] \to M_n[I]$ and an $S_n$-linear map 
\[ \varphi_n : M_{n+1}[I] \underset{S_{n+1}}{\otimes}{}S_n \longrightarrow M_n[I]\]
which features in the following commutative diagram:
\[ \xymatrix{ (M_{n+1}[I] \underset{S_{n+1}}{\otimes}S_n) \underset{S_n}{\otimes} \frac{T_n}{I T_n} \ar^{\cong}[r]\ar_{\varphi_n \otimes 1}[dd]  &   (M_{n+1}[I] \underset{S_{n+1}}{\otimes} \frac{T_{n+1}}{I T_{n+1}}) \underset{T_{n+1}}{\otimes} T_n \ar^{\epsilon_{n+1} \otimes 1}[d]  \\    &  M_{n+1} \underset{T_{n+1}}{\otimes}T_n \ar[d]^{\alpha_n} \\   M_n[I] \underset{S_n}{\otimes} \frac{T_n}{I T_n} \ar[r]_{\epsilon_n} & M_n.}\]
The map $\alpha_n$ is an isomorphism because $M$ is a coadmissible $\w\cD(\bX,G)$-module, so $\varphi_n \otimes 1$ is also an isomorphism since $\epsilon_n$ and $\epsilon_{n+1}$ are isomorphisms. The functor $- \utimes{S_n} \frac{T_n}{I T_n}$ is an equivalence on ${}^r \Coh(S_n)$ by Theorem \ref{MainEqMKE}, so $\varphi_n$ is an isomorphism. Therefore $M[I] = \invlim M_n[I]$ is a coadmissible right $\w\cD(\bY,G) = \invlim S_n$-module. 
\end{proof}

We can now state and prove the second main result of $\S \ref{EqMKE}$.
\begin{thm}\label{LocalKash} Let $\bX$ be a smooth $K$-affinoid variety and let $\bY = \Sp \cO(\bX)/I$ be a smooth, Zariski closed subvariety of $\bX$. Assume that: 
\be\item$\cT(\bX)$ admits a free $\cA$-Lie lattice $\cL = \cA \partial_1 \oplus \cdots \oplus \cA \partial_d$ for some affine formal model $\cA \subset \cO(\bX)$, satisfying $[\cL,\cL] \subseteq \pi \cL$ and $\cL \cdot \cA \subseteq \pi \cA$,
\item the ideal $I$ admits a generating set $F := \{f_1,\ldots, f_r\}$ such that $\partial_i(f_j) = \delta_{ij}$ whenever $1 \leq i \leq d$ and $1 \leq j \leq r$,
\item $G$ is a compact $p$-adic Lie group which acts continuously on $\bX$ and which preserves $\bY \subset \bX$, $\cA \subset \cO(\bX)$ and $\cL \subset \cT(\bX)$.
\ee
Then the equivariant pushforward functor 
\[ \begin{array}{lcccl} \iota_+ &:& {}^r\cC_{\w\cD(\bY,G)} &\longrightarrow& \left\{M \in {}^r\cC_{\w\cD(\bX,G)} : M = M_\infty(I)\right\}\\
&& N &\mapsto &N \WO{\w\cD(\bY,G)} \left(\cO(\bY) \underset{\cO(\bX)}{\otimes}{} \w\cD(\bX,G)\right)\end{array}\] 
is an equivalence, with inverse $\iota^\natural : M \mapsto M[I]$.
\end{thm}
\begin{proof} The functor $\iota^\natural$ preserves coadmissibility by Proposition \ref{Ivictims}, and using the proof of \cite[Theorem 5.9(c)]{DCapTwo} together with Theorem \ref{MainEqMKE}, we see that $\iota_+ N = (\iota_+ N)_\infty(I)$ for any $N \in {}^r \cC_{\w\cD(\bY,G)}$. Now we can use the proof of \cite[Theorem 5.9(b)]{DCapTwo} and the universal property of $\w\otimes$ to see that $\iota^\natural$ is right adjoint to $\iota_+$.  To show that this adjunction is in fact an equivalence, by \cite[Proposition 4.10]{DCapTwo} it will suffice to show that the counit is an isomorphism and that $\iota_+$ reflects isomorphisms. But both of these statements follow from Theorem \ref{MainEqMKE}. \end{proof}

For future use, we record the equivariant analogue of \cite[Theorem 6.9]{DCapTwo}.

\begin{prop}\label{MIX} Let $\bX, \bY, \cA, \cL, F$ and $G$ be as in Theorem \ref{LocalKash}. Let $\bX'$ be a $G$-stable affinoid subdomain of $\bX$ and let $\bY' := \bY \cap \bX'$. Then the natural map
\[ M[\sI(\bX)] \wtimes{\w\cD(\bY,G)} \w\cD(\bY',G) \quad\longrightarrow\quad \left(M \wtimes{\w\cD(\bX,G)} \w\cD(\bX',G)\right)[\sI(\bX')]\]
is an isomorphism for every $M \in {}^r \cC_{\w\cD(\bX,G)}$ such that $M = M_\infty(\sI(\bX))$. 
\end{prop}
\begin{proof} By applying \cite[Proposition 7.6]{DCapOne} and \cite[Lemma 4.3.5]{EqDCap}, we may assume that $\cO(\bX')$ contains an $\cL$-stable and $G$-stable affine formal model $\cA'$; then $\cL' := \cA' \otimes_{\cA} \cL$ is a $G$-stable $\cA'$-Lie lattice in $\cT(\bX')$.  Let $U_n := \hK{U(\pi^n \cL)}$ and $U'_n := \hK{U(\pi^n\cL')}$ for each $n \geq 0$  and note that $G_{\pi^n \cL} \leq G_{\pi^n \cL'}$ by \cite[Proposition 4.3.6(b)]{EqDCap}. Choose a good chain $(H_\bullet)$ for $\cL$ in $H$ and for each $n \geq 0$ consider the ring maps $T_n := U_n \rtimes_{H_n} H \to T'_n := U'_n \rtimes_{H_n} H$, so that $T := \w\cD(\bX,G) = \invlim T_n$ and $T' := \w\cD(\bX',G) = \invlim T'_n$ by \cite[Lemma 3.3.4]{EqDCap}. Define $S := \w\cD(\bY,G)$ and $S' := \w\cD(\bY',G)$ so that $S = \invlim S_n$ and $S' = \invlim S'_n$ for appropriate crossed products $S_n$ and $S'_n$ as in Proposition \ref{LocBanIvict}. Let $I := \sI(\bX)$.  Then $(M \wtimes{T} T')[\sI(\bX')] = (M \wtimes{T} T')[I]$ because $\sI(\bX') = \cO(\bX') \cdot I$, and the map in question
\[ M[I] \wtimes{S} S' \quad\longrightarrow (M \wtimes{T} T')[I] \]
is the inverse limit of the natural maps
\[ M_n[I] \utimes{S_n} S'_n \quad\longrightarrow (M_n \utimes{T_n} T'_n)[I]\]
where $M_n := M \otimes_T T_n$ for each $n \geq 0$. As in the proof of \cite[Theorem 5.7]{DCapTwo}, using Proposition \ref{DpUSubmod}, our assumption $M = M_\infty(I)$ implies that $M_n = (M_n)_{\Dp}(F / \pi^n)$ for each $n \geq 0$.  So these maps are isomorphisms by Proposition \ref{LocBanIvict}. 
\end{proof}
\subsection{The equivariant Kashiwara equivalence} Let $\bY$ be a smooth, Zariski closed subset of the smooth rigid analytic space $\bX$ defined by the vanishing of a radical, coherent ideal $\sI$ of $\cO_{\bX}$, and let $\iota : \bY \hookrightarrow \bX$ be the inclusion of $\bY$ into $\bX$. Let $G$ be a $p$-adic Lie group acting continuously on $\bX$ and stabilising $\bY$. 

We begin by constructing the equivariant pushforward functor 
\[ \iota_+ : {}^r\cC_{\bY/G} \quad\longrightarrow {}^r\cC_{\bX/G}\]
for $G$-equivariant right $\cD$-modules. We will work with the following slightly more restrictive version of the basis for the topology on $\bX$ introduced at \cite[\S 6.3]{DCapTwo}.

\begin{defn} \label{BasisB} \hsp \be 
\item Let $\cB$ denote the set of connected affinoid subdomains $\bU$ of $\bX$ such that 
\begin{enumerate}[{(}i{)}]
\item$\cT(\bU)$ admits a free $\cA$-Lie lattice $\cL = \cA \partial_1 \oplus \cdots \oplus \cA \partial_d$ for some affine formal model $\cA \subset \cO(\bU)$, satisfying $[\cL,\cL] \subseteq \pi \cL$ and $\cL \cdot \cA \subseteq \pi \cA$, 
\item either $\cI(\bU) = \cI(\bU)^2$, or $\cI(\bU)$ admits a generating set $\{f_1,\ldots, f_r\}$ such that $\partial_i(f_j) = \delta_{ij}$ whenever $1 \leq i \leq d$ and $1 \leq j \leq r$.
\ee
\item For each $\bU \in \cB$ we say that a compact open subgroup $H$ of $G$ is \emph{$\bU$-good} if for some choice of the data $\cA \subset \cO(\bU)$ and $\cL \subset \cT(\bU)$ satisfying the conditions in (a), $H$ stabilises $\bU \subset \bX$, $\cA \subset \cO(\bU)$ and $\cL \subset \cT(\bU)$.
\ee\end{defn}
It is clear that if $\bU \in \cB$ then $g \bU \in \cB$ for all $g \in G$.
\begin{lem}\label{GoodBasis} \hsp \be 
\item $\cB$ is a basis for the topology on $\bX$. 
\item For each $\bU \in \cB$ we can find at least one $\bU$-good subgroup $H$ of $G$.
\item Let $\bU \in \cB$ and let $H$ be $\bU$-good. Then $\cN(\bU \cap \bY)$ is a coadmissible right $\w\cD(\bU \cap \bY, H)$-module for every $\cN \in {}^r \cC_{\bY/G}$.
\ee\end{lem}
\begin{proof} (a) Because $\bX$ and $\bY$ are both smooth, the second fundamental sequence is exact by \cite[Proposition 2.5]{BLR3} so the normal bundle $\cN_{Y/X}$ is locally free and the canonical map $\iota^\ast \cT \to \cN_{Y/X}$ is surjective.  Now \cite[Theorem 6.2]{DCapTwo} tells us that we can find an admissible affinoid covering $\{\bX_\alpha\}$ for $\bX$ such that for each $\alpha$, either $\cI(\bX_\alpha) = \cI(\bX_\alpha)^2$ or there is an $\cO(\bX_\alpha)$-module basis $\{\partial_1,\ldots,\partial_d\}$ for $\cT(\bX_\alpha)$ and a generating set $\{f_1,\ldots,f_r\}$ for $\cI(\bX_\alpha)$ with $1 \leq r \leq d$ such that $\partial_i \cdot f_j = \delta_{ij}$ whenever $1 \leq i \leq d$ and $1 \leq j \leq r$. Because $\bX$ is smooth, we can replace each $\bX_\alpha$ by its finite set of connected components and thereby arrive at an admissible covering with the same properties, but where in addition every $\bX_\alpha$ is connected. Let $\cA_\alpha := \cO(\bX_\alpha)^\circ$ for each $\alpha$ and let $\cL_\alpha := \oplus_{i=1}^d \pi \cA_\alpha \partial_i$. Then both conditions (i) and (ii) above are satisfied.

(b) Because $G$ acts continuously on $\bX$, the stabiliser $H_1$ of $\bU$ in $G$ is open by \cite[Definition 3.1.8(a)]{EqDCap} and any affine formal model $\cA \subset \cO(\bU)^\circ$ is stabilised by an open subgroup $H_2$ of $H_1$ by \cite[Definition 3.1.8(b)]{EqDCap}. Then every $\cA$-Lie lattice $\cL \subset \cT(\bU)$ is stabilised by an open subgroup $H_3$ of $H_2$ by \cite[Lemma 3.2.8(b)]{EqDCap}, and we may take $H$ to be any compact open subgroup of $H_3$.

(c) It follows from Lemma \ref{Normaliser} that the pair $(\bU \cap \bY, H)$ is small in the sense of \cite[Definition 3.4.4]{EqDCap}. Now apply the right-module version of \cite[Theorem 4.4.3]{EqDCap}.
\end{proof}
Using Lemma \ref{GoodBasis}(c), Lemma \ref{LocalTransBimod} and the right-module version of \cite[Lemma 7.3]{DCapOne}, we can now make the following construction.
\begin{defn}\label{M[UH]} Let $\cN \in {}^r \cC_{\bY/G}$, let $\bU \in \cB$ and let $H$ be $\bU$-good.  We define
\[ M[\bU,H] \quad := \quad \cN(\bU \cap Y) \quad\wtimes{\w\cD(\bU \cap Y, H)}\quad \frac{\w\cD(\bU,H)}{\cI(\bU) \w\cD(\bU,H)},\]
this is a coadmissible right $\w\cD(\bU,H)$-module.
\end{defn}
This construction is functorial in the following sense.
\begin{lem}\label{MUHfunc} Let $\bV \subseteq \bU$ be members of $\cB$ and let $H \leq N$ be compact open subgroups of $G$ which are $\bU$-good and $\bV$-good. 
\be 
\item  There is a natural commutative diagram of right $\w\cD(\bU,H)$-modules 
\begin{equation}\label{MUHbifunc} \xymatrix{ M[\bU,H] \ar[r]\ar[d] & M[\bU,N] \ar[d] \\ M[\bV,H] \ar[r] & M[\bV,N]. }\end{equation}
\item For every $g \in G$ there is a $K$-linear map
\[ g^M_{\bU,H} : M[\bU,H] \longrightarrow M[g\bU, gHg^{-1}]\]
such that for every $a \in \w\cD(\bU,H)$ and every $m \in M[\bU,H]$, we have 
\[ g^M_{\bU,H}(m \cdot a) = g^M_{\bU,H}(m) \cdot \w{g}_{\bU,H}(a).\]
\item The following diagram is commutative for all $g \in G$:
\begin{equation}\label{MUHequiv}
 \xymatrix{  M[\bU,H] \ar[rr]^{g^M_{\bU,H}} && M[g\bU, gHg^{-1}] \\
  M[\bU,N] \ar[rr]^{g^M_{\bU,N}}\ar[d]\ar[u] && M[g\bU, gNg^{-1}] \ar[d]\ar[u] \\ 
  M[\bV,N] \ar[rr]_{g^M_{\bV,N}} && M[g\bV, gNg^{-1}].  } 
 \end{equation}
\ee
\end{lem}
\begin{proof} We give the construction of $g^M_{\bU,H}$ and leave the rest to the reader. By construction, the $K$-algebra homomorphism $\w{g}_{\bU,H} : \w\cD(\bU,H) \to \w\cD(g\bU, gHg^{-1})$ from \cite[Lemma 3.4.3]{EqDCap} restricts to $g^{\cO}(\bU) : \cO(\bU) \to \cO(g \bU)$ which sends $\cI(\bU)$ to $\cI(g\bU)$ because $\bY$ is assumed to be $G$-stable. So the map 
\[ \overline{\w{g}_{\bU,H}} : \frac{ \w\cD(\bU,H)}{\cI(\bU) \w\cD(\bU,H)} \quad\longrightarrow\quad  \frac{ \w\cD(g\bU,gHg^{-1})}{\cI(g\bU) \w\cD(g\bU,gHg^{-1})}\]
is well-defined and we can set $g^M_{\bU,H} := g^{\cN}(\bU \cap \bY) \quad \w\otimes \quad \overline{\w{g}_{\bU,H}}$.
\end{proof}
Using the connecting maps $M[\bU,H] \to M[\bU,N]$, we can make the following 
\begin{defn}\label{i+OnB} Let $\cN \in {}^r \cC_{\bY/G}$ and let $\bU \in \cB$. We define
\[ (\iota_+ \cN)(\bU) := \lim\limits_H M[\bU,H]\]
where the limit is taken over all $\bU$-good subgroups $H$ of $G$.
\end{defn}
\begin{lem}\label{i+NGpre} $\iota_+ \cN$ is a $G$-equivariant presheaf of $\cD$-modules on $\cB$.
\end{lem}
\begin{proof} This follows from Lemma \ref{MUHfunc} --- see the proof of \cite[Theorem 3.5.8]{EqDCap} for more details in a similar situation.
\end{proof}

\begin{prop}\label{MUHiso} Let $\bU \in \cB$, let $H$ be a $\bU$-good subgroup of $G$ and let $N$ be an open normal subgroup of $H$. Then for every $\cN \in {}^r \cC_{\bY/G}$, the canonical map $M[\bU, N] \to M[\bU,H]$ is an isomorphism.
\end{prop}
\begin{proof} Let $\bV := \bU \cap \bY$. By Proposition \ref{BimodVariesWithG}, there is an isomorphism
\[\w\cD(\bV, H) \quad\wtimes{\w\cD(\bV, N)}\quad \frac{ \w\cD(\bU,N)}{\cI(\bU) \w\cD(\bU,N)} \quad \congs \quad \frac{ \w\cD(\bU,H)}{\cI(\bU) \w\cD(\bU,H)}\]
of $\w\cD(\bV, H)-\w\cD(\bU,N)$-bimodules. Now apply the functor $\cN(\bV) \w\otimes_{\w\cD(\bV,H)}-$ to this isomorphism to see that the arrow $M[\bU,N] \to M[\bU,H]$ is bijective. \end{proof}
\begin{cor}\label{I+Coadm} Let $\bU \in \cB$, let $H$ be a $\bU$-good open subgroup of $G$ and let $\cN \in {}^r \cC_{\bY/G}$. Then the canonical map $(\iota_+\cN)(\bU) \to M[\bU,H]$ is an isomorphism, and $(\iota_+\cN)(\bU)$ is a coadmissible right $\w\cD(\bU,H)$-module. \end{cor}
\begin{proof} Let $J$ be an open subgroup of $H$ and choose an open normal subgroup $N$ of $H$ contained in $J$. The maps $M[\bU,N] \to M[\bU,H]$ and $M[\bU,N] \to M[\bU,J]$ are isomorphisms by Proposition \ref{MUHiso}. Therefore $M[\bU,J] \to M[\bU, H]$ is also an isomorphism for every open subgroup $J$ of $H$, and the result follows.
\end{proof}
\begin{prop}\label{i+Loc}  Let $\cN \in {}^r\cC_{\bY/G}$, let $\bU \in \cB$ and let $H$ be a $\bU$-good open subgroup of $G$. Then there is a natural continuous $H$-$\cD$-linear isomorphism 
\[ \cP^{\w\cD(\bU,H)}_{\bU}((\iota_+\cN)(\bU)) \congs (\iota_+\cN)_{|\bU_w}.\]
\end{prop}
\begin{proof} Note that $(\iota_+\cN)(\bU) \cong M[\bU,H]$ is a coadmissible right $\w\cD(\bU,H)$-module by Corollary \ref{I+Coadm}, so it will be enough to exhibit a continuous $H$-$\cD$-linear isomorphism
\[ \varphi : \cP^{\w\cD(\bU,H)}_{\bU}(M[\bU,H]) \congs (\iota_+\cN)_{|\bU_w}.\]
Let $\bV$ be an open affinoid subdomain of $\bU$ and let $J$ be an open subgroup of $H_{\bV}$. Then there are natural right $\w\cD(\bV,J)$-linear isomorphisms
\[ \begin{array}{lll} M[\bU,J] \wtimes{\w\cD(\bU,J)} \w\cD(\bV,J) &\cong& \cN(\bU \cap \bY) \wtimes{\w\cD(\bU \cap \bY,J)} \left( \frac{\w\cD(\bU,J)}{\cI(\bU)\w\cD(\bU,J)} \wtimes{\w\cD(\bU,J)} \w\cD(\bV,J)\right) \\
&\cong&\cN(\bU\cap\bY) \wtimes{\w\cD(\bU\cap\bY,J)} \frac{\w\cD(\bV,J)}{\cI(\bV)\w\cD(\bV,J)} \\
&\cong&\left(\cN(\bU \cap \bY)\wtimes{\w\cD(\bU\cap\bY,J)}\w\cD(\bV\cap\bY,J)\right) \wtimes{\w\cD(\bV\cap\bY,J)} \frac{\w\cD(\bV,J)}{\cI(\bV)\w\cD(\bV,J)} \\
&\cong&\cN(\bV\cap\bY) \wtimes{\w\cD(\bV\cap\bY)} \frac{\w\cD(\bV,J)}{\cI(\bV)\w\cD(\bV,J)}=M[\bV,J]\end{array}\]
where the isomorphism on the last line follows from the isomorphism 
\[{}^r\Loc^{\w\cD(\bU\cap\bY)}_{\bU\cap \bY}\left(\cN(\bU \cap\bY)\right) \cong \cN_{|\bU \cap \bY}\]
given by the right-module version of \cite[Theorem 4.4.3]{EqDCap} and by the fact that $\cN \in {}^r\cC_{\bY/G}$. Because the canonical map $M[\bU,J]\to M[\bU,H]$ is an isomorphism by Proposition \ref{MUHiso}, we obtain a right $\w\cD(\bV,J)$-linear isomorphism
\[\varphi(\bV,J) : M[\bU,H]\wtimes{\w\cD(\bU,J)} \w\cD(\bV,J)  \quad\stackrel{\cong}{\longrightarrow}\quad M[\bV,J].\]
Denote the left hand side by $M[\bU,H](\bV,J)$ following \cite[Definition 3.5.1]{EqDCap}. We leave it to the reader to verify that the following diagram 
\[\xymatrix@C=20pt{ &M[\bU,H](g\bV,gJg^{-1}) \ar[rr]^(0.55){\varphi(\bV, gJg^{-1})}&& M[g\bV,gJg^{-1}]& \\
M[\bU,H](\bV,J) \ar[rrrr]^{\varphi(\bV,J)}\ar[dr]\ar@/_2pc/[ddr]\ar[ur] &&&& M[\bV,J] \ar[ul]\ar[dl]\ar@/^2pc/[ddl] \\
&M[\bU,H](\bW,J) \ar[rr]_{\varphi(\bW,J)} && M[\bW,J]& \\
&M[\bU,H](\bV,N) \ar[rr]_{\varphi(\bV,N)} && M[\bV,N]&}\]
is commutative, for any $H$-stable open affinoid subdomain $\bW$ of $\bV$, any open subgroup $N$ of $H_{\bV}$ containing $J$, and any $g \in G$. In view of the right-module version of \cite[Definition 3.5.1]{EqDCap} and Definition \ref{M[UH]}, passing to the inverse limit over all open $J \leq H_{\bV}$ induces a continuous right $\cD(\bV)$-linear isomorphism
\[\varphi(\bV) : \cP^{\w\cD(\bU,H)}_{\bU}(M[\bU,H])(\bV) \quad \stackrel{\cong}{\longrightarrow}\quad (\iota_+\cN)(\bV).\]
The commutativity of the diagram above implies that these maps patch together to define the required continuous $H$-$\cD$-linear isomorphism 
\[ \varphi : \cP^{\w\cD(\bU,H)}_{\bU}(M[\bU,H]) \congs (\iota_+\cN)_{|\bU_w}. \qedhere\]
\end{proof}

\begin{cor}\label{i+NsheafonB} $\iota_+\cN$ is a sheaf on $\cB$ whenever $\cN \in {}^r \cC_{\bY/G}$. 
\end{cor}
\begin{proof} Use \cite[Theorem 3.5.11]{EqDCap}, Corollary \ref{I+Coadm} and Proposition \ref{i+Loc}.
\end{proof}
Using \cite[Theorem 9.1]{DCapOne}, we can make the following
\begin{defn}\label{DefnOfi+N} We define $\iota_+\cN$ to be the unique extension of the sheaf $\iota_+\cN$ on $\cB$ to a $G$-equivariant sheaf of $\cD$-modules on $\bX_{\rig}$.
\end{defn}
We can now record the equivariant generalisation of \cite[Proposition 6.4]{DCapTwo}.
\begin{prop}\label{i+CoadmFunc} $\cN \mapsto \iota_+\cN$ defines a functor $\iota_+ : {}^r\cC_{\bY/G} \longrightarrow {}^r\cC^{\bY}_{\bX/G}$.
\end{prop}
\begin{proof} By Lemma \ref{GoodBasis}(a) we can find an admissible covering $\{\bX_j\}$ for $\bX$ by members of $\cB$, and by Lemma \ref{GoodBasis}(b), for each $j$ we can choose a $\bX_j$-good subgroup $G_j$ of $G$. Let $\bY_j := \bX_j \cap \bY$; then $(\bY_j,G_j)$ is small so by \cite[Theorem 4.4.3]{EqDCap} for each $j$ there is a continuous right $\cD-G_j$-linear isomorphism
\[ \cN_{|\bY_j} \quad \cong\quad {}^r\Loc^{\w\cD(\bY_j,G_j)}_{\bY_j} \cN(\bY_j).\]
Proposition \ref{i+Loc} now implies that for each $j$ there is an isomorphism of locally Fr\'echet $H_j$-equivariant right $\cD$-modules
\[ (\iota_+\cN)_{|\bX_j} \quad\cong\quad {}^r\Loc^{\w\cD(\bX_j,G_j)}_{\bX_j} \left((\iota_+\cN)(\bX_j)\right)\]
In view of Lemma \ref{i+NGpre}, Corollary \ref{i+NsheafonB} and Definition \ref{DefnOfi+N}, this means that $\iota_+\cN \in {}^r \cC_{\bX/G}$. The functorial nature of $\iota_+$ is clear. Finally, note that if $\bU$ is an affinoid subdomain of $\bX$ such that $\bU \cap \bY = \emptyset$, then $\cI(\bU) = \cO(\bU)$ and it follows directly from Definition \ref{M[UH]} that $M[\bU,H] = 0$ for any $\bU$-good subgroup $H$. Therefore $(\iota_+\cN)(\bU) = 0$ for any such $\bU$ by Definition \ref{i+OnB}, which means that the restriction of $\iota_+\cN$ to $\bX -\bY$ is zero. Hence $\iota_+\cN \in {}^r \cC_{\bX/G}^{\bY}$ as claimed.
\end{proof}

Next, we construct the equivariant pullback functor 
\[ \iota^\natural : {}^r \cC^{\bY}_{\bX/G} \quad\longrightarrow \quad {}^r\cC_{\bY/G}.\]
Recall from \cite[\S A.1]{DCapTwo} that $\iota^{-1}$ denotes the pullback functor from abelian sheaves on $\bX_{\rig}$ to abelian sheaves on $\bY_{\rig}$. 
\begin{defn} Let $\cM \in {}^r \cC^{\bY}_{\bX/G}$. We define $\iota^\natural\cM := \iota^{-1}( \iota^\natural_\ast \cM )$ where $\iota^\natural_\ast\cM$ is the extension to $\bX_{\rig}$ of the sheaf $\iota^\natural_\ast\cM$ on $\cB$ given by
\[(\iota^\natural_\ast\cM)(\bU) := \cM(\bU)[\sI(\bU)] \qmb{for all} \bU \in \cB.\]
\end{defn}
Note that $\iota^\natural_\ast\cM$ is a presheaf of $\iota_\ast \cD_{\bY}$-modules on $\cB$; the proof of \cite[Lemma 6.5]{DCapTwo} shows that it is in fact a sheaf. Because $\iota_\ast \cD_{\bY}$ is supported on $\bY$, $\iota^\natural_\ast\cM$ is also supported on $\bY$ and by \cite[Theorem A.1]{DCapTwo} there is a natural isomorphism
\[ \iota_\ast(\iota^\natural\cM) \stackrel{\cong}{\longrightarrow} \iota^\natural_\ast \cM.\]
We will need the following generalisation of \cite[Theorem 6.7]{DCapTwo}.
\begin{prop}\label{RecognitionOfSupport} Suppose that $\bX \in \cB$ and that $G$ is $\bX$-good. Then the following are equivalent for $\cM \in {}^r\cC_{\bX/G}$:
\be\item $\cM$ is supported on $\bY$,
\item $\cM(\bX) = \cM(\bX)_\infty([\sI(\bX)])$.
\ee\end{prop}
\begin{proof} The proof of \cite[Theorem 6.7]{DCapTwo} reduces us \footnote{use \cite[Theorem 4.4.3]{EqDCap} in place of \cite[Theorem 9.4]{DCapOne}} to showing that for any non-zero $f \in \sI(\bX)$, the module $\cM(\bX(1/f))$ is zero if and only if $f$ acts locally topologically nilpotently on $\cM(\bX)$. Our assumptions on $\bX$ and $G$ allow us to choose a $G$-stable affine formal model $\cA \subset \cO(\bX)$ and a $G$-stable $\cA$-Lie lattice $\cL \subset \cT(\bX)$. By rescaling $\cL$ and applying \cite[Lemma 7.6(a)]{DCapOne}, we may assume that $\cL \cdot f \subset \cA$. 

Let $\bX' := \bX(1/f)$, let $H := G_{\bX'}$, write $T := \w\cD(\bX,H)$ and $T' := \w\cD(\bX',H)$. Let $M := \cM(\bX) \in {}^r\cC_T$ and $M' := \cM(\bX') \in {}^r\cC_{T'}$ so that $M' \cong M \w\otimes_{T}T'$ by \cite[Theorem 4.4.3]{EqDCap}. Let $\cA' := \cA \langle 1/f \rangle$ and $\cL' := \cA' \otimes_{\cA} \cL$ and set $U_n := \hK{U(\pi^n \cL)}$ and $U'_n := \hK{U(\pi^n\cL')}$ for each $n \geq 0$. Choose a good chain $(H_\bullet)$ for $\cL$ in $H$ and set $T_n := U_n \rtimes_{H_n} H$ and $T'_n := U'_n \rtimes_{H_n} H$ for each $n \geq 0$, so that $T = \invlim T_n$ and $T' = \invlim T'_n$ by \cite[Lemma 3.3.4]{EqDCap}. Let $M_n := M \otimes_T T_n$ and $M'_n := M' \otimes_{T'}T'_n$ for each $n \geq 0$. Using the isomorphism $T'_n \cong T_n \otimes_{U_n} U'_n$ of $T_n-U'_n$-bimodules coming from the right-module version of \cite[Proposition 4.3.11]{EqDCap}, we obtain
\[M'_n \cong (M \wtimes{T}T')\utimes{T'} T'_n = M \utimes{T} T'_n \cong  M_n \utimes{T_n} T'_n \cong M_n \utimes{U_n} U'_n \qmb{for all} n \geq 0.\]
Now we argue as in the proof of \cite[Corollary 6.6]{DCapTwo}: $M' = \cM(\bX(1/f))$ is zero if and only if $M'_n \cong M_n \otimes_{U_n} U'_n$ is zero for all $n \geq 0$, if and only if $f$ acts locally topologically nilpotently on $M_n$ \footnote{by \cite[Proposition 6.6]{DCapTwo}} for all $n \geq 0$, if and only if $f$ acts locally topologically nilpotently on $\invlim M_n$. But $\cM(\bX) = M \cong \invlim M_n$ by \cite[Corollary 3.3]{ST}. \end{proof}

We can now construct our equivariant pullback functor.
\begin{thm}\label{InatCoadm} Let $\cM \in {}^r \cC^{\bY}_{\bX/G}$. Then $\iota^\natural \cM \in {}^r \cC_{\bY/G}$, and this defines a functor
\[ \iota^\natural : {}^r \cC^{\bY}_{\bX/G} \quad\longrightarrow \quad {}^r\cC_{\bY/G}.\]
\end{thm}
\begin{proof} The proof of \cite[Lemma 6.5]{DCapTwo} shows that $(\iota^\natural_\ast\cM)(\bU) := \cM(\bU)[\sI(\bU)]$ holds for every affinoid subdomain $\bU$ of $\bX$. Because $\cM \in {}^r \cC_{\bX/G}$ is locally Fr\'echet,  $\cM(\bU)$ is a Fr\'echet space for each $\bU \in \bX_w(\cT)$. Since $\sI(\bU)$ is finitely generated as an $\cO(\bU)$-module, it follows that $(\iota^\natural_\ast\cM)(\bU)$ is a closed subspace of $\cM(\bU)$ and therefore itself carries a natural Fr\'echet topology; thus $\iota^\natural_\ast\cM$ is a locally Fr\'echet $\cD$-module on $\bX$ in the sense of \cite[Definition 3.6.1]{EqDCap} applied with $G=1$.

Let $g \in G$ and let $\bU \in \bX_w(\cT)$. If $a \in \sI(g \bU)$ then $(g^{\cO})^{-1}(a) \in \sI(\bU)$ because $\bY$ is $G$-stable, so $a \cdot g^\cM(m) = g^{\cM}( (g^{\cO})^{-1}(a) \cdot m) = 0$ for all $m \in \cM(\bU)[\sI(\bU)]$. Thus $g^{\cM}$ maps $\cM(\bU)[\sI(\bU)]$ into $\cM(g \bU)[\sI(g \bU)]$, and we can therefore define $g^{\iota^\natural_\ast\cM}(\bU)$ to be the restriction of $g^\cM(\bU)$ to $(\iota^\natural_\ast\cM)(\bU) \subset \cM(\bU)$. Being the restriction of the continuous map $g^{\cM}(\bU)$, $g^{\iota^\natural_\ast\cM}(\bU)$ is continuous. We leave it to the reader to check that in this way $\iota^\natural_\ast \cM$ becomes a $G$-equivariant $\iota_\ast\cD$-module on $\bX$. It follows that $\iota^\natural\cM$ is a locally Fr\'echet $G$-equivariant $\cD$-module on $\bY$.

Because of the local nature of ${}^r\cC_{\bX/G}$ --- see \cite[Definition 3.6.7]{EqDCap} --- as in the proof of \cite[Theorem 6.10(a)]{DCapTwo}, we can now reduce to the case where $\bX \in \cB$ and $G$ is $\bX$-good in order to show that $\iota^\natural \cM \in {}^r \cC_{\bY/G}$. Let $T := \w\cD(\bX,G)$, $S := \w\cD(\bY,G)$, $M := \cM(\bX)$ and $I := \sI(\bX)$; then $M \in {}^r\cC_T$ by \cite[Theorem 4.4.3]{EqDCap} and therefore $M[I] \in {}^r \cC_S$ by Proposition \ref{Ivictims}. Let $\bX'$ be an affinoid subdomain of $\bX$, write $\bY' := \bX' \cap \bY$ and choose any open subgroup $H$ of $G_{\bX'}$. Because $\cM$ is supported on $\bY$, we know that $M = M_\infty(I)$ by Proposition \ref{RecognitionOfSupport}. Then by Proposition \ref{MIX}, the natural right $\w\cD(\bY',H)$-linear map 
\[ M[\sI(\bX)] \wtimes{\w\cD(\bY,H)} \w\cD(\bY',H) \quad\longrightarrow\quad \left(M \wtimes{\w\cD(\bX,H)} \w\cD(\bX',H)\right)[\sI(\bX')]\]
is an isomorphism. These maps are compatible with variation in $H \leq G_{\bX'}$; passing to the limit over all such open subgroups $H$ induces a $\cD(\bX')$-linear isomorphism
\[  {}^r\cP^S_{\bY}(M[I])(\bY') \tocong {}^r\cP^{T}_{\bX}(M)(\bX')[ \sI(\bX') ]\]
which is automatically continuous by \cite[Lemma 3.6.5]{EqDCap}. We leave it to the reader to verify that these maps are in turn compatible with restrictions to affinoid subdomains $\bX'' \subset \bX'$ as well as with the $G$-action. They therefore assemble to give a right $\cD$-$G$-linear isomorphism of sheaves on $\bX_w$
\[  \iota_\ast {}^r\cP^S_{\bY}(M[I]) \tocong {}^r\cP^{T}_{\bX}(M)[ \sI(-) ], \]
and then by \cite[Theorem 9.1]{DCapOne}, to a continuous right $\cD_{\bX}$-$G$-linear isomorphism 
\[ \iota_\ast\hsp  {}^r \Loc^S_{\bY}(M[I]) \tocong \iota_\ast \iota^\natural \hsp {}^r \Loc^T_{\bX}(M) \]
and hence by \cite[Theorem A.1]{DCapTwo}, to a continuous right $\cD_{\bY}$-$G$-linear isomorphism 
\[ \Loc^S_{\bY}(M[I]) \tocong \iota^\natural \left({}^r \Loc^T_{\bX}(M)\right).\]
Because $\cM \cong {}^r \Loc^T_{\bX}(M)$ by \cite[Theorem 4.4.3]{EqDCap}, we conclude that $\iota^\natural \cM\in {}^r \cC_{\bY/G}$. The functorial nature of $\cM \mapsto \iota^\natural \cM$ is clear. \end{proof}
Here is the Kashiwara equivalence for equivariant right $\cD$-modules.
\begin{thm}\label{KashiwaraGeneral} Let $\iota : \bY \hookrightarrow \bX$  be the inclusion of a smooth, Zariski closed subset $\bY$ into the smooth rigid analytic space $\bX$. Let $G$ be a $p$-adic Lie group acting continuously on $\bX$ and stabilising $\bY$.  Then the functors
\[\iota_+ : {}^r\cC_{\bY/G} \to {}^r\cC^{\bY}_{\bX/G} \qmb{and} \iota^\natural :  {}^r\cC^{\bY}_{\bX/G} \to {}^r\cC_{\bY/G}\]
are mutually inverse equivalences of categories.
\end{thm}
\begin{proof} The functors $\iota_+$ and $\iota^\natural$ were constructed in Proposition \ref{i+CoadmFunc} and Theorem \ref{InatCoadm}, respectively. We will first show that $\iota_+$ is left adjoint to $\iota^\natural$; to this end, fix $\cN \in {}^r\cC_{\bY/G}$ and $\cM \in {}^r\cC_{\bX/G}$, and using Lemma \ref{GoodBasis}(b), for each $\bU \in \cB$ choose a $\bU$-good open subgroup $H(\bU)$ of $G$. \cite[Theorem 9.1]{DCapOne} gives us an inclusion
\[ \Hom_{{}^r\cC_{\bX/G}}( \iota_+\cN, \cM ) \hookrightarrow \prod\limits_{\bU \in \cB} \Hom^{\cts}_{\cD(\bU) \rtimes H(\bU)}\left( (\iota_+\cN)(\bU), \cM(\bU) \right)\]
whose image consists of all tuples $(\alpha_{\bU})_{\bU \in \cB}$ that commute with the restriction and $G$-equivariant maps in $\iota_+\cN$ and $\cM$. Similarly, by \cite[Theorem A.1]{DCapTwo} and \cite[Theorem 9.1]{DCapOne} there is an inclusion
\[ \Hom_{{}^r\cC_{\bY/G}}(\cN, \iota^\natural \cM ) \hookrightarrow \prod\limits_{\bU \in \cB} \Hom^{\cts}_{\cD(\bU\cap \bY) \rtimes H(\bU)}\left( \cN(\bU\cap \bY), (\iota^\natural_\ast\cM)(\bU)\right)\]
whose image consists of all tuples $(\beta_{\bU})_{\bU \in \cB}$ that commute with the restriction and $G$-equivariant maps in $\iota_\ast\cN$ and $\iota^\natural_\ast\cM$. 

Fix $\bU \in \cB$, write $\bV:=\bU \cap \bY$, $N := \cN(\bV)$, $M := \cM(\bU)$, $H := H(\bU)$ and $I := \cI(\bU)$ so that $(\iota^\natural_\ast\cM)(\bU) = M[I]$. By Corollary \ref{I+Coadm}, there is a bijection
\begin{equation}\label{Adj1} \Hom^{\cts}_{\cD(\bU) \rtimes H}\left( (\iota_+\cN)(\bU), \cM(\bU) \right) \cong \Hom^{\cts}_{\cD(\bU) \rtimes H}\left( M[\bU,H], M \right).\end{equation}
Now $M[\bU, H]$ is a coadmissible right $\w\cD(\bU,H)$-module by Definition \ref{M[UH]}, whereas $M$ is a coadmissible right $\w\cD(\bU,H)$-module by \cite[Theorem 4.4.3]{EqDCap}, so 
\begin{equation}\label{Adj2}\Hom^{\cts}_{\cD(\bU) \rtimes H}\left( M[\bU,H], M \right) = \Hom_{\w\cD(\bU,H)}\left( M[\bU,H], M \right)\end{equation}
and similarly it follows from Proposition \ref{Ivictims} that
\begin{equation}\label{Adj3}\Hom^{\cts}_{\cD(\bU \cap \bY) \rtimes H}\left( \cN(\bU \cap \bY),  (\iota^\natural_\ast\cM)(\bU)\right) = \Hom_{\w\cD(\bV,H)}\left( N,M[I] \right).\end{equation}
Now, as in the proof of \cite[Theorem 6.10(b)]{DCapTwo}, we can define 
\[ \Phi_{\bU} : \Hom_{\w\cD(\bU,H)}\left( M[\bU,H], M \right) \quad\longrightarrow\quad \Hom_{\w\cD(\bV,H)}\left( N,M[I] \right)\]
by the formula $\Phi_\bU(\alpha)(x) = \alpha( x \w\otimes 1 )$ for all $x \in N$. Then $\Phi_\bU$ is injective because $M[\bU,H]$ is generated by the image of $N$ as a right $\w\cD(\bU,H)$-module, and the universal property of $\w\otimes$ implies that it is in fact bijective. Combining $\Phi_{\bU}$ with the bijections $(\ref{Adj1}), (\ref{Adj2})$ and $(\ref{Adj3})$ gives us a canonical bijection 
\[ \Hom^{\cts}_{\cD(\bU) \rtimes H(\bU)}\left( (\iota_+\cN)(\bU), \cM(\bU) \right) \overset{\cong}{\underset{\Psi_{\bU}}{\longrightarrow} }\Hom^{\cts}_{\cD(\bU\cap \bY) \rtimes H(\bU)}\left( \cN(\bU\cap \bY), (\iota^\natural_\ast\cM)(\bU)\right)\]
and therefore a bijection
\[\xymatrix{\prod\limits_{\bU \in \cB} \Hom^{\cts}_{\cD(\bU) \rtimes H(\bU)}\left( (\iota_+\cN)(\bU), \cM(\bU) \right) \ar[d]^{\cong}_{\Psi := \prod\limits_{\bU \in \cB} \Psi_{\bU}} \\
\prod\limits_{\bU \in \cB} \Hom^{\cts}_{\cD(\bU\cap \bY) \rtimes H(\bU)}\left( \cN(\bU\cap \bY), (\iota^\natural_\ast\cM)(\bU)\right).}\]
We leave it to the reader to check that $\Psi$ and $\Psi^{-1}$ send morphisms of $G$-equivariant sheaves to morphisms of $G$-equivariant sheaves. Thus we obtain a bijection
\[\Psi : \Hom_{{}^r\cC_{\bX/G}}( \iota_+\cN, \cM ) \tocong \Hom_{{}^r\cC_{\bY/G}}(\cN, \iota^\natural \cM ).\]
It can also be checked that $\Psi$ does not depend on the choice of $\bU$-good open subgroups $H(\bU)$, and that it is functorial in $\cN \in {}^r\cC_{\bY/G}$ and $\cM \in {}^r\cC_{\bX/G}$.

We can now use \cite[Proposition 4.10]{DCapTwo} to conclude that in fact these two functors are mutually inverse equivalences of categories. To do this, follow the proof of \cite[Theorem 6.10(c)]{DCapTwo}, replacing \cite[Theorem 6.7]{DCapTwo} by Proposition \ref{RecognitionOfSupport} and \cite[Theorem 5.9(c)]{DCapTwo} by Theorem \ref{LocalKash}.
\end{proof}

We can now finally prove the Kashiwara equivalence, Theorem \ref{MainB}, for equivariant $\cD$-modules on rigid analytic spaces.

\begin{thm}Let $\iota : \bY \hookrightarrow \bX$  be the inclusion of a smooth, Zariski closed subset $\bY$ into the smooth rigid analytic space $\bX$. Let $G$ be a $p$-adic Lie group acting continuously on $\bX$ and stabilising $\bY$.  Then the functors 
\[\iota_+ : \cC_{\bY/G} \to \cC_{\bX/G}^{\bY} , \quad\quad \cN \mapsto \mathpzc{Hom}_{\cO_\bX}\left(\Omega_{\bX} , \iota_+(\Omega_{\bY} \utimes{\cO_\bY} \cN)\right)\]
and
\[\iota^\natural : \cC_{\bX/G}^{\bY} \to \cC_{\bY/G}, \quad\quad\cM \mapsto\mathpzc{Hom}_{\cO_\bY}\left(\Omega_{\bY} , \hspace{0.05cm}\iota^\natural\hspace{0.05cm}(\Omega_{\bX} \utimes{\cO_\bX} \cM)\right)\]
are mutually inverse equivalences of abelian categories.\end{thm}
\begin{proof}This follows from Theorem \ref{MainLRswitch} and Theorem \ref{KashiwaraGeneral}, once we observe that the side-switching functors preserve the condition of being supported on $\bY$.
\end{proof}

\section{Examples of regular orbits of positive dimension}\label{ExamplesSect}

\subsection{The twisted cubic} \label{IntersectVar}

Throughout $\S \ref{IntersectVar}$, all algebraic varieties are assumed to be defined over an algebraically closed base field $k$. Following the traditional abuse of notation, we will confuse them with their corresponding sets of $k$-points. 

We assume throughout that $G$ is an affine algebraic group acting morphically on a projective variety $X$, and that $Y$ is a Zariski closed subset of $X$.

\begin{defn} We define the \emph{intersection obstruction} of $Y$ to be
\[ \cZ_Y := \{ g \in G : Y \cap g Y \neq \emptyset \}.\]
\end{defn}

It is clear that $\cZ_Y$ contains the stabiliser $\Stab_G(Y)$ of $Y$ in $G$, and $\cZ_Y = \Stab_G(Y)$ if and only if the $G$-orbit of $Y$ is regular in $X$ in the sense of Definition \ref{RegOrbit}.

\begin{lem}\label{IntObsZarClosed} $\cZ_Y$ is a Zariski closed subset of $G$.
\end{lem}
\begin{proof} Let $a : G \times X \to X$ be the action map, and let $\tilde{Y} := (G \times Y) \cap a^{-1}(Y)$. Then $\tilde{Y}$ is a Zariski closed subset of $G \times X$, and since $\tilde{Y} = \{(g,y) \in G \times Y : g\cdot y\in Y\}$ we see that the image $p(\tilde{Y})$ of $\tilde{Y}$ in $X$ under the projection map $p : G \times X \to G$ is exactly $\cZ_Y$. Since $X$ is a projective variety, its structure map $X \to \Spec(k)$ is proper by \cite[Theorem II.4.9]{Hart}, so in particular it is universally closed. So by \cite[Definition on p.100]{Hart}, the projection map $p : G \times X \to X$ is closed, and therefore $p(\tilde{Y}) = \cZ_Y$ is closed in $X$.
\end{proof}
\begin{defn} We define the \emph{core} of the intersection obstruction to be
\[ \cZ_Y^\circ := \bigcap_{g \in G} g \cZ_Y g^{-1}.\]
\end{defn}

This is a union of conjugacy classes in $G$ and one might hope this core to be just $\{1\}$, at least when the group $G$ is simple. Unfortunately, this is not true in general as our next results show.

\begin{prop}\label{LazyGroupElements} Suppose that $G = \GL_{n+1}(k)$ acting naturally on $X = \P^n(k)$ and that $\dim Y \geq 1$. Then $h \in \cZ_Y$ whenever $h \in G$ and $\rk (h - 1) \leq 1$.
\end{prop}
\begin{proof} By passing to an irreducible component of $Y$ of positive dimension, we may assume that $Y$ is irreducible. Let $W = k^{n+1}$ and $A = \Sym_k(W)$, so that $\P^n = \Proj(A)$ and $G$ acts naturally on $A$. The closed subvariety $Y$ of $X$ is the vanishing set $V(I)$ of some radical homogeneous ideal $I$ of $A$; since $Y$ is irreducible and $\dim Y \geq 1$, we see that $A/I$ is a domain of Krull dimension at least $2$. Let $h \in G$ be such that $\rk(h - 1) \leq 1$ and let $f \in W$ span the image of $h - 1$. Then $(h - 1)A \subseteq Af$, and therefore $I + hI = I + (h-1) I \subseteq I + Af.$ The image of $f$ in $A/I$ is a homogeneous element of degree one in this $\N$-graded domain. Hence $\Kdim(A / (I + Af)) \geq \Kdim(A / I) - 1 \geq 1$ by \cite[Proposition 11.3]{AMac}. Therefore the non-empty set $V(I + Af)$ is contained in $Y \cap hY = V(I + hI)$.
\end{proof}

\begin{cor}\label{TransVect} With the notation of Proposition \ref{LazyGroupElements}, the core of the intersection obstruction $\cZ_Y^\circ$ always contains $\{h \in G : \rk(h-1) \leq 1\}$.
\end{cor}
\begin{proof} By Proposition \ref{LazyGroupElements}, $\{h \in G : \rk(h-1) \leq 1\} \subseteq \cZ_{gY}$ for every $g\in G$. But $h \in \cZ_{gY} \Leftrightarrow gY \cap h gY \neq \emptyset \Leftrightarrow g^{-1}h g \in \cZ_Y \Leftrightarrow h \in g \cZ_Y g^{-1}$, so $\cZ_{gY} = g \cZ_Y g^{-1}$.
\end{proof}

We will now explicitly compute the core of the intersection obstruction $\cZ_Y^\circ$ in the case where $Y$ is the twisted cubic curve in $\P^3(k)$. Recall that $Y$ can be defined as the image of the degree $3$ Veronese embedding $\P^1(k) \to \P^3(k)$ given by $[a,b] \mapsto [a^3, a^2b, ab^2, b^3]$, or equivalently as the closed subvariety of $\P^3(k)$ cut out by the equations $x_0x_2 = x_1^2, x_1x_3 = x_2^2, x_0x_3 = x_1x_2$. Our first task will be to show that for sufficiently many matrices in $x \in G := \GL_4(k)$ in Jordan normal form, $xC \cap C = \emptyset$ for some $G$-translate $C$ of $Y$.

Let $a,b,c\in \{0,1\}$, $\alpha,\beta,\gamma \in k^\times$, and $d,r \in k$ be parameters. For these we define the following matrices in $G$:
\[x := \begin{pmatrix} 1 & 0 & 0 & 0 \\ a & \alpha & 0 & 0 \\ 0 & b & \beta & 0 \\ 0 & 0 & c & \gamma \end{pmatrix}\qmb{and} h_{d,r} := \begin{pmatrix} 1 & 0 & 0 & r \\ 0 & 1 & 0 & 0 \\ 0 & 0 & 1 & 0 \\0 & 0 & d & 1\end{pmatrix}\]
\begin{lem}\label{MovingTheCubic} Let $Y$ be the twisted cubic in $\P^3(k)$. Suppose that
\be
\item $c = 1$ whenever $\beta = \gamma$, and
\item either $a \neq 0$, or $b \neq 0$, or $\alpha \neq 1$ or $\beta \neq 1$.
\ee
Then there exist $d \in k^\times$ and $r \in k$ such that 
\[x h_{d,r} Y \cap h_{d,r} Y = \emptyset.\]
\end{lem}
\begin{proof} Fix a coordinate $t$ on $\P^1(k)$. Let $\psi_{d,r} : \P^1(k) \to \P^3(k)$ be given by $\psi_{d,r}(t) = [1,t,dt^3+t^2,t^3+r]$ if $t \neq \infty$, and $\psi_{d,r}(\infty) = [0,0,d, 1]$. Then $h_{d,r}Y$ is the image of $\psi_{d,r}$. Note that $x \psi_{d,r}(\infty) = [0,0,\beta d, cd + \gamma]$; this can only lie in $h_{d,r}Y$ if $[0,0,\beta d, cd + \gamma] = [0,0,d, 1]$, or equivalently, if $\beta d = d(cd + \gamma)$. Since $c = 1$ whenever $\beta \neq \gamma$ by assumption (a), by restricting $d$ to satisfy $d \neq 0,\beta - \gamma$, we see that $x \psi_{d,r}(\infty) \notin h_{d,r}(Y)$. Thus, if $d \neq 0, \beta - \gamma$ and $v \in x \psi_{d,r} Y \cap \psi_{d,r} Y$, then $v = x \psi_{d,r}(t) = \psi_{d,r}(s)$ for some $t,s \neq \infty$. 

We now have to show that under the given hypotheses on $a,b,c,\alpha,\beta,\gamma$, the system of simultaneous equations in $s,t \in k$
\[ \begin{array}{lll} s &=& \alpha t + a \\ ds^3 + s^2 &=& \beta( d t^3 + t^2) + bt \\ s^3 + r &=& c ( dt^3 + t^2) + \gamma(t^3 + r) \end{array}\]
has no solutions for at least one value of $(d,r)$. Substituting the first of these equations into the second gives us a polynomial equation in $t$ of degree at most $3$, whose leading coefficient is $d(\alpha^3 - \beta)$.  If $\alpha^3 \neq \beta$, choose any $d \in k \backslash\{0, \beta - \gamma\}$ to see that $t$ can take at most three values. Since $\gamma \neq 0$, we can choose $r$ in such a way that the third equation does not hold for any of these values of $t$. 

Thus we can assume that $\alpha^3 = \beta$; in this case the coefficient of $t^2$ in the second equation is $3 d \alpha^2 a - \alpha^3 + \alpha^2$. As long as $3a \neq 0$, we can choose $d$ to ensure that this coefficient is non-zero. Then $t$ can take at most two values, and again as $\gamma \neq 0$ we can choose $r$ so that the third equation fails for either of those values of $t$.  So, we may assume that $3a = \alpha^3 - \alpha^2 = 0$, i.e. $3a = 0$ and $\alpha = 1$. Then $\beta = \alpha^3 = 1$ also.

In the case $\Char(k) \neq 3$ we must have $a = 0$, so our assumption (b) implies that $b = 1$ and the second equation reduces to $t = 0$. In the case $\Char(k) = 3$, the second equation reduces to $(2a - b)t + da^3 + a^2 = 0$. If $a = 0$ then $b = 1$ by assumption and $2a - b \neq 0$, and if $a = 1$ then $2a - b\neq 0$ for either value of $b$. In all of these cases, the values of $s$ and $t$ are uniquely determined by the first two equations. We can now again choose $r$ to ensure that the third equation does not hold.
\end{proof}

\begin{prop}\label{JNF} Let $g \in G$ be such that that $\rk(\mu^{-1}g - 1) \geq 2$ for any $\mu \in k^\times$. Then we can find $a,b,c \in \{0,1\}$ and $\alpha,\beta,\gamma \in k^\times$ satisfying conditions (a) and (b) of Lemma \ref{MovingTheCubic}, and $\mu \in k^\times$, such that $\mu^{-1} g$ is $G$-conjugate to 
\[x = \begin{pmatrix} 1 & 0 & 0 & 0 \\ a & \alpha & 0 & 0 \\ 0 & b & \beta & 0 \\ 0 & 0 & c & \gamma \end{pmatrix}.\]
\end{prop}
\begin{proof} Amongst all possible eigenvalues of $g$ with largest possible algebraic multiplicity, choose $\mu$ to be an eigenvalue of largest possible geometric multiplicity. Let $y$ be the Jordan normal form of $\mu^{-1}g$ so that $y$ is $G$-conjugate to $\mu^{-1}g$. Let $a,b,c \in \{0,1\}$ be the subdiagonal entries of $y$ and let $1, \alpha, \beta, \gamma$ be the eigenvalues of $y$. We can take $x = y$ unless either condition (a) or condition (b) of Lemma \ref{MovingTheCubic} fails. Suppose for a contradiction that condition (b) fails. Then $\alpha = \beta = 1$ and $a = b = 0$. Because $y$ is in Jordan normal form, either $\gamma = 1$, or $\gamma \neq 1$ and $c = 0$. In both of these cases, $\rk( y - 1) \leq 1$, and hence $\rk(\mu^{-1} g - 1)  \leq 1$, a contradiction. 

We are left to consider the case where condition (a) fails. In this case, $\beta = \gamma$ and $c = 0$. Because $1$ is the eigenvalue of $y$ with largest possible algebraic multiplicity by construction, the diagonal entries of $y$ are necessarily $1,1,\beta, \beta$. 

Suppose first that $\beta = 1$. Because $\rk(y - 1) \geq 2$ and $c = 0$, we must have $a = b = 1$. Then $\mu^{-1}g$ is $G$-conjugate to 
\[\begin{pmatrix} 1 & 0 & 0 & 0 \\0 & 1 & 0 & 0 \\ 0 & 1 & 1 & 0 \\ 0 & 0 & 1 & 1 \end{pmatrix}\]
and this matrix satisfies the conditions of Lemma \ref{MovingTheCubic}.

Suppose now that $\beta \neq 1$. Then $b = 0$ because $y$ is in Jordan normal form. Because $c = 0$, the geometric multiplicity of $\beta$ is $2$, and because the geometric multiplicity of $\mu$ was chosen to be largest possible, we must also have $a = 0$. But then $y$ is $G$-conjugate to
\[\begin{pmatrix} 1 & 0 & 0 & 0 \\0 & \beta & 0 & 0 \\ 0 & 0 & 1 & 0 \\ 0 & 0 & 0 & \beta \end{pmatrix}\]
which also satisfies the conditions of Lemma \ref{MovingTheCubic}.
\end{proof}

\begin{thm}\label{CubicCoreIO} Let $G = \GL_4(k)$ acting naturally on $\P^3(k)$ and let $Y \subset \P^3(k)$ be the twisted cubic. Then the core of the intersection obstruction of $Y$ is
\[ \cZ_Y^\circ = \{g \in G: \rk(g - 1) \leq 1\} \cdot k^\times.\]
\end{thm}
\begin{proof} Let $g \in G$ be such that $\rk(g - 1) \leq 1$, and let $\mu \in k^\times$. The actions of $\mu g$ and $g$ on $\P^3(k)$ coincide, and $g \in \cZ_Y^\circ$ by Corollary \ref{TransVect}. Hence $\mu g \in \cZ_Y^\circ$ also. 

Conversely, suppose $\rk(\mu^{-1} g - 1) \geq 2$ for any $\mu \in k^\times$. Then by Proposition \ref{JNF}, $\mu^{-1}g$ is $G$-conjugate to a matrix $x$ in Jordan normal form which satisfies conditions (a) and (b) of Lemma \ref{MovingTheCubic}. So, by Lemma \ref{MovingTheCubic}, $xhY \cap hY = \emptyset$ for some $h \in G$. This means that $x \notin \cZ_Y^\circ$. Because $\cZ_Y^\circ$ is stable under $G$-conjugation as well as multiplication by scalars inside $G$, we conclude that $g \notin \cZ^\circ_Y$.\end{proof}

\subsection{Avoiding intersections for groups of units of $p$-adic division algebras}\label{RatSect}

Throughout $\S \ref{RatSect}$ we assume that we are given field extensions of $\Qp$ as follows:
\[ \Qp \subseteq L \subset k \subset K\]
where $L/\Qp$ is finite, $k$ is an algebraic closure of $L$, and $K$ is a complete non-archimedean field extension of $k$.

Let $\G$ be an affine group scheme over $L$.  We will first establish some notation involving the functorial conjugation representation \cite[I.2.6]{Jantzen} of the affine group scheme $\G$ on its coordinate ring $A := \cO(\G)$. 

For every $L$-algebra $R$, the group of $R$-points $\G(R)$ of $\G$ acts on itself by conjugation, and hence also on $A \otimes_L R = \cO(\G \otimes_L R)$ by functoriality: for each $L$-algebra $R$, $r\in R$, $f \in A$ and $g,h \in \G(R)$ we have
\begin{equation} \label{Rat1} \left(g \cdot(f \otimes r) \right)(h) = r f(g^{-1}hg) \end{equation}
Let $\rho : A \to A \otimes_L A$ be the corresponding comodule map, and let $\{c_i\}$ be a basis for $A := \cO(\G)$ as an $L$-vector space. Then for every $f \in A$ and every $i \in I$ there is a unique $\rho_i(f) \in A$ such that for all $f \in A$ we have
\[ \rho(f) = \sum_i \rho_i(f) \otimes c_i.\]
With this notation in place, formula \cite[I.2.8(1)]{Jantzen} now implies that 
\begin{equation} \label{Rat2}  g \cdot ( f \otimes 1 ) = \sum_i \rho_i(f) \hsp c_i(g) \end{equation}
for each $L$-algebra $R$, $g \in \G(R)$ and each $f \in A$. Combining $(\ref{Rat1})$ and $(\ref{Rat2})$, we obtain
\begin{equation} \label{Rat3}  (f \otimes r)( g^{-1} h g) = \left(g \cdot (f \otimes r) \right)(h) = r \sum_i \rho_i(f)(h) \hsp c_i(g)  \end{equation}
for each $f \in A$, each $L$-algebra $R$, each $r \in R$ and $g,h \in \G(R)$. 

\begin{defn} We say that $x \in \G(K)$ is \emph{generic} if the $L$-algebra homomorphism $x : \cO(\G) \to K$ is injective.
\end{defn}

Generic points exist only when $\G$ is connected, however this is not a serious restriction to impose in practice. Our next result explains why we are interested in the core of the intersection obstruction.

\begin{thm}\label{RatLemma} Suppose that $\G$ is connected and of finite type over $L$ and that $G := \G(k)$ acts morphically on the projective algebraic variety $X$ over $k$. Let $Y$ be a Zariski closed subset of $X$. Then for every generic $x \in \G(K)$, we have
\[ x \hsp \cZ_Y(K) \hsp x^{-1} \hsp \cap \hsp \G(L) \hsp  \subset \hsp \cZ_Y^\circ.\]
\end{thm}
\begin{proof} 
Because $\G$ is an algebraic affine group scheme over the field $L$ of characteristic zero, it follows from \cite[Theorems 6.6 and 11.4]{Waterhouse} that $A \otimes_L k$ is a domain. The $L$-algebra homomorphism $x : A \to K$ extends uniquely to an $k$-algebra homomorphism $\overline{x} : A \otimes_L k \to K$. Suppose for a contradiction that $\ker \overline{x} \neq 0$. Since $A \otimes_L k$ is integral over $A$, and since it is a domain, it follows that $\ker x = A \cap \ker \overline{x} \neq 0$. This contradicts our assumption on $x$, so in fact $\ker \overline{x} = 0$.

Let $h \in x \hsp \cZ_Y(K) \hsp x^{-1} \hsp \cap \hsp \G(L)$; we must show that $h \in \cZ_Y^\circ$. Let $J$ be the ideal of functions in $\cO(\G \otimes_L k) = A \otimes_L k$ vanishing on $\cZ_Y$ and fix $f \in J$. Then $f(x^{-1}hx) = 0$ because $x^{-1}hx \in \cZ_Y(K)$. Choose a basis $\{\zeta_j\}$ for $k$ over $L$. Then $\{1 \otimes \zeta_j\}$ is a basis for $\cO(\G \otimes_L k) = A \otimes_L k$ as an $A$-module and we can write $f = \sum_j f_j \otimes \zeta_j$ for some unique $f_j \in A$. Applying $(\ref{Rat3})$ we obtain
\[ \begin{array}{lll} 0 & = & f( x^{-1}h x) = \sum_j (f_j \otimes \zeta_j)(x^{-1}h x) =  \\ &=& \sum_{i,j} \zeta_j \hsp \rho_i(f_j)(h) \hsp c_i(x) = \overline{x} \left( \sum_{i,j} \rho_i(f_j)(h) \hsp c_i \otimes \zeta_j \right).\end{array}\]
Because $h \in \G(L)$ and $\rho_i(f_j) \in A$ for each $i,j$, we see that $\rho_i(f_j)(h) \in L$ for all $i,j$. Since $\ker \overline{x} = 0$ and $\{c_i \otimes \zeta_j \}$ is linearly independent over $L$, it follows that
\[ \rho_i(f_j)(h) = 0 \qmb{for all} i,j.\]
Substituting this information back into $(\ref{Rat3})$ we deduce that 
\[f(g^{-1}h g) = 0 \qmb{for all} f \in J, \quad g \in G = \G(k).\]
By Lemma \ref{IntObsZarClosed}, $\cZ_Y$ is a Zariski closed subset of $G$, so $V(J)= \cZ_Y$. Thus $g^{-1} h g \in \cZ_Y$ for all $g \in G$, which means $h \in \cZ_Y^\circ$.
\end{proof}

We can now present our main method of constructing examples of regular orbits of analytic subvarieties of positive dimension.

\begin{cor}\label{GenTransReg} Let $\G$ be an affine algebraic group over $L$, let $G := \G(k)$ act morphically on the projective algebraic $k$-variety $X$ and let $Y$ be a Zariski closed subset of $X$. Let $\bX$ and $\bY$ be the rigid analytifications of the base change of $X$ and $Y$ to $K$, respectively. Suppose further that
\be \item $\G$ is connected and of finite type over $L$, and 
\item $\cZ_Y^\circ \cap \G(L)$ acts trivially on $X$.
\ee
Then for every generic $x \in \G(K)$, the $\G(L)$-orbit of $x\bY$ is regular in $\bX$.
\end{cor}
\begin{proof} Suppose that $h \in \G(L)$ is such that $h x \bY \cap x \bY \neq \emptyset$. Then $x^{-1} hx Y(K) \cap Y(K) \neq \emptyset$, so $x^{-1} h x \in \cZ_Y(K)$ and $h \in x \cZ_Y(K) x^{-1} \cap \G(L)$. Assumption (a) allows us to apply Theorem \ref{RatLemma}, which implies that $h \in \cZ_Y^\circ \cap \G(L)$. Now assumption (b) tells us that $h$ acts trivially on $X$ and hence also on $\bX$, so necessarily we must have $h x \bY = x \bY$.\end{proof}

We will now give an example where condition (b) of Corollary \ref{GenTransReg} holds. 

\begin{lem}\label{RedCharPol} Let $D$ be a division algebra of degree $4$ over $L$. Fix an isomorphism $D \otimes_L k \cong M_4(k)$ and regard $D^\times$ as a subgroup of $\GL_4(k)$. Then
\[ \left(\{ g \in \GL_4(k) : \rk( g - 1 ) \leq 1 \} \cdot k^\times \right) \cap D^\times = L^\times.\]
\end{lem}
\begin{proof} Suppose that $g \in D^\times$ and $\mu \in k^\times$ are such that $\rk( \mu^{-1} g - 1 ) \leq 1$. Let $\chi_g(t)$ be the reduced characteristic polynomial of $g$. Since $\rk( g - \mu) \leq 1$, we see that $\chi_g(t) = (t-\mu)^3(t-\lambda)$ for some $\lambda \in k^\times$. Now $\chi_g(t)$ has coefficients in $L$ by \cite[Lemma IV.2.1]{BerOgg}, and its irreducible factors in $L[t]$ have no repeated roots in $k$ because $\Char(L) = 0$. Therefore at least one of these factors must be $t - \mu$ which forces $\mu \in L$. Looking at the constant term of $\chi_g(t)$, we deduce that $\lambda \in L$ also. Since $\chi_g(g) = 0$ by the Cayley-Hamilton Theorem \cite[Lemma IV.2.3(5)]{BerOgg}, $g \in D$ satisfies the equation $(g-\mu)^3(g-\lambda) = 0$. Because $D$ is a division algebra by assumption, it follows that $g = \mu$ or $g = \lambda$. In either case, $g \in L^\times$ as required.
\end{proof}

\begin{proof}[Proof of Theorem \ref{MainD}] It is easy to see that $\cZ_Y = \cZ_C$. Now Theorem \ref{CubicCoreIO} tells us that $\cZ_Y^\circ = \cZ_C^\circ = \{g \in G: \rk(g - 1) \leq 1\} \cdot k^\times$ , so $\cZ_Y^\circ \cap D^\times = L^\times$ by Lemma \ref{RedCharPol}. Scalars in $D^\times$ act trivially on the flag variety $X$, so condition (b) of Corollary \ref{GenTransReg} is satisfied. Condition (a) is also satisfied because $D^\times$ is known to be the group $L$-points of a connected $L$-form $\G$ of $\GL_4$. Now apply Corollary \ref{GenTransReg}.
\end{proof}

\bibliographystyle{plain}
\bibliography{../references}

\def\cprime{$'$} \def\cprime{$'$} \def\cprime{$'$}
  \def\cftil#1{\ifmmode\setbox7\hbox{$\accent"5E#1$}\else
  \setbox7\hbox{\accent"5E#1}\penalty 10000\relax\fi\raise 1\ht7
  \hbox{\lower1.15ex\hbox to 1\wd7{\hss\accent"7E\hss}}\penalty 10000
  \hskip-1\wd7\penalty 10000\box7}
\begin{thebibliography}{10}

\bibitem{EqDCap}
K.~Ardakov.
\newblock Equivariant {$\mathcal{D}$}-modules on rigid analytic spaces.
\newblock {\em Ast\'erisque}.

\bibitem{ABB}
K.~Ardakov and O.~Ben-Bassat.
\newblock Bounded linear endomorphisms of rigid analytic functions.
\newblock {\em Proc. Lond. Math. Soc. (3)}, 117(5):881--900, 2018.

\bibitem{DCapTwo}
K.~Ardakov and S.~J. Wadsley.
\newblock $\wideparen{{\mathcal{D}}}$-modules on rigid analytic spaces {II}:
  {K}ashiwara's equivalence.
\newblock {\em J. Algebraic Geom.}, 27(4):647--701, 2018.

\bibitem{DCapOne}
K.~Ardakov and S.~J. Wadsley.
\newblock $\wideparen{{\mathcal{D}}}$-modules on rigid analytic spaces {I}.
\newblock {\em J. Reine Angew. Math.}, 747:221--275, 2019.

\bibitem{AMac}
M.~F. Atiyah and I.~G. Macdonald.
\newblock {\em Introduction to commutative algebra}.
\newblock Addison-Wesley Publishing Co., Reading, Mass.-London-Don Mills, Ont.,
  1969.

\bibitem{BerOgg}
Gr\'egory Berhuy and Fr\'ed\'erique Oggier.
\newblock {\em An introduction to central simple algebras and their
  applications to wireless communication}, volume 191 of {\em Mathematical
  Surveys and Monographs}.
\newblock American Mathematical Society, Providence, RI, 2013.
\newblock With a foreword by B. A. Sethuraman.

\bibitem{BGR}
S.~Bosch, U.~G\"untzer, and R.~Remmert.
\newblock {\em Non-Archimedean analysis}.
\newblock Springer-Verlag, Berlin, 1984.

\bibitem{BLR3}
S.~Bosch, W.~L{\"u}tkebohmert, and M.~Raynaud.
\newblock Formal and rigid geometry. {III}. {T}he relative maximum principle.
\newblock {\em Math. Ann.}, 302(1):1--29, 1995.

\bibitem{BourGenTop}
N.~Bourbaki.
\newblock {\em General topology. {C}hapters 1--4}.
\newblock Elements of Mathematics (Berlin). Springer-Verlag, Berlin, 1998.
\newblock Translated from the French, Reprint of the 1989 English translation.

\bibitem{ConradIrreducible}
Brian Conrad.
\newblock Irreducible components of rigid spaces.
\newblock {\em Ann. Inst. Fourier (Grenoble)}, 49(2):473--541, 1999.

\bibitem{DDMS}
J.~D. Dixon, M.~P.~F. du~Sautoy, A.~Mann, and D.~Segal.
\newblock {\em Analytic pro-{$p$} groups}, volume~61 of {\em Cambridge Studies
  in Advanced Mathematics}.
\newblock Cambridge University Press, Cambridge, second edition, 1999.

\bibitem{FvdPut}
Jean Fresnel and Marius van~der Put.
\newblock {\em Rigid analytic geometry and its applications}, volume 218 of
  {\em Progress in Mathematics}.
\newblock Birkh\"auser Boston Inc., Boston, MA, 2004.

\bibitem{Hart}
R.~Hartshorne.
\newblock {\em Algebraic Geometry}, volume~52 of {\em Graduate Texts in
  Mathematics}.
\newblock Springer-Verlag, New York, 1997.

\bibitem{HartLocCoh}
Robin Hartshorne.
\newblock {\em Local cohomology}, volume 1961 of {\em A seminar given by A.
  Grothendieck, Harvard University, Fall}.
\newblock Springer-Verlag, Berlin-New York, 1967.

\bibitem{Hue99}
Johannes Huebschmann.
\newblock Duality for {L}ie-{R}inehart algebras and the modular class.
\newblock {\em J. Reine Angew. Math.}, 510:103--159, 1999.

\bibitem{Jantzen}
J.~C. Jantzen.
\newblock {\em Representations of algebraic groups}, volume 107 of {\em
  Mathematical Surveys and Monographs}.
\newblock American Mathematical Society, Providence, RI, second edition, 2003.

\bibitem{Kisin99}
Mark Kisin.
\newblock Local constancy in {$p$}-adic families of {G}alois representations.
\newblock {\em Math. Z.}, 230(3):569--593, 1999.

\bibitem{KS06}
Mark Kisin and Matthias Strauch.
\newblock Locally analytic cuspidal representations for {${\rm GL}_2$} and
  related groups.
\newblock {\em J. Inst. Math. Jussieu}, 5(3):373--421, 2006.

\bibitem{LVO}
H.~Li and F.~Van~Oystaeyen.
\newblock {\em Zariskian filtrations}.
\newblock Kluwer Academic Publishers, 1996.

\bibitem{Mathers}
R.~Mathers.
\newblock {\em Twisted coadmissible equivariant $\mathcal{D}$-modules on rigid
  analytic spaces}.
\newblock PhD thesis, University of Oxford, UK, 2019.

\bibitem{SchPts93}
P.~Schneider.
\newblock Points of rigid analytic varieties.
\newblock {\em J. Reine Angew. Math.}, 434:127--157, 1993.

\bibitem{SchNFA}
P.~Schneider.
\newblock {\em Nonarchimedean functional analysis}.
\newblock Springer Monographs in Mathematics. Springer-Verlag, Berlin, 2002.

\bibitem{ST6}
P.~Schneider and J.~Teitelbaum.
\newblock {$p$}-adic {F}ourier theory.
\newblock {\em Doc. Math.}, 6:447--481, 2001.

\bibitem{ST3}
P.~Schneider and J.~Teitelbaum.
\newblock Banach space representations and {I}wasawa theory.
\newblock {\em Israel J. Math.}, 127:359--380, 2002.

\bibitem{ST}
P.~Schneider and J.~Teitelbaum.
\newblock Algebras of {$p$}-adic distributions and admissible representations.
\newblock {\em Invent. Math.}, 153(1):145--196, 2003.

\bibitem{SchVdPut}
M.~van~der Put and P.~Schneider.
\newblock Points and topologies in rigid geometry.
\newblock {\em Math. Ann.}, 302(1):81--103, 1995.

\bibitem{Waterhouse}
W.~C. Waterhouse.
\newblock {\em Introduction to affine group schemes}, volume~66 of {\em
  Graduate Texts in Mathematics}.
\newblock Springer-Verlag, New York-Berlin, 1979.

\end{thebibliography}
\end{document}